\documentclass[a4paper,10pt,twoside]{article}
\usepackage{fancyhdr}
\usepackage{amsbsy,amsfonts,amssymb,amsmath}
\usepackage{amsthm}
\usepackage{enumitem}
\numberwithin{equation}{section}
\usepackage{stmaryrd} 
\usepackage[active]{srcltx}
\usepackage{graphicx}
\usepackage{geometry}
\usepackage{bera}
\usepackage{dsfont}
\usepackage{color}
\definecolor{bleuf}{rgb}{0.1,0.1,0.6}
\usepackage{multicol}
\usepackage[pdftex,pagebackref=false]{hyperref}
 \hypersetup{
        colorlinks = true,           
        breaklinks = true,          
        urlcolor = bleuf,           
        linkcolor =bleuf,                  
        citecolor = bleuf,                     
        linktocpage = true,
        pdftitle={A multiplicative coalescent with asynchronous multiple mergers},   
        pdfauthor={S. Lemaire},                  
     }
\hypersetup{pdfborder=0 0 0}
\hypersetup{pdfstartview=FitH} 

\newcommand{\shortitle}{A multiplicative coalescent} 

\DeclareMathOperator{\NN}{\mathbb{N}}
\DeclareMathOperator{\ZZ}{\mathbb{Z}}

\DeclareMathOperator{\RR}{\mathbb{R}}
\DeclareMathOperator{\Pd}{\mathbb{P}}
\DeclareMathOperator{\Ed}{\mathbb{E}}
\DeclareMathOperator{\un}{\mathds{1}}

\newcommand{\enu}[1][n]{\llbracket {#1} \rrbracket}

\newcommand\mW[1][\enu]{\mathcal{W}(#1)}
\newcommand\mWs[1][\enu]{\mathcal{W}^{*}(#1)}
\newcommand\mP{\mathcal{P}} 
\newcommand\mPs{\mathcal{P}^{*}} 
\newcommand\fP[1][n]{\mathfrak{P}(#1)} 
\newcommand\Pin[1][n]{\Pi_{#1}} 
\newcommand\Gp{G_{p}} 
\newcommand\pb{\tilde{p}} 
\newcommand\mom[1][1]{m_{p,#1}}
\newcommand\moms{m^{*}_{p,1}}
\newcommand\Ne{\mathcal{N}} 
\newcommand\Tn{\bar{T}_{n}}       
\newcommand\Tp[1][1]{T^{(#1)}_{p}} 
\newcommand\ER{\mathcal{H}}
\newcommand\CPois{\text{CPois}}
\newcommand\BGW{\text{BGW}}
\newcommand\swait{\tau^{(singl)}_n}
\newcommand\cwait{\tau^{(coal)}_n}
\newcommand\mK[1]{\mathcal{K}_{#1}} 
\newcommand\ld{\ell}
\newcommand\DL{\mathcal{DL}}

\newcommand\mup{\dot{\mu}}

\newcommand\Pois{\text{Poisson}}
\newcommand\ssi{ if and only if }

\DeclareMathOperator{\et}{ and }

\DeclareMathOperator{\TV}{d_{\text{TV}}}

\DeclareMathOperator{\var}{Var}

\renewcommand{\epsilon}{\varepsilon}
\newcommand\iid{i.i.d.\ }
\newtheoremstyle{bodyit}%
 {3pt}
 {3pt}
 {\itshape}
 {}
 {\bfseries\sffamily}
 {.}
 { }
 {}
 \newtheoremstyle{bodyrm}%
 {3pt}
 {3pt}
 {}
 {}
 {\bfseries\sffamily}
 {.}
 { }
 {}
\theoremstyle{bodyit}
\newtheorem{thm}{Theorem}[section]
\newtheorem{lem}[thm]{Lemma}
\newtheorem{prop}[thm]{Proposition}
\newtheorem{corol}[thm]{Corollary}

\newtheorem*{bibli}{Theorem}
\newtheorem*{biblil}{Lemma}
\theoremstyle{bodyrm}
\newtheorem{ex}[thm]{Example}

\newtheorem{rem}[thm]{Remark}

\newcommand\comm[1]{\noindent\emph{#1}}

\title{A multiplicative coalescent with asynchronous multiple mergers}

\author{S. Lemaire\thanks{Laboratoire de Math\'ematiques d'Orsay, Univ.~Paris-Sud, CNRS, 
Universit\'e Paris-Saclay, 91405 Orsay, France; email: \url{sophie.lemaire@u-psud.fr}}}
\begin{document}
\maketitle
\begin{abstract}
We define a Markov process on the partitions of $[n]=\{1,\ldots,n\}$ by drawing a sample 
in $[n]$ at each time of a Poisson process,  by merging blocks that contain one of 
these points and by leaving all other blocks unchanged. This coalescent process 
appears in the study of the connected components of  random graph processes in 
which connected subgraphs are added over time  with probabilities that depend 
only on their size. \\  
First, we  determine  the asymptotic distribution of the coalescent time. 
Then, we define a Bienaym\'e-Galton-Watson (BGW) process such that its total 
population size  dominates the block size of an element. We compute a bound 
for  the distance between the total population size  distribution and the block 
size distribution at a time proportional to $n$. As a first application of 
this result, we establish the coagulation equations associated with this 
coalescent process. As a second application, we estimate the size of the largest 
block in the subcritical and supercritical regimes as well as in the critical window.  
\end{abstract}
{\small 
}
\section*{Introduction}
 The  paper is devoted to studying a family of multiplicative coalescent processes  
 on a finite set $S$ defined by a simple algorithm. To present this algorithm, 
 let us fix a probability distribution $p$ on $\NN^*$. We construct a coalescent process denoted  
 $(\Pin[S,p](t))_{t\geq 0}$ by the following algorithm: 
\begin{enumerate}
 \item $\Pin[S,p](0)$ is the partition defined by the  singletons of $S$; 
 \item At each event $\tau$ of a Poisson process $(Z_t)_t$ with intensity one, we choose 
 a positive integer $k$ according to  $p$  and we draw $k$ elements 
 $x_1,\ldots,x_k$ in $S$ by a simple random sampling with replacement.   
 The partition at time $\tau$ is defined by merging blocks of $\Pin[S,p](\tau^{-})$   
 that contain $x_1,\ldots,x_k$ into one block  and by leaving all other blocks 
 unchanged. 
\end{enumerate}
By construction, only one merger can occur at a given time but  it may involve more than two  blocks.  
The probability  that blocks coalesce depends only on the product of their sizes. Such a coalescent process  
naturally appears when considering a random hypergraph process on the set 
of vertices $S$ of size $|S|=n$.  \\

A random hypergraph process can be defined as a Markov process 
$(\mathcal{G}(t))_{t\geq 0}$ 
whose states are hypergraphs on $S$: it starts with the empty graph and hyperedges (i.e. subsets of $S$) 
are added over time according to a given rule. 
There are several possible definitions of hypergraph components. One way is to identify a hyperedge 
$A$ to a  connected subgraph  and then a hypergraph to 
a multigraph; the component of a vertex  can be defined as usual in a graph.
The process defined by the connected components of $\mathcal{G}(t)$ 
for $t\geq 0$ is a coalescent process. Here are two examples of classical random hypergraph processes. 
\begin{itemize}
\item \emph{Erd\"os-R\'enyi random graph}.  If  a pair, chosen with a uniform distribution on 
$S^2$, is added at each time of a Poisson process with intensity one, we obtain a variant of 
the Erd\"os-R\'enyi random graph process denoted  $(\ER(n,t))_{t\geq 0}$:  
the probability that ${e=(i,j)\in S^2}$ is an edge of 
$\ER(n,t)$ is equal to $1-\exp(-\frac{2t}{n^2})$ and the coalescent process 
associated with  $(\ER(n,t))_{t\geq 0}$ has the same distribution as 
$(\Pin[S,p](t))_{t\geq 0}$ where $p$ is the Dirac measure on $2$ and $|S|=n$. 
\item \emph{Uniform random graph process.}
For a fixed $d>2$, if a subset of size $d$ chosen with a uniform distribution on 
$S^d$, is added at each time of a Poisson process with intensity one, it defines 
a  random hypergraph process  whose components have similar properties as a  
$d$-uniform random graph process. The partition defined by the connected 
components of this random hypergraph process has the same law as 
$(\Pin[S,p](t))_{t\geq 0}$ where $p$ is the Dirac measure on $d$.
\end{itemize}
More generally, if  each new hyperedge $A$ added is chosen with a distribution 
$\nu$ that depends only on the number of  vertices in $|A|$, then the associated 
coalescent process  has the same distribution as $(\Pi_{S,p}(t))_t$, where 
$p(|A|)=\nu(A)$ for every $A\subset S$. \\ 

Let us  present the properties of $\mathcal{H}(n,\frac{nt}{2})$  which are 
related to our study. For each property, we shall also review works done on 
random hypergraphs to introduce our contribution. Precise statements of our 
results will be described in Section \ref{sect:results}. 
\begin{enumerate}
\item \emph{Connectivity threshold}\\
 Erd\"os and R\'enyi in \cite{ErdosRenyi59} and independently Gilbert in 
 \cite{Gilbert59}  have studied the probability that the random graph models 
 they introduced  are connected. Erd\"os and R\'enyi results can be formulated 
 for the random graph process $(\mathcal{H}(n,\frac{nt}{2}))_t$ as follows:
\begin{bibli}
For every $c\in\RR$ and every $k\in\NN$, the probability that  
$\mathcal{H}(n,\frac{n}{2}(\log(n)+c))$ contains a connected component of size $n-k$ and $k$ isolated points converges to 
$\displaystyle{\exp(-e^{-c})\frac{e^{-ck}}{k!}}$ as $n\rightarrow$ tends to $+\infty$.
\end{bibli}
This shows that $\frac{n}{2}\log(n)$ is a sharp threshold function for the connectivity property.  

Poole in his thesis \cite{PooleThesis}, has extended this result for uniform random hypergraphs:  
the threshold for connectivity of  a $d$-uniform random hypergraph  is $\frac{n}{d}\log(n)$ for every $d\geq 2$. 
Kordecki in \cite{Kordecki85} has given a general formula for  the probability that a random 
hypergraph is connected for non-uniform random hypergraph with bounded hyperedges. 

Poisson point processes of Markov loops on a finite graph give examples of 
random graph processes for which connected subgraphs (close walks here) are added over time 
(see \cite{stfl} and \cite{lawlerlimic} for a survey of their properties).  Some general
properties of  the coalescent process induced by  them have been presented by Le Jan
 and the author in \cite{LeJanLemaire13}.  In particular, it has been shown  that 
 when loops are constructed by a random walk killed at 
a constant rate on the complete graph $K_n$,  the coalescent process associated with the 
Poissonian ensembles of loops can be constructed as  $\Pin[S,p]$, where $p$ is a 
logarithmic distribution with a parameter depending on the killing rate;  
the connectivity threshold function have been established. 

By a similar study, we  extend the statement of the previous theorem  for a large 
class of distributions for $p$ that contains probability distributions having 
a finite moment of order two showing in particular that the connectivity threshold
 for a random hypergraph  whose components are described by $\Pi_{S,p}$ is  
 $\displaystyle{\frac{|S|\log(|S|)}{\sum_{k \geq 2}kp(k)}}$ (Theorem \ref{coaltime}).

\item \emph{Phase transition.}\\
The largest block size of $\mathcal{H}(n,\frac{nt}{2})$ undergoes a phase transition. It was first 
proved by Erd\"os and R\'enyi in \cite{ErdosRenyi60}. The statement
we present is taken from \cite{book1vanderHofstad}, where the proof is based on 
the use of Bienaym\'e-Galton-Watson (BGW)  processes.  
\begin{bibli}[\cite{book1vanderHofstad}]
Let $c^{(n)}_t(x)$ denote the component size of a vertex $x$  of  $\mathcal{H}(n,\frac{nt}{2})$ and let 
$c^{(n)}_{1,t}\geq c^{(n)}_{2,t}$ denote the two largest component sizes. 
\begin{enumerate}
\item Assume that $t<1$. 
\begin{itemize}
\item For every vertex $x$, $c^{(n)}_t(x)$ converges in distribution to the total population 
size of a BGW process  with one
progenitor and \Pois($t$) offspring distribution. 
\item  Let $I_{t}$ be the value at 1 of the Cram\'er
function of the \Pois$(t)$-distribution:
$I_{t}=t-1-\log(t)$.  \\
The sequence $\displaystyle{\left(\frac{c^{(n)}_{1,t}}{\log(n)}\right)_n}$ converges in probability to $1/I_{t}$.  
\end{itemize}
\item Assume that $t>1$ and denote by $q_t$ the extinction probability of a BGW process with one
progenitor and \Pois($t$) offspring distribution.\\
For every $a\in]1/2,1[$, there exist $b>0$ and $c>0$ such that 
$$\Pd(|c^{(n)}_{1,t}-(1-q_t)n|\geq n^a)+\Pd(c^{(n)}_{2,t}\geq c\log(n))=O(n^{-b}).$$
\item Assume that $t=1+\theta n^{-1/3}$ for some $\theta\in\RR$. 
There exists a constant $b(\theta)>0$ such that for every $w>1$ and every $n\in \NN^*$, 
$$
\Pd(c^{(n)}_{1,1+\theta n^{-1/3}}  > w n^{2/3})\leq \frac{w}{b(\theta)}\;
\text{and}\;
\Pd(c^{(n)}_{1,1+\theta n^{-1/3}}  < w^{-1}n^{2/3})\leq \frac{w}{b(\theta)}.
$$ 
\end{enumerate}
\end{bibli}
In \cite{SchmidtPruzan},  Schmidt-Pruzan and Shamir studied  the size of the largest component for non-uniform random 
hypergraphs: in their model,  the size of hyperedges is bounded and the 
probability that the hypergraph has a fixed hyperedge depends only on the size of the hyperedge. 
They established similar statements for the largest component  when the average degree 
of a vertex in the hypergraph is less than 1, equal to 1 and greater than 1.   
More precise results on the phase transition have been established later in 
the case of uniform random hypergraphs (see \cite{Karonski02}). 
Bollob\'as, Janson and Riordan in \cite{Bollobas11} have studied 
the size of the  connected components for a  general model of 
random hypergraph: in their model a type is associated with each vertex and the 
probability to add a hyperedge $A$ depends on the 
types of the elements in $A$. From their study we can deduce that
 the size of the largest block of $\Pin[S,p](|S|t)$  is  
 $o_p(|S|)$   if $t\sum_{k\geq 2} k(k-1)p(k)<1$ and  
 $\rho |S|+o_p(|S|)$ if $t\sum_{k\geq 2} k(k-1)p(k)>1$, where $1-\rho$ is 
the smallest positive solution of the following equation: 
$$x=\exp\left(-t\sum_{k\geq 2}kp(k)(1-x^{k-1})\right).$$ 
($\rho$ can be seen as the survival probability of a BGW process with a compound
 Poisson offspring distribution). 
Janson in \cite{Janson08} proved a conjecture proposed by Durrett in \cite{DurrettBookRG} saying that for a random graph with a power law degree distribution with exponent $\gamma>3$, the largest component in the subcritical phase is of order $n^{\frac{1}{\gamma-1}}$. This result suggests that the size of the largest block of $\Pin[S,p](|S|t)$ in the subcritical phase would be also order $|S|^{\alpha}$ for some $0<\alpha$ if $p$ does not have all its power moments finite. 

Under the assumption that $p$ has a finite third moment, we give a 
bound for the distance between the cumulative distributions of the block size 
of an element and  of the total population size of a BGW process with compound 
Poisson offspring distribution (Theorem \ref{thmblocksize}). 
We deduce from this the asymptotic distribution of two block size as 
$|S|$ tends to $+\infty$ (Corollary \ref{coroltwoblocks}). We  also 
study the largest block size in  three different regimes 
(Theorems \ref{thmtransition} and \ref{subcritlowerbound}): in the subcritical phase, we show  that the size of the largest block is $\displaystyle{o_p(|S|^{\frac{1+\epsilon}{u-1}})}$ for every $\epsilon>0$, if $p$ has a finite moment of order $u \geq 3$ and is $O_p(\log(|S|))$, if  $p$ is a light-tailed distribution. When $p$ is a regularly varying distribution with index smaller than $-3$, we also establish that the size of the largest block grows faster than a positive power of $|S|$ as $|S|$ tends to $+\infty$.   
In the critical window, 
we show that the size of the largest block is $O_p(n^{2/3})$. 
Although the supercritical regime is studied in \cite{Bollobas11},
 to complete the analysis of the largest block we present a simple proof of 
 the property  stated in (b) for our model.   

\item \emph{Hydrodynamic behavior}\\
 Let us now consider the average number of components of  size $x$ in $\ER(n,\frac{nt}{2})$.  
\begin{itemize}
\item For any  $t>0$ and $x\in \NN^*$, the average number of components of  size $x$ 
in $\ER(n,\frac{nt}{2})$ converges in $L^2$ to 
$$v(x,t)= \frac{(tx)^{x-1}e^{-tx}}{x.x!}.$$ 
The value $xv(x,t)$ is equal to the probability that $x$ is the total population size of  
a BGW process with
one progenitor and \Pois($t$) offspring distribution\footnote{For $t\leq 1$,  $\{xv(x,t), x\in\NN^*\}$ 
is a probability distribution called Borel-Tanner distribution with parameter~$t$.}.
\item $\{v(x,\cdot),\ x\in \NN^*\}$ is the solution on $\RR_+$ of the Flory's coagulation equations with 
multiplicative kernel:
\begin{multline}
\label{Floryeq}
\frac{d}{dt}v(x,t)=\frac{1}{2}\sum_{y=1}^{x-1}y(x-y)v(y,t)v(x-y,t)\\
-\sum_{y=1}^{+\infty}xyv(x,t)v(y,t)- xv(x,t)\sum_{y=1}^{+\infty}y\big(v(y,0)-v(y,t)\big)
\end{multline}
Up to time 1, this solution coincides with the solution of the 
Smoluchowski's coagulation equations with multiplicative kernel starting from  the monodisperse state:  
\begin{equation}
\label{Smoluchowskieq}
\frac{d}{dt}v(x,t)=\frac{1}{2}\sum_{y=1}^{x-1}y(x-y)v(y,t)v(x-y,t)-xv(x,t)\sum_{y=1}^{+\infty}yv(y,t).
\end{equation}
Equations \eqref{Smoluchowskieq} introduced by Smoluchowski in \cite{Smolu16} are used for example 
to describe aggregations of polymers in an homogeneous medium where diffusion effects are ignored.  
The first term in the right-hand side  describes the formation of a  particle of mass $x$ by
aggregation of two particles, the second sum describes the ways  a particle of
mass $x$ can  be aggregated with  another particle. If the total mass of particles decreases after 
a finite time, the system is said to  exhibit a `phase transition' called `gelation': the loss of 
mass is interpreted as the formation of infinite mass particles called  gel. Smoluchowski's equations 
do not take into account interactions between gel and finite mass particles.  
Equations \eqref{Floryeq} introduced by Flory in \cite{Flory} are a modified version of the 
Smoluchowski's equations with   an extra term describing the loss of a particle of mass $x$ by 
`absorption' in the gel. 
Let $T_{gel}$ denote the largest time such that the Smoluchowski's coagulation
equations with monodisperse initial condition have a solution which has the
mass-conserving property\footnote{Different definitions of the `gelation time' $T_{gel}$ are used 
in the literature:  the gelation time is sometimes defined as the smallest time when the second 
moment diverges (see \cite{Aldousreview})}. Then, $T_{gel}=1$ and $T_{gel}$ coincides with the 
smallest time when  the second moment $\sum_{x=1}^{+\infty}x^2v(x,t)$ diverges (see \cite{McLeod}).
Let us note that the random graph process $(\mathcal{H}(n,\frac{nt}{2}))_{t\geq 0}$ is equivalent to the 
microscopic model  introduced by Marcus \cite{Marcus} and further studied by Lushnikov \cite{Lushnikov} 
(see \cite{BuffetPule} for a first study of the relationship between  these two models and 
\cite{Aldousreview} for a review, \cite{Norris99}, \cite{Norris00} and \cite{FournierGiet} 
for convergence results of  Marcus-Lushnikov's model to \eqref{Floryeq}). 
\end{itemize}
Recently, Riordan and Warnke in \cite{RiordanWarnke16} gave  sufficient 
conditions under which the average number 
of blocks of size $x$ converges for a class of random graph processes 
in which a bounded number of edges can be added at each step according to 
a fixed rule.  This class includes uniform random hypergraph processes.  
As far as we know  such a result has not been established for more general 
random hypergraph processes. 

Under the assumption that $p$ has a finite third moment, we show that  the average number 
of blocks of size $x$ in the coalescent process $\Pi_{S,p}$ converges in $L^2$ to the solution of coagulation equations  in which more than two
particles can collide at the same time at a rate that depends  on the product of their masses (Theorem \ref{thm:hydrodyn}). 
\end{enumerate}

\comm{Remark.} Let us note that  Darling, Levin and Norris have  introduced in \cite{DarlingNorris04} a random hypergraph model called 
\emph{Poisson$(\rho)$ random hypergraph process and denoted $(\Lambda_t)_{t\geq 0}$}.  The process $(\Lambda_t)_{t\geq 0}$ is defined as follows: 
\begin{itemize}
\item  Start with the set of vertices $S$; 
\item At each event $\tau$ of a Poisson process with intensity 
$1$,  choose a positive integer $k\leq |S|$ with probability $\rho(k)$ and a subset $A$ 
uniformly at random from the  subsets of $S$ of size $k$. Then, add $A$ in the hyperedges subset of $\Lambda_{\tau^{-}}$. 
\end{itemize}
One can choose $\rho$ so that the coalescent process defined by the connected components of  $(\Lambda_t)_{t\geq 0}$ is described by $\Pin[S,p]$. 
Indeed,  $p(k)$ in the definition of $\Pin[S,p]$ describes the probability to add a subset defined by  $k$ elements of $S$ chosen by a simple 
random sampling with replacement.  Hence, if we set  
\[\rho(j)=\frac{\binom{|S|}{j}}{|S|}\sum_{k=j}^{+\infty}\frac{p(k)}{|S|^k}
\sum_{\substack{(k_1,\ldots, k_{j})\in(\NN^*)^{j},\\ k_1+\cdots+k_{j} =k} }
  \binom{k}{k_1,\ldots,k_{|A|}}\;\text{for every}\;1\leq j\leq |S|,
  \]
then  $\Pin[S,p]$  is the coalescent process defined by the connected components of  $(\Lambda_t)_{t\geq 0}$. 
 In \cite{DarlingNorris04}, the object of study  is not the connected components of $(\Lambda_t)_{t\geq 0}$ as 
 we have defined them in our study but  identifiable vertices. \\

\comm{Organization of the paper. }
Section \ref{sect:model} is devoted to a presentation of  general properties of the coalescent process we study. 
The  main results are  stated in Section \ref{sect:results}.  
In Section \ref{sec:coaltime}, we first study the distribution of the  number 
of singletons in the coalescent process and the first time $\swait$ the 
coalescent $\Pi_{\enu,p}$ does not have singleton.
 Next we show that the  distribution of  the coalescent time  $\cwait$ coincides with  
the asymptotic distribution of  $\swait$ as $n$  tends to $+\infty$ which proves Theorem \ref{waittime}. 
In Section \ref{sect:exploration}, we describe the exploration process used to compute 
the block size of an element and to  construct the associated BGW process.  
The asymptotic distribution of the block size of an element  is studied in 
Section \ref{sect:proofth}:  
proofs of Theorem \ref{thmblocksize} and its corollaries are presented. 
Section \ref{sect:coageq} is devoted to the proof of Theorem  \ref{thm:hydrodyn} 
that describes the hydrodynamic behaviour of the coalescent process.  
In Section \ref{sect:transition}, we prove Theorems \ref{thmtransition} and \ref{subcritlowerbound} which 
present some properties of the largest block size in the subcritical, critical 
and supercritical regimes.  
 Appendix \ref{annex:propGW} contains some properties of BGW processes with a compound Poisson offspring distribution. 
%
\section{Description of the model \label{sect:model} and general properties}
To study the properties of $(\Pin[S,p](t))_{t>0}$, it is useful to construct 
it by the mean of a Poisson point process instead of the algorithm presented 
in the introduction. 
Let us first introduce some notations associated with a finite set $S$:
\begin{itemize}
\item The number of elements of $S$ is denoted by $|S|$. 
 \item   $\mW[S]:=\cup_{k\in\NN^*}S^k$ denotes the set of nonempty tuples over 
 $S$ and  $\fP[S]$ is the set of nonempty subsets of $S$. 
 \item A tuple is called \emph{nontrivial} if it contains at least two different elements of $S$. 
 We write $\mWs[S]$ for the set on nontrivial tuples over $S$. 
 \item The length of a tuple $w\in \mW[S]$ is denoted by $\ell(w)$. 
\end{itemize}

\subsection{The Poisson sample sets}
Let $p=\sum_{i=1}^{+\infty}p(i)\delta_{i}$ be a probability measure on $\NN^*$ such that $p(1)<1$.  
We denote by $G_p$ its probability generating function: 
$G_p(s)=\sum_{k=1}^{+\infty}p(k)s^k$ for $|s|\leq 1$.  
The following algorithm  
  `\emph{Choose an integer $K$ with  probability distribution $p$ and sample with replacement $K$ elements of $S$}'
defines a probability measure on $\mW[S]$ denoted by $\mu_{S,p}$:
  \[\mu_{S,p}(\{x\})=\frac{p(\ell(x))}{|S|^{\ell(x)}} 
  \text{ for every }x\in \mW[S]. \]
  We consider a Poisson point process $\mP_{S,p}$ with intensity $\text{Leb}\times \mu_{S,p}$ on 
$\RR_+\times \mW[S]$ and  for $t\geq 0$, we define  $\mP_{S,p}(t)$ as 
the projection of the set $\mP_{S,p} \cap([0,t]\times \mW[S])$ on $\mW[S]$: $\mP_{S,p}(t)$ corresponds
to the set of samples chosen before time $t$.
  \begin{rem}\label{restrictionmeasure}
  Let $S'$ be a subset of $S$. 
  \begin{enumerate}
   \item  The conditional probability $\mu_{S,p}(\cdot \mid \mW[S'])$ seen
   as a probability on $\mW[S']$ is  equal to  $\mu_{S',p_{S|S'}}$ where $p_{S|S'}$ is the probability 
   on $\NN^*$ defined by: $$p_{S|S'}(k)=\Big(\frac{|S'|}{|S|}\Big)^k\frac{p(k)}{\Gp\big(\frac{|S'|}{|S|}\big)}
   \text{ for every }k\in \NN^*.$$ 
   In particular, the restriction of $\mP_{S,p}(t)$ to tuples in $\mW[S']$ before time $t$ has 
   the same distribution as $\mP_{S',p_{S|S'}}\Big(\Gp\big(\frac{|S'|}{|S|}\big)t\Big)$. \\
 \item   Let us also note that the pushforward measure of $\mu_{S,p}$ by the projection $\pi_{S,S'}$ from   
   $\mW[S]$ to $\mW[S']$ is equal to $\mu_{S',p^{\{S'\}}}$ where $p^{\{S'\}}$ is the probability on $\NN^*$ defined by:
  $$p^{\{S'\}}(k)=\Big(\frac{|S'|}{|S|}\Big)^k\sum_{\ell=0}^{+\infty}p(k+\ell)\binom{k+\ell}{k}
  \Big(1-\frac{|S'|}{|S|}\Big)^\ell \text{ for every }k\in \NN^*.$$
\end{enumerate}
\end{rem}
\begin{rem}\
  The order of elements in a tuple $w$ will play no role in the definition 
  of the coalescent process, the main object is the subset of $S$ formed by 
  the elements of $w$. 
  The pushforward measure of $\mu_{S,p}$ on $\fP[S]$ is the probability measure 
  $\bar{\mu}_{S,p}$ defined by 
  $$\bar{\mu}_{S,p}(\{A\})=\sum_{k=|A|}^{+\infty}\frac{p(k)}{|S|^k}\sum_{\substack{(k_1,\ldots, k_{|A|})\in(\NN^*)^{|A|},\\ k_1+\cdots+k_{|A|} =k} }
  \binom{k}{k_1,\ldots,k_{|A|}}\;\text{for every}\; A \in \fP[S]. $$
  We choose to work with the Poisson point process on $\RR_+\times \mW[S]$ instead of the associated 
  Poisson point process on $\RR_+\times \fP[S]$ because some  proofs are simpler to write. 

  \end{rem}\medskip
  
  To  shorten the description we  use sometimes a tuple $w\in \mW[S]$ as the subset formed 
  by its elements and write $x\in w$ for $x\in S$ to mean that 
  $x$ is an element of the tuple $w$ and $w\cap A\neq \emptyset$ for $A\subset S$ to mean 
  that $w$ contains some elements of the subset $A$. 
\subsection{The coalescent process}
If $A$ is a subset of $S$, we define the $\mP_{S,p}(t)$-neighborhood of $A$ as follows:
$$
 \mathcal{V}_{A}(t)=A\cup\{i\in S,\; \exists w\in \mP_{S,p}(t) \text{ such that } i\in w
 \text{ and  } w\cap A\neq\emptyset\}.
$$
We can iterate this definition by setting: $\mathcal{V}^k_{A}(t)=\mathcal{V}^{k-1}_{\mathcal{V}_{A}(t)}(t)$ for $k\in\NN^*$.  \\
Given any $(i,j)\in S^2$, set $i\underset{t}{\sim} j \text{\ssi} 
\exists k\in \NN^*\; \text{ such that } j\in\mathcal{V}^{k}_{\{i\}}(t)$.  
This defines  an equivalence relation on $S$. We denote   
by $\Pi_{S,p}(t)$ the partition of $S$ defined by $\underset{t}{\sim}$. 
In other words, two elements $i$ and $j$ are in a same block of the partition $\Pi_{S,p}(t)$ 
\ssi there exists a finite number of tuples $w_1,w_2,\ldots, w_k\in \mP_{S,p}(t)$ 
such that $i\in w_1$, $j\in w_k$ and $w_i\cap w_{i+1}\neq \emptyset$ 
for every $1\leq i\leq k-1$. \\

The evolution in $t$ of $\Pin[S,p](t)$ defines a coalescent process on $S$. Let 
us note that this coalescent process  depends only on the restriction 
of $\mP_{S,p}$ to $\RR_+\times \mWs[S]$. 
\subsubsection{Transition rates and semigroup of the coalescent process}
 Let us describe
the transition rates and the semigroup of  $\Pin[S,p]$. 
\begin{prop}\label{transitionKn}
Let $\pi$  be a partition of $S$ into non-empty blocks $\{B_i,\; i\in I\}$. 
\begin{enumerate}
\item[(i)] From state $\pi$, 
the only possible transitions of $(\Pi_{S,p}(t))_ {t\geq 0}$ are 
to partitions $\pi^{\oplus J}$ obtained  by merging blocks, indexed 
by some subset $J$ of $I$ of size greater than or equal to two, 
to form one block $B_J=\cup_{j\in J}B_j$  and leaving 
all other blocks unchanged. 
Its transition rate from $\pi$ to  $\pi^{\oplus J}$ is equal to: 
\begin{align}
\tau_{\pi,\pi^{\oplus J}}&=\sum_{k\geq |J|}\frac{p(k)}{|S|^k}
\sum_{\substack{(k_1,\ldots,k_{|J|})\in (\NN^*)^{|J|},
 \\ k_1+\cdots+k_{|J|}=k}}\binom{k}{k_1,\ldots,k_{|J|}} \prod_{j\in J}|B_{j}|^{k_j}\label{eqtransitionKn}\\
 &=\sum_{H\subset J}(-1)^{|H|}\Gp\Big(\frac{|B_{J\setminus H}|}{|S|}\Big). 
\end{align}
\item[(ii)] 
 For every partition $\pi_0$ of $S$, 
 \begin{multline}\label{semigroup}
\Pd(\Pi_{S,p}(t) \text{ is finer than } \pi\mid \Pi_{S,p}(0)=\pi_0)\\=
\exp\left(-t\left(1-\sum_{i\in I}\Gp\left(\frac{|B_{i}|}{|S|}\right)\right)\right)\un_{
\{\pi_0\;\text{is finer than}\;\pi\}}
\end{multline}
\end{enumerate}
\end{prop}
\begin{proof}
 \begin{enumerate}
  \item The transition rate $\tau_{\pi,\pi^{\oplus J}}$ is equal to the $\mu_{S,p}$-measure of  
  tuples $w\in \mW[B_J]$ that contain elements of each block $B_j$ 
  for $j\in J$. The first formula is obtained by enumerating such tuples ordered by their length.  
  The inclusion-exclusion formula yields the second formula since
  $$\tau_{\pi,\pi^{\oplus J}}=
  \mu_{S,p}(\mW[B_J])-\mu_{S,p}\left(\underset{i\in J}{\cup}\mW[B_{J\setminus \{i\}}]\right).$$
  \item $\Pi_{S,p}(t)$  is finer than  $\pi$ \ssi every tuple chosen before time $t$ is 
  included in a block of the partition $\pi$. Therefore, if $\pi_0$  is finer than  $\pi$, 
  $$\Pd(\Pi_{S,p}(t) \text{ is finer than } \pi\mid \Pi_{S,p}(0)=\pi_0)=
  \exp\Big(-t\big(\mu_{S,p}(\mW[S])-\sum_{i\in I}\mu_{S,p}(\mW[B_{i}])\big)\Big).$$
\end{enumerate}
\end{proof}
\begin{ex}
\label{exER}
 If $p$ is the Dirac measure  $\delta_{\{2\}}$, then the only possible transitions of $(\Pi_{S,p}(t))_t$ are 
 from a partition $\pi=(B_i,\; i\in I)$ to   partitions  obtained by merging two blocks $B_i$ and $B_j$; the transition 
 rate for such a transition is: $\displaystyle{\tau_{\pi,\pi^{\oplus \{i,j\}}}=2\frac{|B_i||B_j|}{|S|^2}}$. 
 Therefore,  for a partition   $\pi$ of $S$ into non-empty blocks $\{B_i,\; i\in I\}$ coarser  
 than a partition $\pi_0$ of $S$, 
 $$\Pd(\Pi_{S,\delta_{\{2\}}}(t) \text{ is finer than }  \pi\mid \Pi_{S,\delta_{\{2\}}}(0)=\pi_0)=
 \exp\left(-2t\sum_{i,j\in I \text{ s.t. } i<j}\frac{|B_i||B_j|}{|S|^2}\right).$$
\end{ex}
\begin{ex}
\label{exloop}
 Let  $p$ be the logarithmic distribution with parameter $a\in]0,1[$: 
 $p(k)=c\frac{a^k}{k}$ with $\frac{1}{c}=-\log(1-a)$ for every $k\in\NN^*$. 
 For a partition   $\pi$ of $S$ into non-empty blocks $\{B_i,\; i\in I\}$ coarser  
 than a partition $\pi_0$ of $S$, 
 \begin{equation}\label{logcoaldistr}\Pd(\Pi_{S,p}(t) \text{ is finer than } \pi\mid \Pi_{S,p}(0)=\pi_0)=
  (1-a)^{ct}\prod_{i\in I}\left(1-a\frac{|B_i|}{|S|}\right)^{-ct}.
  \end{equation}
  
  This shows that $\Pi_{S,p}(t)$ has the same distribution as a coalescent 
  process describing the evolution of the clusters of Poissonian loop sets on a 
  complete graph  defined in \cite{LeJanLemaire13}. 
  
  Let us briefly present how these Poissonian loop sets are defined.  
  Let  $S$ stand for the vertices of a finite graph $\mathcal{G}$  with $n$ 
  vertices and  let consider a simple random walk on
$\mathcal{G}$ killed at each step with probability $1-a$. In other words, $\mathcal{G}$  is endowed with  unit
conductances and a uniform killing measure with intensity $\kappa_n=n(\frac{1}{a}-1)$.
A discrete based loop
$\ell$ of length $k\in\NN^*$ on $\mathcal{G}$ is defined as an element of
$\mathcal{G}^k$.  To each element $\ell=(x_1,\ldots,x_k)$ of $\mathcal{G}^k$ of length $k\geq
2$ is assigned the weight $\mup(\ell)=\frac{1}{k}P_{x_1,x_2}\ldots P_{x_{k},x_{1}}$ where $P$ denotes 
the transition matrix of the random walk. When $\mathcal{G}$ is the complete graph $K_n$ then 
$\mup(\ell)=\frac{a^k}{k n^k}$ for every $\ell\in K_{n}^{k}$.  
 A based loop  $\ell=(x_1,\ldots,x_k)$  is said to be equivalent to the based loop
 $(x_i,\ldots,x_{k},x_1,\ldots,x_{i-1})$  for every $i\in\{2,\ldots,k\}$.  
 An equivalent class of based loops is called a loop. Let $\DL(\mathcal{G})$ denote the set of loops 
 on $\mathcal{G}$. 
 The measure $\mup$ on the set of based loops of length at least two induces  a measure on loops denoted  by $\mu$. 
The Poisson loop sets on $\mathcal{G}$ is defined as a Poisson point process
$\mathcal{DP}$ with intensity $\text{Leb}\times\mu$  on $\RR_+\otimes \DL(\mathcal{G})$.
For $t>0$, let  $\DL^{(n)}_{t}$ be the projection of the set 
$\mathcal{DP}\cap ([0,t]\times \DL(\mathcal{G}))$ on  $\DL(\mathcal{G})$. The loop set $\DL^{(n)}_{t}$ 
defines a subgraph of $\mathcal{G}$. 
The  connected components of this subgraph form a partition of $S$ denoted by  $\mathcal{C}_{t}$. 
 the distribution of which  is computed in  \cite{LeJanLemaire13}. 
It follows that  if the graph $\mathcal{G}$ is the complete graph  $K_n$ 
 then  $\mathcal{C}_{-t\log(1-a)}$ has the same distribution as $\Pi_{S,p}(t)$. 
\begin{figure}[htb]
\begin{minipage}{0.4\textwidth}
\centering \includegraphics[scale=0.7]{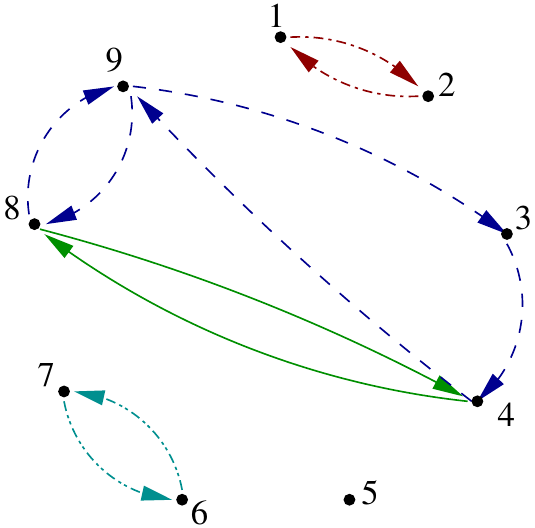}
\end{minipage}
\begin{minipage}{0.6\textwidth}
The loop set  is formed of the equivalent classes of a based loop of length 5 $\ell_1=(8,9,3,4,9)$ 
and three based loops of length 2,  $\ell_2=(1,2)$,  $\ell_3=(8,4)$ and $\ell_4=(6,7)$. 
The partition of $S=\{1,\ldots,9\}$ associated with this loop set is $\pi=(\{1,2\},\{3,4,8,9\},\{5\},\{6,7\})$. 
\end{minipage}
\caption{A loop set on the complete graph $K_9$}
\end{figure}
\end{ex}
\subsubsection{Restriction of the coalescent process to a subset \label{subsect:restriction}}
In our model: 
\begin{enumerate}[label=(\Roman*)]
\item\label{itm:proprI} 
each element of $S$ plays the same role,
\item\label{itm:proprII}  for every subset $A$ of $S$, the Poisson tuple set inside 
$A$ at time $t$,
$\mP_{S}(t,A)$ has the same distribution as $\mP_{A,p_{S|A}}\Big(\Gp\big(\frac{|A|}{|S|}\big)t\Big)$
where $$p_{S|A}(k)=\Big(\frac{|A|}{|S|}\Big)^k\frac{p(k)}{\Gp\big(\frac{|A|}{|S|}\big)} 
   \text{ for every }k\in \NN^*$$  and is independent of $\mP_{S}(t)\setminus\mP_{S}(t,A)$. 
\end{enumerate}
We can deduce from these  properties a formula for 
the  block size distribution of the coalescent process associated with $\mP_{S}(t,A)$ for every subset $A$ of $S$: 
\begin{prop}
\label{blocksizerestrict}
For $x\in S$,  let $\Pi_{S,p}^{(x)}(t)$ denote the block of the partition $\Pi_{S,p}(t)$ that contains $x$. 
 Let $A$ be a subset of $S$ that contains $x$. 
For $k\in\{1,\ldots,|S|\}$, 
\begin{equation}\label{blocksizedist}
\Pd\left(\Big\lvert\Pi_{A,p_{S|A}}^{(x)}\Big(tG_{p}\big(\frac{|A|}{ |S|}\big)\Big)\Big\rvert =k\right)=H_p(t,|S|,|A|,k)\Pd(|\Pi_{S,p}^{(x)}(t))|=k)
\end{equation}
where $$H_p(t,n,m,k)=\left(\prod_{i=1}^{k-1}\frac{m-i}{n-i}\right)e^{t\big(1-G_p(1-\frac{k}{n})-G_p(\frac{m}{n})+G_{p}(\frac{m-k}{n})\big)}$$
with the convention $\prod_{i=1}^{0}=1$.\\ 
In particular, 
\begin{equation}
\label{blocksizefunct}
\Ed\big( H_p(t, |S|, j, |\Pi_{S,p}^{(x)}(t)| ) \big)=1\quad \forall j \in\{1,\ldots,|S|\}. 
\end{equation} 
\end{prop}
\begin{rem}
 The system of equations \eqref{blocksizefunct} characterizes the distribution of $|\Pi_{S,p}^{(x)}(t))|$ since it can be written 
 as a lower triangular linear system with positive coefficients and with $\Pd(|\Pi_{S,p}^{(x)}(t))|=k)$ for $k\in\{1,\ldots,|S|\}$ as unknowns. 
When $p=\delta_{\{2\}}$, we recover a formula presented by R\`{a}th in a
 recent preprint (formula (1.1) of \cite{Rath17}): as applications of this formula,  
  R\`{a}th proposes in \cite{Rath17} new proofs of some properties of the component sizes of the  
  Erd\"os-R\'{e}nyi random graph in the subcritical and supercritical phases.  
\end{rem}
\begin{proof}[Proof of Proposition \ref{blocksizerestrict}]
Let $\Pi^{(x)}_{S,p}(t,A)$ denote the block of $x$ in the partition generated by $\mP_{S}(t,A)$. 
Let $B$ be a subset of   $A$ containing $x$: 
\begin{alignat*}{1}
\Pi^{(x)}_{S,p}(t)=B\iff&\Pi^{(x)}_{S,p}(t,A)=B \et \\
&\text{no tuple in}\;\mP_{S}(t)\;\text{contains  both }\;B \;\text{elements and}\;S\setminus A\;\text{elements}.
\end{alignat*}
 By property \ref{itm:proprII}, 
 \begin{alignat*}{1}
 &\Pd(\Pi^{(x)}_{S,p}(t,A)=B)=\Pd\left(\Pi^{(x)}_{A,p_{S|A}}\Big(t\Gp\big(\frac{|A|}{|S|}\big)\Big)=B\right) \et  \\
 & \Pd(\Pin[S,p]^{(x)}(t)=B)=\Pd(\Pi^{(x)}_{S,p}(t,A)=B)e^{-tI_{S,A}(B)}
 \end{alignat*}
 where 
 \begin{eqnarray*}
 I_{S,A}(B)&=&\mu_S(\{\omega\in\mW[S],\ \omega\cap B\neq \emptyset \et  \omega\cap (S\setminus A)\neq \emptyset\})\\
  &=&\mu_S(\mW[S])-\mu_S(\mW[S\setminus B]) - \mu_S(\mW[A])+\mu_S(\mW[A\setminus B])\\
  &=& 1-G_p\left(1-\frac{|B|}{|S|}\right)-G_p\left(\frac{|A|}{|S|}\right)+G_p\left(\frac{|A|-|B[}{|S|}\right). 
\end{eqnarray*}
Then, formula \eqref{blocksizedist} follows from property \ref{itm:proprI}. Indeed,  
\begin{eqnarray*}
\Pd(|\Pi^{(x)}_{S,p}(t,A)|=|B|)&=&\binom{|A|-1}{|B|-1}\Pd(\Pi^{(x)}_{S,p}(t,A)=B)\\
&=&
\binom{|A|-1}{|B|-1}\binom{|S|-1}{|B|-1}^{-1}\Pd(\Pin[S,p]^{(x)}(t)=|B|)e^{tI_{S,A}(B)}.
\end{eqnarray*}
Let us note that $H_p(t,n,m,k)=0$ if $m$ and $k$ are two integers such that $k\geq m+1$. 
Therefore, equality \eqref{blocksizedist}  holds for every $k\in\{1,\ldots,|S|\}$. 
The sum of over $k\in\{1,\ldots,|S|\}$ of \eqref{blocksizedist} yields equation \eqref{blocksizefunct}. 
\end{proof}
\section{Main results\label{sect:results}}
Let us recall that $p$ is a probability distribution on $\NN^*$ such that $p(1)<1$. 
To shorten the notations, we assume now that $S=\enu$ and  omit the reference
 to the probability $p$ in the notation: the shorten notations  $\mu_n$, $\mP_n$,  $\mP_n(t)$ and $\Pin(t)$  are used instead 
 of  $\mu_{S,p}$, $\mP_{S,p}$, $\mP_{S,p}(t)$ and $\Pi_{S,p}(t)$. 
\noindent Before stating the main results, let us introduce other notations.
\begin{itemize}
\item For $t>0$, $\mPs_n(t)$ denotes the projection of the set $\mP_n([0,t]\times \mWs[S])$ on $\mWs[S]$.
\item For $x\in\enu$,  $\Pi_{n}^{(x)}(t)$ designates the block of the partition $\Pi_{n}(t)$ that contains $x$.  
\item The $i$-th factorial moment of $p$ is denoted by 
$\mom[i]=\sum_{k=i}^{+\infty}k(k-1)\ldots (k-i+1)p(k)$ (let us recall that  its
probability generating function is denoted by $G_p$).  
\item Let $\pb$ denote  the size-biased probability measure defined on $\NN^*$  by 
$\pb(k)=\frac{(k+1)p(k+1)}{\moms}$ for every $k\in\NN^*$, where $\moms=\mom-p(1)$. 
 \item 
 For a positive real $\lambda$ and a probability distribution $\nu$ on $\RR$, let
$\CPois(\lambda,\nu)$ denote the compound Poisson distribution with parameters
$\lambda$ and $\nu$:  $\CPois(\lambda,\nu)$  is the probability distribution of $\sum_{i=1}^{N}X_i$,
where  $N$ is a Poisson distributed random variable with expected value
$\lambda$ and $(X_i)_i$ is a sequence of independent random variables with law
$\nu$, which is independent of $N$. 
\item For an integer $u\in\NN^*$, a positive real $a$ and a probability measure 
$\eta$ on $\NN$, we write $\BGW(u, a,\eta)$ for a BGW
process with family size distribution $\CPois(a,\eta)$ and $u$ ancestors. 
Finally for $t>0$ and $u\in\NN^*$, we use $\Tp[u](t)$ to  denote 
the total number of descendants of a $\BGW(u, t\moms,\tilde{p})$ process.  
\end{itemize}
\subsection{Time to coalescence}
The first result shows that the properties of having no singleton and of having 
only one block have the same sharp threshold function $\frac{n\log(n)}{\moms}$.   
\begin{thm}
\label{waittime}
Assume that $p$ is a probability distribution on $\NN^*$ such that $p(1)<1$,  $\mom$ is finite and $1-\Gp(1-h)=h\mom[1]+o(\frac{h}{\log(h)})$ as $h$ tends to $0^+$. 
Let $\swait$ and  $\cwait$ denote the first time $t$ for which the partition $\Pin(t)$ has no singleton and consists of a single block respectively. 
  For every $n\in\NN^*$, set $t_n=\frac{n}{\moms}(\log(n)+a+o(1))$, where $a$ is a fixed real. 
\begin{enumerate}
\item[(i)] For every $k\in\NN$, the probability that  $\Pin(t_n)$ has $k$ 
singletons converges to $\displaystyle{\frac{e^{-ak}}{k!}}e^{-e^{-a}}$ as $n$ tends to $+\infty$.
\item[(ii)] For every $k\in\NN$, the probability that $\Pin(t_n)$ consists  of a block of size $n-k$ and 
$k$ singletons converges to
$\displaystyle{\frac{e^{-ak}}{k!}e^{-e^{-a}}}$ as $n$ tends to $+\infty$.
\end{enumerate}
In particular, $\Big(\moms\dfrac{\swait}{n}-\log(n)\Big)_n$ and 
$\Big(\moms\dfrac{\cwait}{n}-\log(n)\Big)_n$ converge in distribution to 
the Gumbel distribution\footnote{The cumulative distribution function 
of the Gumbel distribution is $x\mapsto e^{-e^{-x}}$.}. 
\end{thm}
\begin{rem}\
\begin{itemize}
\item Assumptions on $p$ in Theorem \ref{waittime} are satisfied by 
probability distributions on $\NN^*$ having a finite second moment 
but not only:  
the  distribution $p$ on $\NN^*$ defined by $p(k)=\frac{4}{k(k+1)(k+2)}\; \text{for}\; k\in \NN^*$, 
satisfies the assumptions of Theorem \ref{waittime}  and has an infinite 
variance. Its generating function is 
\[G_p(z)=1+2(z-1)+2(z-1)^2\frac{1}{z^2}(-\log(1-z)-z)\; \forall z\in[0,1].
\]  
\item 
When $p=\delta_{\{d\}}$ with $d\geq 2$,  $\Pin$ corresponds to the 
partition made by the components of a random hypergraph process 
$\mathcal{G}_n$ that have similar properties as the $d$-uniform random 
hypergraph process. It is not surprising to recover the threshold function 
$\frac{n\log(n)}{d}$ for  connectivity of a $d$-uniform random hypergraph
 (see \cite{PooleThesis}). 
\item  When $p$ is a logarithmic distribution with parameter $a$ 
(example \ref{exloop}),  $$\moms=\frac{-a^2}{(1-a)\log(1-a)}.$$
\end{itemize}
\end{rem}
\subsection{Block sizes}
Let us turn to the study of the block size of an element at a time proportional to $n$:
\begin{thm} 
\label{thmblocksize}
Let  $t$ be a positive real.   Assume that $p$ has a finite third moment and that $p(1)<1$.  
Then there exists $C(t)>0$ such that for all $k,n\in \NN^*$ and $x\in\enu$, 
 $$|\Pd(|\Pin^{(x)}(nt)|\leq k)-\Pd(\Tp(t)\leq k)|\leq C(t)\frac{k^{2}}{n}.$$ 
\end{thm}
\begin{rem}\
Let us present some properties of the distribution of $\Tp(t)$ for $t>0$.  A BGW process with family size distribution 
$\CPois(t\moms,\pb)$ is subcritical if and only if $t<\frac{1}{\mom[2]}$. Let
$q_{p,t}$ denote the extinction probability of  such a BGW
process starting with one ancestor. It is a decreasing function of $t$. \\ 
Moreover,
\begin{equation}
\label{lawTGW}
\left\{\begin{array}{l}
\Pd(\Tp[u](t) = u) = 
\displaystyle{e^{-tu\moms}}\\
\Pd(\Tp[u](t) = k) = \displaystyle{\frac{u}{k}e^{-kt\moms}
\sum_{j=1}^{k-u}\frac{(tk\moms)^j}{j!}(\tilde{p})^{\star j}(k-u)}\quad \forall k\geq u+1.  
\end{array}\right.
\end{equation}
For $t\leq \frac{1}{\mom[2]}$, $\Tp[u](t)$ is almost surely finite and for 
$t>\frac{1}{\mom[2]}$, $\Pd(\Tp[u](t)<\infty)=(q_{p,t})^{u}<1$. \\
 For $t<\frac{1}{\mom[2]}$, the distribution of  $\Tp[u](t)$ has a light tail (that is there exists $s_0>0$ such that $\Ed(e^{s\Tp[u](t)})$ is finite for every $s\leq s_0$) if and only if $p$ is a light-tailed distribution (application of Theorem 1 in  \cite{Heyde64}).  
\end{rem}\medskip

The statement of Theorem \ref{thmblocksize} still holds if $|\Pin[\enu,p]^{(x)}(nt)|$ is replaced 
by  $|\Pin[\enu,p_n]^{(x)}(nt_n)|$ where $(t_n)_n$ and
$(p_n)_n$ converge rapidly to $t$ and $p$ respectively: 
\begin{corol}
\label{corolblocksize}
Let $(t_n)_n$ be a sequence of positive reals that converges to a real $t$ and 
let $(p_n)_n$ be a sequence of probability measures on $\NN^*$ that converges weakly 
to a probability measure $p$ on $\NN^*$ such that $p(1)<1$. If 
   $\sup_{n\in\NN^*}\sum_k k^3p_n(k)$ is finite, $t_n m^{*}_{p_n,1}-t\moms=O(\frac{1}{n})$ 
   and $\TV(\tilde{p}_n,\tilde{p})=O(\frac{1}{n})$ 
then there exists $C(t)>0$ such that $\forall n,k\in\NN^*$  and $\forall x\in\enu$, 
$$|\Pd(|\Pin[\enu,p_n]^{(x)}(nt_n)|\leq k)-\Pd(\Tp(t)\leq k)|\leq C(t) \frac{k^{2}}{n}.$$ 
\end{corol}
As a first application of Corollary \ref{corolblocksize}, let us 
consider the block size distribution 
for the partition defined by the Poisson tuple set inside a macroscopic subset of $\enu$ at time~$t$: 
\begin{corol}
\label{corolrestrictblocksize}
Assume that $p$ is a probability distribution on $\NN^*$ such that $p(1)<1$. 
Let  $a\in]0,1[$. Set  $a_n=\lfloor an\rfloor$ and $p_n=p_{\enu| \enu[a_n]}$ for $n\in \NN^*$. 
Let $\hat{p}_a$ denote the probability distribution on $\NN^*$ defined by 
$\hat{p}_a(k)=\frac{a^k p(k)}{\Gp(a)}$ for every $k\in\NN^*$.
\begin{itemize}
\item There exists $C_a(t)>0$ such that for every $k,n\in \NN^*$, 
\[
\Big\lvert\Pd\left(|\Pin[\llbracket a_n \rrbracket, p_n]^{(1)}\big(n\Gp(\frac{a_n}{n})t\big)|
\leq k\right)-\Pd\left(T^{(1)}_{\hat{p}_a}\big(\frac{\Gp(a)}{a}t\big)
\leq k\right)\Big\rvert\leq C_a(t) \frac{k^{2}}{n}.
\]
\item For every $u,k\in \NN^*$ such that $k\geq u$,  
\begin{equation}
\label{identtotalpop}
\Pd\left(T^{(u)}_{\hat{p}_a}\big(\frac{\Gp(a)}{a}t\big) = k\right) 
= a^{k-u}e^{tk(\mom-\Gp'(a))}\Pd(\Tp[u](t)=k).
\end{equation} 
\end{itemize}
\end{corol}
\begin{rem}
\begin{enumerate}
\item It is not necessary to assume that the first moments of $p$ are finite  since $\hat{p}_a$ has finite moments of all order for every $a\in]0,1[$. 
\item 
Formula \eqref{identtotalpop} for $u=1$ corresponds to the limit as $n$ tends to $+\infty$ of the identity \eqref{blocksizedist} 
satisfied by $|\Pin[\llbracket a_n \rrbracket, p_n]^{(1)}(n\Gp(\frac{a_n}{n})t)|$. 
\item If $t\geq 1/\mom[2]$ and $a$ is equal to the probability extinction $q_{p,t}$ of the \BGW$(1,t\moms,\tilde{p})$ process, then 
 $T^{(1)}_{\hat{p}_a}(\frac{\Gp(a)}{a}t)$  has the same distribution 
as the total population size of a \BGW$(1,t\moms,\tilde{p})$ process conditioned to become extinct. 
\end{enumerate}
\end{rem}\medskip

Properties \ref{itm:proprI} and \ref{itm:proprII} stated in Subsection \ref{subsect:restriction} 
and Corollary \ref{corolblocksize} allow to prove a joint 
limit theorem for the block sizes of two 
elements: 
\begin{corol}
\label{coroltwoblocks}
Let $x$ and $y$ be two distinct elements of $\enu$. 
For every $t>0$, $j,k\in\NN^*$, $\Pd(|\Pin^{(x)}(nt)|=j \et |\Pin^{(y)}(nt)|=k)$ converges to 
$\Pd(\Tp(t)=j)\Pd( \Tp(t)=k)$ as $n$ tends to $+\infty$. 
\end{corol}
\subsection{Coagulation equations}
Let us consider now the hydrodynamic behavior of the coalescent process $(\Pin(t))_{t\geq 0}$.  
A block of size $k$ can be seen as a cluster of $k$ particles of unit mass; 
at the same time, several clusters of masses $k_1,\ldots, k_j$ can merge into a single cluster 
of mass ${k_1+\ldots+k_j}$ at a rate proportional to the product $k_1\ldots k_j$.  The initial 
state corresponds to the monodisperse configuration ($n$ particles of unit mass).  
Corollary \ref{coroltwoblocks} is used to establish the convergence of 
the average number of blocks of size $k$ at time $nt$ as the number of particles 
$n$ tends to $+\infty$. The limit seen as a function of $k$  is a solution to coagulation equations:
\begin{thm}
\label{thm:hydrodyn}
Let $p$ be a probability measure on $\NN^*$ such that $p(1)<1$ and $\mom[3]$ is finite. \\
For $k\in \NN^*$, $n\in \NN$ and $t>0$, let $\rho_{n,k}(t) =
\frac{1}{nk}\sum_{x=1}^{n}\un_{\{|\Pin^{(x)}(nt)| = k\}}$ be the average number of
blocks of size $k$  and let $\rho_{k}(t) =\frac{1}{k}P(\Tp(t) = k)$.
\begin{enumerate}
\item $(\rho_{n,k}(t))_n$ converges to $\rho_{k}(t)$ in $L^2$ for every
$t>0$. 
\item $(\rho_{k}(t),\ k\in \NN^* \text{ and } t\geq 0)$ is a
solution to the following coagulation equations:
\begin{equation}
\label{coageq}
\frac{d}{dt}\rho_k(t) =
\sum_{j=2}^{+\infty}p(j)\mK{j}(\rho(t),k)
\end{equation}
where 
\begin{equation}
\label{Gj}
\mK{j}(\rho(t),k)=
\Big(\sum_{\substack{(i_1,\ldots,i_{j})\in(\NN^*)^{j}\\  i_1+\cdots
+i_j=k}}\prod_{u=1}^{j}i_u\rho_{i_u}(t)\Big)\un_{\{j\leq
k\}}-jk\rho_{k}(t).
\end{equation}
\end{enumerate} 
\end{thm}
\begin{rem}\
\begin{enumerate}
\item Consider a medium with integer mass particles and let $\rho_{k}(t)$ denote
the density of mass $k$ particles at time $t$. Equation \eqref{coageq} describes
the evolution of $\rho_{k}(t)$ if for every $j\geq 2$ the number of aggregations
of $j$ particles of mass $i_1,\ldots,i_j$ in time interval $[t,t+dt]$ is assumed
to be $p(j)\rho_{i_1}(t)\ldots\rho_{i_j}(t)\kappa_j(i_1,\ldots,i_j) dt$, 
where $\kappa_j(i_1,\ldots,i_j)=i_1\cdots i_j$ is the multiplicative kernel. \\
The first term in $\mK{j}$ describes the formation of a particle of mass $k$ by
aggregation of $j$ particles, the second term $jk\rho_{k}(t)$ can be 
decomposed into the sum of the following two terms: 
\begin{itemize}
\item ${\displaystyle jk\rho_{k}(t)\Big(\sum_{i=1}^{+\infty}
i\rho_i(t)\Big)^{j-1}}$ that  describes the ways  a particle of
mass $k$ can  be aggregated with $j-1$ other particles. 
\item ${\displaystyle jk\rho_{k}(t)\sum_{h=1}^{j-1}\binom{j-1}{h}\Big(\sum_{i=1}^{+\infty}i(\rho_{i}(0)-\rho_i(t))\Big)^h \Big(\sum_{u=1}^{+\infty}
u\rho_u(t)\Big)^{j-1-h}}$. This term  is null if the total mass is preserved. 
Otherwise, the decrease of the total mass can be interpreted as the appearance of a 
`gel' and this term  describes the different ways a particle of mass $k$ can 
be aggregated with the gel and other particles. 
\end{itemize}
\item The system of equations 
\[
\frac{d}{dt}\rho_{k}(t) = \mK{2}(\rho(t),k), \quad \forall k\in\NN^*
\]
corresponds to the Flory's coagulation equations with the multiplicative
kernel (see equation \eqref{Floryeq}). \\
An application of Theorem \ref{thm:hydrodyn} with $p=\delta_{\{j\}}$ for $j\geq 2$, shows that
an approximation of  the solution of the system of equations  
\[
\frac{d}{dt}\rho_k(t) = \mK{j}(\rho(t),k), \quad \forall k\in\NN^*
\]
can be constructed by drawing  tuples of fixed size $j$. 
\begin{corol}
\label{corol:hydrodynj}
 Let $j$ be an integer greater than or equal to $2$. 
 For $k\in \NN^*$, $n\in \NN$ and $t>0$, let $\rho^{(j)}_{n,k}(t)$ be the average number of
blocks of size $k$ in the partition $\Pin[\enu,\delta_{\{j\}}](\frac{nt}{j})$.
\begin{enumerate}
\item $(\rho^{(j)}_{n,k}(t))_n$ converges to 
$\rho^{(j)}_{k}(t)=e^{-tk}\dfrac{(tk)^{\frac{k-1}{j-1}}}{k^2(\frac{k-1}{j-1})!}\un_{\{k-1\in (j-1)\NN\}}$ 
in $L^2$ for every $t>0$. 
\item $(\rho^{(j)}_k(t),\ k\in \NN^* \text{ and } t\geq 0)$ is a
solution to the following coagulation equations:
\begin{equation}
\label{coageqj}
\frac{d}{dt}\rho_k(t) =\mK{j}(\rho(t),k)
\end{equation}
where $\mK{j}$ is defined by equation \eqref{Gj}. 
\end{enumerate}
\end{corol}
\item The function $\rho(t)$ defined by  $\rho_{k}(t)=\frac{1}{k}\Pd(\Tp(t)=k)$  for every $k\in \NN^*$ 
gives an explicit solution of \eqref{coageq}  with mass-conserving property on the interval 
$[0;\frac{1}{\mom[2]}]$.  
Its second moment $\displaystyle{\sum_{k=1}^{+\infty}k^2\rho_{k}(t)=(1-t\mom[2])^{-1}}$  
diverges as $t$ tends to $\mom[2]$.  
\end{enumerate}
\end{rem}
\subsection{Phase transition}
As a last application of Theorem \ref{thmblocksize}, we  show that the block sizes
of  $(\Pin(nt))_{t\geq 0}$ undergo
a phase transition at $t=\frac{1}{\mom[2]}$ similar to the phase transition of the 
Erd\"os-R\'enyi random graph process and present some bounds for the sizes of the two largest blocks in the three phases:  
\begin{thm}\label{thmtransition} Let $p$ be a probability measure on $\NN^*$ such that $p(1)<1$. 
Let $B_{n,1}(nt)$ and $B_{n,2}(nt)$ denote the first and second largest blocks of  $\Pin(nt)$. 
\begin{enumerate}
\item \textbf{Subcritical regime.} 
\label{subcrit}
Let $0<t<\frac{1}{\mom[2]}$. 
\begin{enumerate}
\item Assume that $p$ has a finite moment of order $u$ for some $u\geq 3$.
If   $(a_n)_n$ is a sequence of reals that tends to $+\infty$, 
then
$\Pd(|B_{n,1}(nt)|> a_n n^{\frac{1}{u-1}})$ converges 
to $0$ as $n$ tends to $+\infty$.
\item  
Assume that $\Gp$ is finite on $[0,r]$ for some $r>1$. Let $L_{t}$ denote
 the moment-generating function of the  $\CPois(t\moms,\pb)$-distribution.
Set\footnote{$h(t)$ is
the value of the Cram\'er function at 1 of the 
$\CPois(t\moms,\pb)$-distribution.}  $$h(t)=\sup_{\theta>0}(\theta-\log(L_{t}(\theta))).$$
Then $h(t)>0$ and for every $a>(h(t))^{-1}$, 
$\Pd(|B_{n,1}(nt)|> a\log(n))$ converges to $0$ as $n$ tends
to $+\infty$.
\end{enumerate}
\item \textbf{Supercritical regime.}
Assume that $p$ has a finite moment of order three and that $t>\frac{1}{\mom[2]}$.  
Let $q_{t}$ denote the extinction probability of a 
BGW process with one
progenitor and $\CPois(t\mom,\pb)$ offspring distribution.\\
For every $a\in]1/2,1[$, there exist $b>0$ and $c>0$ such that 
$$\Pd\left(\big\lvert |B_{n,1}(nt)|-(1-q_{t})n\big\rvert\geq n^a\right)+\Pd\left(|B_{n,2}(nt)|\geq c\log(n)\right) = O(n^{-b}).$$ 
\item \textbf{Critical window.}
Assume that $p$ has a finite moment of order three.  For every $\theta\geq 0$, there exists a constant $b>0$ such that for every $c>1$ and $n\in \NN^*$ 
\begin{equation}
\label{critupperbound}
\Pd\left(|B_{n,1}(\frac{n}{\mom[2]}(1+\theta n^{-1/3}))| > cn^{2/3}\right)\leq \frac{c}{b}.
\end{equation}
\end{enumerate}
\end{thm}
\begin{rem} Let us provide further information on the subcritical regime (${0<t<\frac{1}{\mom[2]}}$). 
\begin{itemize}
\item The upper bound for $|B_{n,1}(nt)|$ given in assertion 1.(b) is reached when ${p=\delta_2}$; Indeed, it is known since the Erd\"os and R\'enyi's paper \cite{ErdosRenyi60} that $\frac{1}{\log(n)}|B_{n,1}(\frac{ns}{2})|$ converges in probability to $(s-1-\log(s))^{-1}$ as $n$ tends to $+\infty$, when $0<s<1$.
\item Let us assume now that $p$ is regularly varying  with index $-\alpha<-3$:  there exists a slowly varying function $\ell$ such that $\sum_{j>k}p(j)=k^{-\alpha}\ell(k)$ $\forall k\in\NN$. Assertion 1.(a) implies that for every $\varepsilon>0$,  $\Pd(|B_{n,1}(nt)|> n^{\frac{1}{\alpha-1}+\varepsilon})$ tends to $0$ as $n$ tends to $+\infty$. Let us note that  $n^{\frac{1}{\alpha-1}}$ corresponds to the order of the largest size for the total progeny of $n$ independent \BGW$(1,t\moms,\tilde{p})$ processes. Indeed, one can show that:
\begin{quote}
 If $T_{1},\ldots,T_n$ are the total progeny of $n$ independent \BGW$(1,t\moms,\tilde{p})$ processes, then for every $1<\alpha_1<\alpha<\alpha_2$, 
 \[\Pd(\max_{i=1,\ldots,n}T_{i}> n^{\frac{1}{\alpha_1-1}})+\Pd(\max_{i=1,\ldots,n}T_{i}< n^{\frac{1}{\alpha_2-1}})\underset{n\rightarrow +\infty}{\rightarrow 0}.\] 
 \end{quote}
 An application of the second moment method  allows to prove that the largest block size actually grows faster than a positive power of $n$ in the subcritical regime, but gives an exponent smaller than expected: 
\begin{thm}
\label{subcritlowerbound}
Assume  that $p$ is regularly varying  with index $-\alpha<-3$. \\
If $t< \frac{1}{\mom[2]}$ then  for every $\alpha'>\alpha$, $\Pd(\max_{x\in \enu}|\Pin^{(x)}(nt)|\leq n^{\frac{1}{1+\alpha'}})$ converges to $0$ as $n$ tends to $+\infty$.
\end{thm}
\end{itemize} 
\end{rem}
\section{The number of singletons and the coalesence time\label{sec:coaltime}}
 
In a first part, we investigate the distribution of the number of singletons in the partition 
at time $t$ and the asymptotic distribution of the first time $t$
 at which $\Pin(t)$ does not have singleton.  
In a second part, we show that the asymptotic distribution of the coalescence time as 
$n$ tends to $+ \infty$ (that is the first time $t$ at which  $\Pin(t)$ consists of a single block)  
coincides with the asymptotic distribution of the first time 
 $\Pin$ does not have singleton.  
\subsection{Number of singletons}
Let us observe that the block of an element $x$ in the partition $\Pin(t)$ is a 
singleton if and only if 
tuples in $\mPs_n(t)$ do not contain  $x$. The model is thus a variant of a coupon 
collector's problem with group drawings.  
 The exclusion-inclusion lemma provides an exact 
formula for the number of singletons in $\Pin(t)$. 
\begin{prop}
\label{nbsingleton}
Let $Y_{n,p}(t)$ denote the number of singletons in $\Pin(t)$. For every ${k\in\{0,\ldots,n\}}$, 
$$\Pd(Y_{n,p}(t)=k)=\sum_{j=0}^{n-k}(-1)^{j}\frac{n!}{k!j!(n-k-j)!}
\exp\Big(-t\big(1-\Gp(1-\frac{k+j}{n})-(k+j)\Gp(\frac{1}{n})\big)\Big).$$
\end{prop}
\begin{proof}
 Let $N_{n}^{(x)}(t)$ denote the number of  tuples in $\mPs_n(t)$ that contain the element~$x$. 
 $$\Pd(Y_{n,p}(t)=k)=\sum_{F\subset\enu,\ |F|=k}\Pd\left(N_{n}^{(x)}(t)>0\;\forall x\not\in F\; \et\;
 \sum_{x\in F}N_{n}^{(x)}(t)=0\right).
 $$
 By the exclusion-inclusion lemma, 
 $$\Pd\left(N_{n}^{(x)}(t)>0\;\forall x\not\in F\; \et\;
 \sum_{x\in F}N_{n}^{(x)}(t)=0\right)= \sum_{K\subset \enu\setminus F}(-1)^{|K|}
 \Pd\left(\sum_{x\in F\cup K}N_{n}^{(x)}(t)=0\right).$$
We conclude by noting that for any subset $A\subset \enu$,  
\[\Pd\left(\sum_{x\in A}N_{n}^{(x)}(t)=0\right)=
\exp\Big(-t\mu(w\in \mWs,\ w\cap A\neq \emptyset)\Big)
\]
with 
\begin{multline*}
\mu\left(w\in \mWs,\ w\cap A\neq \emptyset\right)=1-\mu\left(\mW[\enu\setminus A]\right)-
\mu\left(\bigcup_{a\in A}\mW[\{a\}]\right)\\ 
=1-\Gp\left(1-\frac{|A|}{n}\right)-|A|\Gp\left(\frac{1}{n}\right).
\end{multline*}
\end{proof}
An analogy to the classical coupon collector's problem provides an idea  of the average time 
 until $\Pin$ has no singleton: the number of tuples in $\mPs_n(t)$ is 
 in average $t\mu_n(\mWs)$ and the length of nontrivial tuples is in average 
 $$(\mu_n(\mWs))^{-1}\sum_{k=2}^{+\infty}kp(k)(1-\frac{1}{n^{k-1}}).$$ Therefore, the total 
 number of elements drawing before time $t$ and belonging to nontrivial tuples is in average
 $t(\moms + O(\frac{1}{n}))$. If the elements are drawn one by one  and not  by groups of 
 random sizes, then the solution of the classical coupon 
collector's problem, suggests that the time until $\Pin$ has no singleton 
would be around  $\frac{n\log(n)}{\moms}$. The following result shows that this analogy holds in particular when $p$ has a finite second moment. 
\begin{bibli}[\ref{waittime}.(i)]
\label{nbisolated}
Assume that $\mom$ is finite and $1-\Gp(1-h)=h\mom[1]+o(\frac{h}{\log(h)})$ as $h$ tends to $0^+$. 
\begin{enumerate}
\item
 For every $a\in\RR$, the number of singletons in $\Pin$ at time $\frac{n}{\moms}(\log(n)+a+o(1))$ converges in distribution to 
 the Poisson distribution  with parameter $e^{-a}$ as $n$  tends to $+\infty$. 
 \item Let   $\swait$ denote the first time $t$ when $\Pin(t)$ has no singleton. The sequence 
  $$\Big(\moms\frac{\swait}{n}-\log(n)\Big)_n$$ converges in 
distribution to the Gumbel distribution.
\end{enumerate} 
\end{bibli}
\begin{proof}
Set $t_n=\frac{n}{\moms}(\log(n)+a+o(1))$.  Using the notation introduced in proof of Proposition \ref{nbsingleton}, 
the number of singletons in $\Pin(t_n)$ is 
$$Y_{n,p}(t_n)=\sum_{x\in\enu}\un_{\{N_{x}(t_n)=0\}}.$$ By the theory of 
moments,  it suffices to
show that the factorial moments of any order of $Y_{n,p}(t_n)$ converge to those of the Poisson
distribution with parameter $e^{-a}$ to prove the convergence in distribution. \\ 
 Let $k\in\NN^*$.
 The $k$-th factorial moment of $Y_{n,p}(t_n)$ is 
\begin{multline*}
\Ed(Y_{n,p}(t_n))_k:=\Ed(Y_{n,p}(t_n)(Y_{n,p}(t_n)-1)\ldots (Y_{n,p}(t_n)-k+1))\\=\sum_{F\subset \enu,\ |F|=k}k!\Pd(\sum_{x\in F}N_{n}^{(x)}(t_n)=0).
\end{multline*}
  Therefore, 
  \[\Ed(Y_{n,p}(t_n))_k=n^k\prod_{i=1}^{k-1}(1-\frac{i}{n})
  \exp\Big(-t_n(1-\Gp(1-\frac{k}{n})-k\Gp(\frac{1}{n}))\Big).\]
  Set $I_{n,k}=-t_n\Big(1-\Gp(1-\frac{k}{n})-k\Gp(\frac{1}{n})\Big)+k\log(n)$. It can be rewritten
  $$I_{n,k} = -t_n\Big(1-\Gp(1-\frac{k}{n})-\frac{k}{n}\mom-k(\Gp(\frac{1}{n})-\frac{p(1)}{n})\Big)-ak+o(1).$$ 
  Therefore, $I_{n,k}$ converges to $-ak$ as $n$ tends to $+\infty$ since $\Gp(\frac{1}{n})=\frac{1}{n}p(1) +O(\frac{1}{n^2})$ and 
  $1-\Gp(1-\frac{k}{n})-\frac{k}{n}\mom=o((n\log(n))^{-1})$ by assumption. 
  This shows that $\Ed(Y_{n,p}(t_n))_k$ converges to $\exp(-ka)$ for every $k\in \NN^*$.\\ 
 To deduce the assertion for $\swait$, it suffices to note that for every $x\in\RR$, 
  \[\moms\frac{\swait}{n}-\log(n)\leq x\; \iff\; Y_{n,p}(t_{n,x})=0.\] 
  where $t_{n,x}=\frac{n(\log(n)+x)}{\moms}$.
\end{proof}
\subsection{Time to coalescence}
Let $\cwait$ denote the first time $t$ for which the partition $\Pin(t)$ consists 
of a single block. 
\begin{bibli}[\ref{waittime}.(ii)]
\label{coaltime}
Assume that $\mom$ is finite and $1-\Gp(1-h)=h\mom[1]+o(\frac{h}{\log(h)})$ as $h$ tends to $0^+$. 
  For every $n\in\NN^*$, set $t_n=\frac{n}{\moms}(\log(n)+a+o(1))$ where $a$ is a fixed real. \\
For every $k\in \NN$, the probability that $\Pin(t_n)$ consists  of a block of size $n-k$ and 
$k$ singletons converges to $\displaystyle{\exp(-e^{-a})\frac{e^{-ak}}{k!}}$ as $n$ tends to $+\infty$.\\
In particular, $\Big(\moms\dfrac{\cwait}{n}-\log(n)\Big)_n$ converges in distribution  to the Gumbel distribution. 
\end{bibli}
\begin{proof}
We adapt the  proof of Theorem 5.6 given  in \cite{LeJanLemaire13} in the context  of Markov loops in the complete graph 
(that is when $p$  is a logarithmic distribution). 
 For $k\in \NN$, let  $H_{n,k}$ denote the event 
 `\textit{$\Pin(t_n)$ consists only of a block of size $n-k$  and $k$ singletons}' and 
 let $J_n$ be the event 
 `\textit{$\Pin(t_n)$ has at least two blocks of size greater or equal to $2$}'.
 We have to prove that $\Pd(H_{n,k})$ converges to $e^{-e^{-a}}\frac{e^{-ka}}{k!}$. 
As $\Pd(Y_{n,p}=k)$ converges to $e^{-e^{-a}}\frac{e^{-ka}}{k!}$  
 and is equal to  ${\Pd(H_{n,k})+\Pd(\{Y_{n,p}=k\}\cap J_n)}$ for $n\geq k+2$, it suffices to prove
 that $\Pd(J_n)$ converges to $0$. 
For a subset $F$ of $\enu$, let $b_{n}(F)$ denote the probability that 
$F$ is a block of $\Pin(t_n)$ and set $S_{n,r}=\sum_{F\subset \enu, |F|=r} b_{n}(F)$ for $r\in \enu$. 
The proof consists in showing that $\sum_{r=2}^{\lfloor n/2\rfloor}S_{n,r}$, 
which is an upper bound of $\Pd(J_n)$, converges to $0$. \\
For every  subset $A$ of $\enu$, let $\mP_n(t,A)$  denote the set of tuples $w\in\mP_n(t)$  the elements of which are in $A$. 
Similarly, let $\mPs_n(t,A)$ denote the subset  of nontrivial tuples of $\mPs_n(t,A)$. 
As $\mP_n(t,F)$ is independent of $\mP_n(t)\setminus\mP_n(t,F)$,  $b_{n}(F)=b^{(1)}_{n}(F)b^{(2)}_{n}(F)$ where:
\begin{itemize}
 \item $b^{(1)}_{n}(F)$ is the probability that the partition associated with $\mP_n(t_n,F)$ 
 consists of the block $\{F\}$, 
 \item $b^{(2)}_{n}(F)$ is the probability that there is no tuple $w\in\mP_n(t_n)$ containing both 
 elements of  $F$ and  $F^c$. 
\end{itemize}
Let $\delta\in]0,1[$. 
 For $|F|\geq n^{1-\delta}$,  it is sufficient to replace $b^{(1)}_{n}(F)$ by $1$ as we show that $(b^{(2)}_{n}(F))_n$ 
 converges to $0$ rapidly.  For $2\leq |F|< n^{1-\delta}$, we use that  $b^{(1)}_{n}(F)$ is bounded by  the probability that the total
number of elements in nontrivial tuples of $\mP_n(t_n,F)$ are greater or equal to $|F|$. The value of this upper bound depends 
only on $|F|$ and $n$. Let denote it $\bar{b}^{(1)}_n(|F|)$. 
 $$S_{n,r}\leq\left\{\begin{array}{ll}\sum_{\stackrel{F\subset \enu}{|F|=r}} b^{(2)}_{n}(F)&\text{if}\;r\geq  n^{1-\delta}\\
\bar{b}^{(1)}_n(r)\sum_{\stackrel{F\subset \enu}{|F|=r}} b^{(2)}_{n}(F)&\text{if}\; 2\leq r< n^{1-\delta}
\end{array}\right.\text{where}\; \bar{b}^{(1)}_n(r)=\Pd\Big(\sum_{w\in \mPs_n(t_n,\llbracket{r}\rrbracket)}\ell(w)\geq r\Big).$$

The expression of $b^{(2)}_{n}(F)$ is $\exp\Big(-t_n(1-\Gp(\frac{|F|}{n})-\Gp(1-\frac{|F|}{n}))\Big)$. 
 Using that $$\binom{n}{r}\leq \frac{1}{\sqrt{2\pi r}\sqrt{1-\frac{r}{n}}}(\frac{n}{r})^r(1-\frac{r}{n})^{-(n-r)}$$ 
(see for example \cite{BollobasBook}, formula 1.5 page 4), we obtain:
$$\sum_{F\subset \enu, |F|=r} b^{(2)}_{n}(F)\leq \frac{1}{\sqrt{r}}\exp(-nf_n(\frac{|F|}{n})),$$ 
where $f_n$ is the function defined by:
$$f_n(x)=x\log(x)+(1-x)\log(1-x)+\frac{t_n}{n}(1-\Gp(x)-\Gp(1-x))\;\text{for}\;x\in]0,1[.$$
To conclude, we need the following two lemmas: 
\begin{lem}
\label{lemtailF}
 Let $\delta$ and $\bar{\delta}$ be two positive reals such that  $0<\bar{\delta}<\delta<1$. Let $a\in\RR$.
Set $t_n=\frac{n}{\moms}(\log(n)+a+o(1))$ for every $n\in\NN$. \\
 There exists $n_{\delta,\bar{\delta}}>0$ such that for every $n\geq n_{\delta,\bar{\delta}}$,  
 and $F\subset \enu$ with $2\leq |F|\leq n^{1-\delta}$, 
\[\Pd\Big(\sum_{w\in \mPs_n(t_n,F)}\ell(w)\geq |F|\Big)
\leq n^{-\frac{\bar{\delta}}{2}|F|}.\]
\end{lem}
\begin{lem}
\label{lemfunct}
Let $f_n$ be  the function defined by:
$$f_n(x)=x\log(x)+(1-x)\log(1-x)+\frac{t_n}{n}(1-\Gp(x)-\Gp(1-x))\;\forall x\in]0,1[.$$
Let $(u_n)$ be a positive sequence  such that $\liminf_n\frac{u_n}{\log(n)}>0$.  
For every $\delta\in]0,1[$,  there is an integer $n_{\delta}>0$ such that for $n\geq n_{\delta}$, 
 \begin{itemize}
  \item $f_n(x)\geq \frac{1-\delta}{2n^{\delta}}\log(n)$ for every $x\in[n^{-\delta},1/2]$,
  \item $f_n(x)+xu_n \geq \frac{u_n}{n}$ for every $x\in[\frac{2}{n},1/2]$.
 \end{itemize}
\end{lem}
Before presenting the proofs of the two lemmas, let us apply them  to complete the proof of Theorem \ref{waittime}. 
By Lemma \ref{lemfunct}, for every $0<\bar{\delta}<\delta <1$, there exists $n_{\delta,\bar{\delta}}\in\NN$, 
such that for every $n>n_{\delta,\bar{\delta}}$, 
$$
S_{n,r}\leq 
\left\{\begin{array}{ll} 
\frac{1}{\sqrt{r}}\exp(-nf_n(\frac{r}{n}))& \text{ if } 
r\in[n^{1-\delta}, \lfloor n/2\rfloor]\\
\frac{1}{\sqrt{r}}\exp\Big(-n\big(f_n(\frac{r}{n})+\frac{r}{n}\frac{\bar{\delta}}{2}\log(n)\big)\Big)
&\text{ if } r\in[2, n^{1-\delta}].\\
\end{array}\right.$$
We deduce from Lemma \ref{lemfunct}  that for sufficiently large values of $n$, 
$$S_{n,r}\leq\left\{\begin{array}{ll} \frac{1}{\sqrt{r}}\exp(-\frac{1-\delta}{2}n^{1-\delta}\log(n))& 
\text{ if } r\in[n^{1-\delta}, \lfloor n/2\rfloor]\\
\frac{1}{\sqrt{r}}n^{-\frac{\bar{\delta}}{2}}&\text{ if } r\in[2, n^{1-\delta}].
\end{array}\right.
$$
Thus for sufficiently large values of $n$, $\Pd(J_n)\leq n^{1-\delta-\frac{\bar{\delta}}{2}}+n\exp(-\frac{1-\delta}{2}n^{1-\delta}\log(n))$. 
If we take $\delta=\frac{3}{4}$ and $\bar{\delta}=\frac{2}{3}$, we obtain that for 
sufficiently large values of $n$, $${\Pd(J_n)\leq n^{-1/12}+ ne^{-\frac{1}{8}n^{1/4}\log(n)}}.$$
It remains to prove Lemma \ref{lemtailF} and Lemma \ref{lemfunct}. 
\paragraph{Proof of Lemma \ref{lemtailF}} 
 The random variable $N_n(F):=\displaystyle{\sum_{w\in \mPs_n(t_n,F)}\ell(w)}$
 has a compound Poisson distribution $\CPois(t_n\beta_{n,F},\nu_{n,F})$, where
 \begin{itemize}
 \item $\beta_{n,F}=\Gp(\frac{|F|}{n})-|F|\Gp(\frac{1}{n})$, 
 \item $\nu_{n,F}(j)=\frac{1}{\beta_{n,F}}\mu(w\in\mWs[F],\ \ell(w)=j)=\frac{p(j)}{\beta_{n,F}}
 \Big((\frac{|F|}{n})^j-\frac{|F|}{n^j}\Big)\; \forall j\in\NN^*$. 
\end{itemize}
Its probability generating function at   $0\leq s\leq\frac{n}{|F|}$ is:
\begin{multline*}
G_{N_n(F)}(s)=\exp\Big(-t_n\beta_{n,F}(1-G_{\nu_{n,F}}(s))\Big)
\\ =\exp\Big(-t_n\big(\Gp(\frac{|F|}{n})-\Gp(s\frac{|F|}{n})-|F|(\Gp(\frac{1}{n})-\Gp(\frac{s}{n}))\big)\Big).
\end{multline*}
For $2\leq r\leq n$ and $0<\theta\leq \log(\frac{n}{r})$, set $$\psi_{n,r}(\theta)= \theta r +t_n\Big(\Gp(\frac{r}{n})-r\Gp(\frac{1}{n})-
\Gp(e^{\theta}\frac{r}{n})+r\Gp(\frac{e^{\theta}}{n})\Big).$$ 
By  Markov's inequality 
$\Pd(N_n(F)\geq |F|)\leq \exp(-\psi_{n,|F|}(\theta))$ for every ${0<\theta\leq \log(\frac{n}{|F|})}$. 
As $\Gp'$ and $\Gp''$ are increasing functions on $[0,1[$, for $s\in[1,\frac{n}{2r}]$, 
\begin{multline*}\Gp(s\frac{r}{n})-\Gp(\frac{r}{n})-r(\Gp(\frac{s}{n})-\Gp(\frac{1}{n}))\\\leq 
\frac{r}{n}(s-1)(\Gp'(s\frac{r}{n})-\Gp'(\frac{1}{n}))\leq \frac{r}{n^2}(s-1)(rs-1)\Gp''(1/2).
\end{multline*}
Thus for every $0<\theta\leq \log(\frac{n}{2r})$, $\psi_{n,r}(\theta)\geq rh_{n,r}(\theta)$ with $h_{n,r}(\theta)=\theta -t_n\frac{r}{n^2}\Gp''(1/2)e^{2\theta}$. 
The function $h_{n,r}$ has a maximum point at $\theta_{n,r}=\frac{1}{2}\log(\frac{n^2}{2t_nr\Gp''(1/2)})$, 
which is less than $\log(\frac{n}{2r})$ for every $r\leq n$ when $n$ is large enough. 
Its value at $\theta_{n,r}$ is \[h_{n,r}(\theta_{n,r})=\frac{1}{2}(\log(n)-\log(r)-\log(\frac{t_n}{n}))+O(1).\]
Therefore, for every $0<\bar{\delta}<\delta<1$, there exists $n_{\delta,\bar{\delta}}\in\NN$
such that for every $n\geq n_{\delta,\bar{\delta}}\in\NN$ and $2\leq r\leq n^{1-\delta}$,  
$h_{n,r}(\theta_{n,r})\geq \frac{\bar{\delta}}{2}\log(n)$ and thus 
$\Pd(N_n(F)\geq |F|)\leq \exp(-|F|\frac{\bar{\delta}}{2}\log(n))$ for $2\leq|F|\leq n^{1-\delta}$.    
\paragraph{Proof of Lemma \ref{lemfunct}}
The proof consists in showing that for sufficiently large $n$, $f_n$ and 
$\bar{f}_n:x \mapsto f_n(x)+xu_n $ are
increasing functions in $]n^{-\delta},\frac{1}{2}[$ and $]\frac2n,\frac{1}{2}[$ respectively and to compute
their values at $n^{-\delta}$ and $\frac2n$ respectively. 
Let us prove the result for the function $f_n$. 
By computations, we obtain that for every $x\in]0,1[$, 
 $$f_n'(x)=\log(x)-\log(1-x)+\frac{t_n}{n}(\Gp'(1-x)-\Gp(x))\;\et$$   
 $$f_n''(x)=\frac{1}{x(1-x)}\Big(1-\frac{t_n}{n}\sum_{k=2}^{+\infty}k(k-1)g_k(x)\Big)\;\text{where}\;
g_{k}(x)=x(1-x)(x^{k-2}+(1-x)^{k-2}).$$ 
The first derivative of $g_k$ is positive on $]0,\frac{1}{2}[$. 
As the value of $1-\frac{t_n}{n}\sum_{k=2}^{+\infty}k(k-1)g_k$ at $0$ is $1$ and at $\frac{1}{2}$ 
is negative for sufficiently large $n$, we deduce that for sufficiently large $n$, there exists 
$a_n\in]0,\frac{1}{2}[$ such that $f_n'$ is increasing in $]0,a_n[$ and decreasing in $]a_n,\frac{1}{2}[$. 
As  $f_n'(\frac{1}{2})=0$ and $f_n'(n^{-\delta})>0$ for sufficiently large $n$,   $f_n$ is an increasing function in 
$]n^{-\delta},\frac{1}{2}[$ for sufficiently large $n$. Finally, using that 
$1-\Gp(1-s)-\Gp(s)=s\moms + o(\frac{s}{\log(s)})$ as $s$ tends to $0$, we obtain  
$f_n(\frac{1}{n^{\delta}})=\dfrac{1-\delta}{n^\delta}\log(n)+O(n^{-\delta})$. 
We deduce that for sufficiently large $n$, $f_n(x)\geq \dfrac{1-\delta}{2n^\delta}\log(n)$ 
$\forall x\in[\dfrac{1}{n^{\delta}},\frac{1}{2}]$. \\
As $\bar{f}'_n=f'_n+u_n$,  $f'_n(\frac{2}{n})=o(1)\log(n)$ and $\bar{f}_n(\frac{2}{n})=\frac{2}{n}(O(1)+u_n)$, 
we obtain that $\bar{f}_n(x)\geq \frac{u_n}{n}$ $\forall x\in[\frac{2}{n},\frac{1}{2}]$ for  sufficiently large $n$.
\end{proof}
\section{Block exploration procedure and associated BGW process \label{sect:exploration}}
In this section, we describe an exploration procedure modeled on 
the Karp \cite{Karp} and Martin-L\"of \cite{MartinLof} exploration algorithm. 
The aim of this  procedure is  to find the block of an element $x$ in the partition $\Pin(t)$
(this block is denoted by $\Pin^{(x)}(t)$),  
and to construct a  BGW 
process such that its total population size  is an upper bound of
$|\Pin^{(x)}(t)|$. 
\subsection{Block exploration procedure} 
For every  subset $A$ of $\enu$, and $x\in A$, let $\mP_{n,x}(t,A)$ denote the set 
of tuples $w\in\mP_n(t,A)$ that contain $x$ and let $\mPs_{n,x}(t,A)$ denote those that are nontrivial. 
Let  define the set of `neighbours' of $x$ in $A$ as
$$\Ne_{x}(t,A)=\{y\in A\setminus \{x\},\; \exists w \in\mP_{n,x}(t,A)\;\text{that contains}\;y\}.$$ 

In each step of the algorithm, an element of $\enu$ is either \emph{active}, \emph{explored}
or \emph{neutral}. Let $A_k$ and $H_k$ be the sets of active elements and
explored vertices in step $k$ respectively in the exploration procedure of the block of $x$.
\begin{itemize} 
\item In step $0$,  $x_1=x$ is said
to be active ($A_0=\{x_1\}$) and other elements are neutral. 
\item In step 1, every
neighbour of $x_1$ is declared active and  $x_1$ is said to be an explored
element: 
$A_{1}=\Ne_{x_1}(t,\enu)$ and  $H_1=\{x_1\}$. 
\item In step $k\geq 1$, let us assume
that  $A_{k-1}$ is not empty. Let $x_k$ denote  the smallest active element in
$A_{k-1}$. Neutral elements that are neighbours of $x_k$ are  added to $A_{k-1}$ 
and the status of  $x_k$ is changed: $A_{k}=A_{k-1}\cup
\Ne_{x_k}(t,\enu\setminus
H_{k-1})\setminus\{x_k\}$ and
${H_{k}=H_{k-1}\cup\{x_k\}}$. In particular, $|A_k| = |A_{k-1}| + \xi_{n,k}(t)
-1$ with \linebreak[4]${\xi_{n,k}(t) = |\Ne_{x_k}(t,\enu\setminus H_{k-1})\setminus
A_{k-1}|}$. 
\end{itemize}
The process stops in step $T_n(t)=\min(k,\; A_k=\emptyset)$.  By
construction, $$T_n(t)=\min(k,\; \sum_{i=1}^{k}\xi_{n,i}(t)\leq k-1).$$ 
The block of $x$ is $\Pin^{(x)}(t)=H_{T_n(t)}$ and its size is
$T_n(t)$. 
\begin{ex}
\label{exsample}
Let $n\geq 10$. Assume that $\mP_n(t)$  is formed by five tuples $(1,2,3,4)$, $(2,5,2,3)$, 
$(3,6,4)$,  $(6,7)$ and $(8,10)$. 
The steps of the exploration procedure starting from 1 are
\begin{itemize}
\item Step 1: $x_{1}=1$ and $A_{1}=\{2,3,4\}$ so that $\xi_{n,1}(t)=3$.
\item Step 2: $x_{2}=2$ and  $A_{2}=\{3,4,5\}$ so that $\xi_{n,2}(t)=1$.
\item Step 3: $x_{3}=3$ and  $A_{3}=\{4,5,6\}$ so that $\xi_{n,3}(t)=1$.
\item Step 4: $x_{4}=4$ and  $A_{4}=\{5,6\}$ so that $\xi_{n,4}(t)=0$.
\item Step 5: $x_{5}=5$ and $A_{5}=\{6\}$ so that $\xi_{n,5}(t)=0$.
\item Step 6: $x_{6}=6$ and  $A_{6}=\{7\}$ so that $\xi_{n,6}(t)=1$.
\item Step 7: $x_{7}=7$ and  $A_{7}=\emptyset$ so that $\xi_{n,7}(t)=0$, $T_n(t)=7$ and $\Pi^{(1)}_n(t)=\{1,2,3,4,5,6,7\}$. 
\end{itemize}
\end{ex}
\subsection{The BGW process associated with a block}
 The random variable $\xi_{n,k}(t)$ is bounded above by 
$$\zeta^{(1)}_{n,k}(t)=\sum_{w \in\mPs_{n,x_k}(t,\enu\setminus H_{k-1})}\!\!\!(\ld(w)-1)$$ 
in which a same element is counted as many times as it appears in  
${w \in \mPs_{n,x_k}(t,\enu\setminus H_{k-1})}$.  
To obtain identically distributed random variables in each step, 
we have to consider also in step $k$,  tuples that contain  $x_k$ and elements of $H_{k-1}$ 
before time $t$. Let denote this set of tuples  $\mP_{n,x_k,H_{k-1}}(t)$ and set     $\displaystyle{\zeta^{(2)}_{n,k}(t)=\sum_{w
\in\mP_{n,x_k,H_{k-1}}(t)}(\ld(w)-1)}$  and   
$$\zeta_{n,k}(t)=\zeta^{(1)}_{n,k}(t)+\zeta^{(2)}_{n,k}(t)=\sum_{w \in\mPs_{n,x_k}(t,\enu)}(\ld(w)-1).$$ 
The distribution of $\zeta_{n,k}(t)$   is 
the  
$\CPois(t\beta_{n},\nu_{n})$-distribution  with
$$\beta_{n}=\mu(\{w\in\mWs, x\in
w\})=\mu(\mWs)-\mu(\mWs[\enu\setminus \{x\}])
=1-\Gp(1-\frac{1}{n})-\Gp(\frac{1}{n}).
$$ and $\forall j\in\NN$, 
 $$\nu_{n}(j)=\frac{1}{\beta_{n}}\mu(\{w\in\mWs,\
x\in w \;\et\; \ld(w)=j+1\})=\frac{p(j+1)}{\beta_n}\Big(1-(1-\frac{1}{n})^{j+1}-(\frac{1}{n})^{j+1}\Big).$$
\begin{ex}
 In example \ref{exsample}, the random variables associated with the 
 first three steps of the exploration procedure of the block of 1 are $\zeta^{(1)}_{n,1}(t)=3$,  
 $\zeta^{(2)}_{n,1}(t)=0$, $\zeta^{(1)}_{n,2}(t)=3$,  $\zeta^{(2)}_{n,2}(t)=3$,  
 $\zeta^{(1)}_{n,3}(t)=2$ and $\zeta^{(2)}_{n,3}(t)=6$.
\end{ex}
Let $\mathcal{F}_k=\sigma(H_j, A_j,\; j\leq k)$. 
Let us note that the random variables $\zeta_{n,j}(t)$ and $\zeta_{n,k}(t)$ for $j<k$ 
are not independent since a same tuple can belong to $\mP_{n,x_k,H_{k-1}}(t)$ and 
$\mP_{n,x_j,H_{j-1}}(t)$. 
Nevertheless, since disjoint subsets of tuples in
$\mP_n(t)$ are independent, 
the random variables $\zeta^{(1)}_{n,j}(t)$ for $j\leq k$ are independent conditionally on 
$\mathcal{F}_{k-1}$, and the random variable $\zeta^{(1)}_{n,k}(t)$ is
independent of  $\zeta^{(2)}_{n,k}(t)$ conditionally on $\mathcal{F}_{k-1}$. 
 Therefore, by using independent copies of the Poisson point process $\mP_n$,  
 we can construct  a sequence of nonnegative random variables  
 $(\bar{\zeta}^{(2)}_{n,k}(t))_k$ such that:
\begin{itemize}
\item $\bar{\zeta}^{(2)}_{n,k}(t)$ has the same distribution as
$\zeta^{(2)}_{n,k}(t)$ and is independent of $\zeta^{(1)}_{n,k}(t)$ conditionally
on $\mathcal{F}_{k-1}$ for every $k\geq 2$; 
 \item $\bar{\zeta}_{n,k}(t)=\zeta^{(1)}_{n,k}(t)+\bar{\zeta}^{(2)}_{n,k}(t)$
are independent with distribution 
$\CPois(\beta_{n} t,\nu_{n})$ for every $k\in\NN^*$. 
\end{itemize}
Set $\Tn(t)=\min(k,\
\bar{\zeta}_{n,1}(t)+\ldots+\bar{\zeta}_{n,k}(t)=k-1)$. By construction,
$\Tn(t)\geq |\Pin^{(x)}(t)|$. 
If $\bar{\zeta}_{n,1}(t)$ is seen as the number of offspring of an individual
$I$ and $\bar{\zeta}_{n,k}(t)$ for $k\geq 2$ as the number of offspring of the
$k$-th individual explored by a breadth-first algorithm of the family tree of $I$,
then $\Tn(t)$ is the total number of individuals in the family tree of
$I$. We call  $(\bar{\zeta}_{n,k}(t))_k$  the associated BGW
process (a bijection between  BGW trees and  lattice  walks was 
described by T.~E.~Harris~\cite{Harris52} in Section 6, 
see also Section 6.2 in \cite{PitmanLN} for a review).
%
%
\section{Approximation of block sizes\label{sect:proofth}}
The number of neighbours of an element  is used to approximate the number of active elements 
added in each step of the exploration process of a block. We begin this section by studying 
its asymptotic distribution.  
Next, we prove Theorem \ref{thmblocksize} and Corollary \ref{corolblocksize}. Its 
proof is divided into two steps:  we  give an upper bound of 
the deviation between the cumulative distribution function of 
$|\Pin^{(x)}(t)|$ and of the total population size of the  associated BGW process 
and then we study the asymptotic distribution of the  BGW process associated with 
$|\Pin^{(x)}(nt)|$.
We end this section by a proof of Corollary \ref{coroltwoblocks}. 
In this section,  the third moment of the distribution $p$ is assumed to be finite. 
\subsection{Neighbours of an element}
Let $V_n$ be a subset of  $\enu$ and let $x\in\enu\setminus V_n$. 
The aim of this section is to show that the 
number of neighbours of $x$ in $\enu\setminus V_n$ at time $nt$ 
(denoted by $|\Ne_{x}(nt,\enu\setminus V_n)|$)
converges in law to the  
$\CPois(t\moms,\tilde{p})$-distribution  if $\frac{V_n}{n}$ tends to $0$. \\
The 
number of neighbours of $x$ in $\enu\setminus V_n$ at time $t$ is equal to 
$\sum_{w\in \mPs_{n,x}(t,\enu\setminus V_n)}(\ld(w)-1)$ except if there exists a tuple 
in $\mPs_{n,x}(t,\enu\setminus V_n)$ which has several copies of a same element  
or if there is an element  $y\neq x$ which appears in several tuples 
of ${\mPs_{x}(t,\enu\setminus V_n)}$. The following lemma yields an upper
bound for the probability that such an event occurs:
\begin{lem}
\label{lemrecoupe}
Let $x\in\enu$. Set  $F_{n,t}$ be the event  `\emph{some tuples in
$\mPs_{x}(t,\enu)$ contain several copies of a same element or have in common 
other elements than $x$.}'   
\[
\Pd(F_{n,t})\leq \frac{t}{2n^2}\left(\mom[2]+\mom[3]+\frac{t}{n}(\mom[2])^2\right).
\]
\end{lem}
\begin{proof}
 We study separately the following two events:
\begin{itemize}
 \item 
 $F^{(1)}_{n,t}$ :`\emph{there exists  $y\neq x$ which is in several tuples of $\mPs_{n,x}(t)$ 
 or several times in one tuple of $\mPs_{n,x}(t)$}' 
\item $F^{(2)}_{n,t}$: `\emph{some tuples of $\mPs_{n,x}(t)$ contain several copies of $x$}'.
\end{itemize}
To compute  $\Pd(F^{(1)}_{n,t})$, we introduce the random variable $S_{t,x}$
as the total length of tuples in $\mPs_{n,x}(t)$ minus the number of copies 
of $x$ in tuples of $\mPs_{n,x}(t)$: 
$S_{t,x}=\sum_{w\in \mPs_{n,x}}\ell_x(w)$ where  $\ell_x(w)$
denotes the number of elements different from $x$ in the tuple $w$. Since 
elements that form a tuple are chosen independently with the uniform distribution
on $\enu$,
\[\Pd(F^{(1)}_{n,t})=1-\Ed\left(\prod_{i=0}^{S_{t,x}-1}(1-\frac{i}{n-1})\right)\leq
\frac{1}{2(n-1)}\Ed(S_{t,x}(S_{t,x}-1)).\] 
 By Campbell's formula, the probability-generating function of $S_{t,x}$ is
\[\Ed( u^{S_{t,x}})=\exp\Big(\sum_{w\in \mWs,\; x\in w} (u^
{\ld_x(w)}-1)t\mu_n(w)\Big).\]  
By decomposing  $\displaystyle{f_n(u)=\sum_{w\in \mWs,\; x\in w} (u^{\ld_x(w)}-1)\mu_n(w)}$ 
according to the size of a tuple and the number of copies of $x$ 
in it and then applying the binomial formula,  we obtain: 
\begin{alignat*}{2}f_n(u)=&\sum_{j=1}^{+\infty}p(j)\sum_{i=1}^{j-1}(u^{j-i}-1)\binom{j}{i
}\left(\frac{1}{n}\right)^{i}\left(1-\frac{1}{n}\right)^{j-i}\\
=&\sum_{j=1}^{+\infty}\frac{p(j)}{n^j}
\Big((u(n-1)+1)^j-u^j(n-1)^j-n^j+(n-1)^j\Big)\\
=&\Gp\left(\frac{1}{n}+u(1-\frac{1}{n})\right)-\Gp\left(u(1-\frac{1}{n})\right)-1+\Gp\left(1-\frac{1}{n}\right).
\end{alignat*}
We deduce the following formula of $\Ed(S_{t,x}(S_{t,x}-1))$
by  computing the first two derivatives of $\Ed( u^{S_{t,x}})$:
\[\Ed(S_{t,x}(S_{t,x}-1))=(1-\frac{1}{n})^2\left(t\big(\Gp^{(2)}(1)-\Gp^{(2)}(1-\frac{1}{n})\big)+
t^2\big(\Gp^{(1)}(1)-\Gp^{(1)}(1-\frac{1}{n})\big)^2\right).
\] 
As the third moment of $p$ is finite,  
\[\Gp^{(1)}(1)-\Gp^{(1)}(1-\frac{1}{n})\leq \frac{\mom[2]}{n} \et  
\Gp^{(2)}(1)-\Gp^{(2)}(1-\frac{1}{n})\leq \frac{\mom[3]}{n}.
\] Thus we obtain: 
$$\Pd(F^{(1)}_{n,t,k})\leq \frac{t}{2n^2}(\mom[3]+\frac{t}{n}\mom[2]^2). $$
To study $F^{(2)}_{n,t}$, let  $N_x(w)$ denote the number of copies of $x$ in a tuple
$w\in\mWs$:
 $$\Pd(F^{(2)}_{n,t})=1-\exp\big(-t\mu_n(w\in \mWs,\ N_x(w)\geq 2)\big).$$ 
We have already seen in Proposition \ref{nbsingleton} that   
$$\mu_n(w\in \mW,\ N_x(w)\geq 1)=1-\Gp(1-\frac{1}{n})-\Gp(\frac{1}{n}).$$
Finally, 
\[\mu_n(w\in \mWs,\ N_x(w)= 1)=
\sum_{k=2}^{+\infty}p(k)\frac{k}{n}(1-\frac{1}{n})^{k-1}=
\frac{1}{n}\Big(\Gp^{(1)}(1-\frac{1}{n})-p(1)\Big).
\]
Therefore, 
\[\mu_n(w\in \mWs,\ N_x(w)\geq 2)=
1-\Gp(1-\frac{1}{n})-\frac{1}{n}\Gp^{(1)}(1-\frac{1}{n})-\Gp^{(1)}(\frac{1}{n})+\frac1n p(1)
\leq \frac{1}{2n^2} \mom[2]. \]
In summary, 
$\Pd(F^{(2)}_{n,t})\leq t\mu_n(w\in \mWs,\ N_x(w)\geq 2)\leq \frac{t}{2n^2} \mom[2].$
\end{proof}
Let us now describe the distribution of the upper bound we have obtained for the number 
of neighbours of $x$ in $\enu\setminus V_n$ at time $nt$ and the total variation distance (denoted 
by $\TV$) between it and the compound Poisson distribution $\CPois(t\moms,\pb)$:
\begin{prop}
\label{TVlengthneig}
 For a subset $V$ of  $\enu$ and $x\in\enu\setminus V$, set 
 $$S_{t,V,x}=\sum_{w\in\mPs_{x}(t,\enu\setminus V)}(\ld(w)-1).$$ 
 \begin{itemize}
 \item[(i)] The random variable $S_{nt,V,x}$ 
  has the compound Poisson distribution  $\CPois(nt\beta_{V},\nu_{n,V})$ where:
  \begin{alignat*}{2}&\beta_{n,V}=\Gp\left(1-\frac{|V|}{n}\right)-\Gp\left(1-\frac{|V|+1}{n}\right)-\Gp\left(\frac{1}{n}\right)\\
  &\nu_{n,V}(j)=\frac{p(j+1)}{\beta_{n,V}}
  \left(\left(1-\frac{|V|}{n}\right)^{j+1}-\left(1-\frac{|V|+1}{n}\right)^{j+1}-\left(\frac{1}{n}\right)^{j+1}\right)\;\forall j\in\NN. 
\end{alignat*}
  \item[(ii)]  
 $\TV\left(\CPois(nt\beta_{n,V},\nu_{n,V}),\CPois(t\moms,\pb)\right)
 \leq 2t\mom[2]\left(\frac{|V|}{n}+\frac{1}{2n}\right) + \frac{t}{2n}.$ 
 \end{itemize}
 \end{prop}
\begin{proof}
\begin{itemize}
\item[(i)] By definition of the Poisson tuple set, $S_{nt,V,x}$
has the compound Poisson distribution $\CPois(nt\beta_{V},\nu_{n,V})$ where
$\beta_{n,V}=\mu_n(w\in\mWs[\enu\setminus V],\ x\in w)$ and for every 
$j\in\NN^*$,
$$\nu_{n,V}(j)=\frac{1}{\beta_{n,V}}\mu_n(w\in\mWs[\enu\setminus V],\ x\in w \et \ld(w)=j+1).$$ 
\item[(ii)]
The total variation distance between two compound Poisson distributions can be bounded as follows
using coupling arguments:
\begin{lem}
\label{lem_couplingCP}
 Let $p_1$ and $p_2$ be two probability measures on $\NN$ and let $\lambda_1$
and $\lambda_2$ be two positive reals such that $\lambda_1\leq \lambda_2$. Then 
$$\TV(\CPois(\lambda_1,p_1),\CPois(\lambda_2,p_2))\leq
1-e^{-(\lambda_2-\lambda_1)}+\lambda_1\TV(p_1,p_2).$$
\end{lem}
\begin{proof}[Proof of Lemma \ref{lem_couplingCP}]
By Strassen's theorem, there exist two independent sequences $(X_i)_{i\in
\NN^*}$ and $(Y_i)_{i\in \NN^*}$  of \iid random variables  with distributions $p_1$ and
$p_2$ respectively such that ${\TV(p_1,p_2)=\Pd(X_i\neq Y_i)}$ for every $i\in\NN\*$. Let
$Z_1$ and $Z_2$ be two independent Poisson-distributed random variables with
parameters $\lambda_1$ and $\lambda_2-\lambda_1$ respectively, 
 which  are independent of the two sequences $(X_i)_i$ and $(Y_i)_i$ (we take  $Z_2=0$ if $\lambda_2=\lambda_1$). 
 Set $Z=Z_1+Z_2$. Then
\[
\Pd\left(\sum_{i=1}^{Z_1}X_i\neq \sum_{i=1}^{Z}Y_i\right)\leq
\Pd(Z_2>0)+\Pd\left(\sum_{i=1}^{Z_1}X_i\neq \sum_{i=1}^{Z_1}Y_i\right)
\] and
\[
\Pd\left(\sum_{i=1}^{Z_1}X_i\neq \sum_{i=1}^{Z_1}Y_i\right)\leq
\sum_{k=0}^{+\infty}\Pd(Z_1=k)\sum_{i=1}^{k}\Pd(X_i\neq Y_i)=
\Ed(Z_1)\TV(p_1,p_2).
\]
\end{proof}
We apply Lemma \ref{lem_couplingCP} with $\lambda_1=tn\beta_{n,V}$, 
$\lambda_2=t\moms$, $p_1=\nu_{n,V}$ and
$p_2=\pb$ and use the following inequalities with $u=\frac{1}{n}$: $\forall j\in\NN^*$, $\forall x,u\geq 0$ such that  $x+u\leq 1$, 
\begin{equation}
\label{ineqpuis}
ju-j(j-1)(xu+\frac{u^2}{2})\leq (1-x)^j-(1-x-u)^j\leq ju.
\end{equation}
We obtain 
$0\leq \moms-n\beta_{n,V}\leq \mom[2]\Big(\frac{|V|}{n}+\frac{1}{2n}\Big)$ and for every $j\in\NN^*$, 
\begin{multline} n\beta_{n,V}|\nu_{n,V}(j)-\pb(j)|\leq \\ p(j+1)\Big((j+1)(1-\frac{n\beta_{n,V}}{\moms})
 +j(j+1)\big(\frac{|V|}{n}+\frac{1}{2n}\big)+\frac{1}{n^{j}}\Big).
 \end{multline}
Therefore
\begin{multline} n\beta_{n,V}\TV(\nu_{n,V},\pb) \leq \frac{1}{2}\mom[2]\big(\frac{|V|}{n}+\frac{1}{2n}\big)
+\frac{1}{2}(\mom-n\beta_{n,V})+\frac{1}{2n}\\\leq\frac{1}{2n}\big(\mom[2](2|V|+1)+1\big)
\end{multline}
and 
\begin{alignat*}{2}
\TV(\CPois(\lambda_1,p_1),\CPois(\lambda_2,p_2))\leq &
1-e^{-t(\mom-n\beta_{n,V})}+tn\beta_{n,V}\TV(\nu_{n,V},\pb)\\
\leq & 2t\mom[2]\Big(\frac{|V|}{n}+\frac{1}{2n}\Big)+\frac{t}{2n}.
\end{alignat*}
\end{itemize}
\end{proof}
In summary, Lemma \ref{lemrecoupe} and Proposition \ref{TVlengthneig} yield 
the following result for the number of neighbours of an element:
\begin{prop}
\label{convvoisins}
 For every  $x\in \enu$ and  $V\subset \enu\setminus\{x\}$, the total variation
  distance between the distribution of 
 $|\Ne_{x}(nt,\enu\setminus V)|$ and the 
 $\CPois(t\moms,\pb)$ distribution 
 is smaller than 
 \[
 2t\mom[2]\Big(\frac{|V|}{n}+\frac{1}{2n}\Big)+ 
 \frac{t}{2n}\big(1+\mom[2]+\mom[3]+t(\mom[2])^2\big).
 \]
\end{prop}
\subsection{Comparison between a block size and the associated BGW
process}
The aim of this section is to prove  that  small block sizes  at time $nt$ are well approximated 
by $\bar{T}_{n}(nt)$ which has the same distribution as the total population size of a 
\BGW$(1,nt\beta_{n},\nu_{n})$ process (first step of the proof of Theorem \ref{thmblocksize}):
\begin{prop}\label{propcomp1}
Let $x\in\enu$. For every 
$k,n\in\NN$ and $t\geq 0$,  
 $$|\Pd(|\Pin^{(x)}(nt)|\leq k)-\Pd(\bar{T}_{n}(nt)\leq k)|\leq 
 \frac{kt}{2n}\big(\mom[2]^2(k-1+t)+\mom[2](k+2)+\mom[3]\big).$$
\end{prop}
Let us recall that the number of new active elements added in the $j$-th step of the 
exploration procedure at time $t$ is $\xi_{n,j}(t)=|\Ne_{x_j}(t,\enu\setminus H_{j-1})\setminus A_{j-1}|$ 
where $A_{j-1}$ and $H_{j}=\{x_1,\ldots,x_{j-1}\}$ are respectively  the set of active 
elements and  explored  elements in step $j-1$. We have already seen one source of difference  
between  $\xi_{n,j}(t)$ and ${\zeta_{n,j}(t)=\sum_{w\in\mPs_{n,x_j}(t)}(\ld(w)-1)}$.
It is described by the  event 
\begin{quote}
$F_{n,t,j}$:  `some tuples in $\mPs_{n,x_j}(t,\enu\setminus H_{j-1})$ contain several 
copies of a same element or have in common other elements than $x_j$'.
\end{quote}
By Lemma \ref{lemrecoupe}, the probability of this event is bounded by: 
$\frac{t}{2n^2}(\mom[2]+\mom[3]+\frac{t}{n}\mom[2]^2)$. \\
There are two other sources of difference described by the following events:
\begin{itemize}
\item $\{\bar{\zeta}^{(2)}_{n,j}(t)>0$\}:  
`there exists a tuple  containing $x_j$ and  already explored elements 
(that is elements of $H_{j-1}$)',
\item\label{eventFntj} $K_{n,t,j}$:  `there exists a tuple   in $\mP_{n,x_j}(t,\enu\setminus
H_{j-1})$ (i.e. containing $x_j$ but no element of $H_{j-1}$) 
which contains active elements (i.e. elements of $A_{j-1}$)',
\end{itemize}
The probability of these two events can be bounded by using the following lemma:
\begin{lem}
\label{lemintersect}
 Let $V$ be a subset of $\enu$ and let $x\in\enu\setminus V$. 
 For every $t>0$, 
\[\Pd(\exists w \in\mPs_{n,x}(t),\  w\cap V\neq \emptyset)\leq 
1-\exp(-\frac{t|V|}{n^2}\mom[2])
\]
\end{lem}
\begin{proof}
Let $K_{x,V}$ be the subset of tuples $w\in\mWs$ which contain $x$ and some elements of $V$. 
$$\Pd(\exists w\in \mP_{n,x}(t), w\cap V\neq \emptyset) = 1-\exp(-t\mu(K_{x,V}))$$
and 
\begin{alignat*}{2}\mu(K_{x,V})&=\mu(\mP_{n,x})-\mu(\mP_{n,x}(\enu\setminus V))\\ 
&=1-\Gp(1-\frac{1}{n})-\big(\Gp(1-\frac{|V|}{n})-\Gp(1-\frac{|V|+1}{n})\big)\\
&\leq \frac{|V|}{n^2}\mom[2], 
\end{alignat*}
where the last upper bound is  a consequence of the following inequality:
$$1-(1-au)^j-(1-bu)^j+(1-(a+b)u)^j\leq j(j-1)abu^2\;
\forall j\in\NN^*,\;\forall  a,b\in\RR_+ \;\et\;\forall  u\in[0,\frac{1}{a+b}[.$$  
\end{proof}
With the help of these estimates, we prove Proposition \ref{propcomp1}. 
\begin{proof}[\textbf{Proof of Proposition \ref{propcomp1}}]
 Set $\Delta_k=|\Pd(|\Pin^{(x)}(nt)|\leq k)-\Pd(\bar{T}_{n}(nt)\leq k)|$. \\
Since ${|\Pin^{(x)}(nt)|\leq \bar{T}_{n}(nt)}$, $\Delta_k=\Pd(|\Pin^{(x)}(nt)|\leq k\; \text{and}\;
\bar{T}^{(n)}_{nt}> k).$ It is bounded above by 
\begin{multline*}\Pd(|\Pin^{(x)}(nt)|\leq k\;
\text{and}\; \exists j\leq |\Pin^{(x)}(nt)|,\
\xi_{n,j}(nt)<\bar{\zeta}_{n,j}(nt))\\ \leq
\sum_{j=1}^{k}\Ed(\un_{\{|\Pin^{(x)}(nt)|\geq
j\}}\Pd(\xi_{n,j}(nt)<\bar{\zeta}_{n,j}(nt)|\mathcal{F}_{j-1})).
\end{multline*}
We have seen  that $$\Pd(\xi_{n,j}(nt)<\bar{\zeta}_{n,j}(nt)|\mathcal{F}_{j-1})\leq 
\Pd(\zeta^{(2)}_{n,j}(nt)>0|\mathcal{F}_{j-1})+\Pd(K_{n,tn,j}|\mathcal{F}_{j-1})+
\Pd(F_{n,tn,j}|\mathcal{F}_{j-1})$$ 
with the notations introduced page \pageref{eventFntj}. 
By Lemma  \ref{lemintersect}
\[\Pd(K_{n,tn,j}|\mathcal{F}_{j-1})\leq
\frac{t|A_{j-1}|}{n}\mom[2]\;\text{and}\;\Pd(\zeta^{(n,2)}_{nt,j}>0|\mathcal{F}_{j-1})\leq
\frac{t(j-1)}{n}\mom[2]
\]
and by Lemma \ref{lemrecoupe} 
\[
\Pd(F_{n,tn,j}|\mathcal{F}_{j-1})\leq
\frac{t}{2n}(\mom[2]+\mom[3]+t\mom[2]^2). 
\]
Therefore, 
\begin{multline*}
\Delta_k \leq 
\frac{t}{n}\mom[2]\sum_{j=1}^{k}\Ed(|A_{j-1}|\un_{\{|\Pin^{(x)}(nt)|\geq j\}})\\ +
\frac{t}{2n}\sum_{j=1}^{k}\Pd(|\Pin^{(x)}(nt)|\geq j)(\mom[3]+(2j-1)\mom[2]+t\mom[2]^2).
\end{multline*}

By construction $|A_{j-1}|-1=\sum_{i=1}^{j-1}(\xi_{n,i}(nt)-1)$.  Let us
recall that  $\xi_{n,i}(nt)$ has nonnegative integer values, it is bounded
above by $\bar{\zeta}_{n,i}(nt)$ and the conditional law of
$\bar{\zeta}_{n,i}(nt)$ given $\mathcal{F}_{i-1}$ is equal to the law of
$\zeta_{n,1}(nt)$. Thus, 
\[
\Ed(\un_{\{|\Pin^{(x)}(nt)|\geq
j\}}(|A_{j-1}|-1))\leq \sum_{i=1}^{j-1}\Ed(\un_{\{|\Pin^{(x)}(nt)|\geq
i\}}\Ed(\bar{\zeta}_{n,i}(nt)|\mathcal{F}_{i-1}))
\leq (j-1)\Ed(\zeta_{n,1}(nt)).
\]
with 
$\Ed(\zeta_{n,1}(nt))=tn\sum_{j=1}^{+\infty}jp(j+1)\big(1-(1-\frac{1}{n})^{j+1}\big)\leq t \mom[2]$. 
Therefore, 
\begin{multline*}\Delta_k \leq 
\frac{t}{n}\mom[2]^2\sum_{j=1}^{k}(j-1)+
\frac{t}{2n}\sum_{j=1}^{k}(\mom[3]+(2j-1)\mom[2]+t\mom[2]^2)\\ =
\frac{kt}{2n}(\mom[2]^2(k-1+t)+\mom[2](k+2)+\mom[3]).
\end{multline*}
\end{proof}
\subsection{The total progeny of the BGW process associated with a block}
 Recall that the offspring distribution of the BGW process associated with a 
 block at time $nt$ is the  $\CPois(tn\beta_{n},\nu_{n})$-distribution with:
\[\beta_{n}=\mu(\{w\in\mWs, x\in
w\})
=1-\Gp(1-\frac{1}{n})-\Gp(\frac{1}{n})\; \et
\]
\begin{multline*}\nu_{n}(j)=\frac{1}{\beta_{n}}\mu(\{w\in\mWs,\
x\in w \;\et\; \ld(w)=j+1\})\\ 
=\frac{p(j+1)}{\beta_n}\Big(1-(1-\frac{1}{n})^{j+1}-(\frac{1}{n})^{j+1}\Big) \quad 
\forall j\in\NN^*.
\end{multline*}
We have shown (Proposition \ref{TVlengthneig}) that the 
 $\CPois(tn\beta_{n},\nu_{n})$-distribution is close to 
the $\CPois(t\moms,\pb)$-distribution for large $n$.
We now consider the distribution of  the total number of individuals in a
BGW process with one ancestor and  offspring distribution
$\CPois(tn\beta_{n}, \nu_{n})$. 
Let us state a general result for the comparison of the total number 
of individuals in two BGW processes: 
\begin{lem}
\label{lem_couplingT}
 Let $\nu_1$ and $\nu_2$ be two probability distributions on $\NN$. 
 Let $\TV$ denote the total variation distance between probability measures. Let $T_1$
and $T_2$ be the total population sizes of the BGW processes with one ancestor and
offspring distributions $\nu_1$ and $\nu_2$ respectively. \\
For every $k\in \NN^*$, $|\Pd(T_1\geq k)-\Pd(T_2\geq k)|\leq
\TV(\nu_1,\nu_2)\sum_{i=1}^{k-1}\Pd(T_2\geq i)$. 
\end{lem}
\begin{proof}
We follow the proof of Theorem 3.20  in \cite{book1vanderHofstad} which  states
an analogous result between binomial and Poisson BGW processes.
The proof is based on the description of the total population size by means of 
the hitting time of a random walk  and  coupling arguments. 
 By Strassen's theorem, there exist two independent sequences $(X_i)_{i\in
\NN^*}$ and $(Y_i)_{i\in \NN^*}$  of \iid random variables with distributions $\nu_1$
and $\nu_2$ respectively such that ${\TV(\nu_1,\nu_2)=P(X_i\neq Y_i)}$ for every
$i\in\NN\*$. Let $\tau_1=\min(n,\ X_1+ \ldots +X_n=n-1)$ and $\tau_2=\min(n,\
Y_1+ \ldots +Y_n=n-1)$.  $\tau_1$ and $\tau_2$ have the same laws as $T_1$ and
$T_2$ respectively. Let $k\in\NN^*$. 
 \[|\Pd(T_1\geq k)-\Pd(T_2\geq k)|\leq 
 \max\big(\Pd(\tau_1\geq k\; \text{and}\; \tau_2<k),\Pd(\tau_1< k\; \text{and}\; \tau_2\geq k)\big).
 \] 
First, let us note that  $$\{\tau_1\geq k\; \text{and}\; \tau_2<k\} \subset \bigcup_{i=1}^{k-1}\{X_j=Y_j\;\forall j\leq i-1, \; X_i\neq Y_i \; \text{and}\;
\tau_1\geq k\}$$ 
As $\{X_j=Y_j\;\forall j\leq i-1\; \text{and}\; \tau_1\geq k\}\subset
\{\tau_2\geq i\}$ for $i\leq k-1$ and $\{\tau_2\geq i\}$ depends only on $Y_1,\ldots,Y_{i-1}$,
we obtain:
$$\Pd(\tau_1\geq k\; \text{and}\; \tau_2<k)\leq \sum_{i=1}^{k-1}\Pd(\tau_2\geq
i)\Pd(X_i\neq Y_i)=\TV(\nu_1,\nu_2)\sum_{i=1}^{k-1}\Pd(\tau_2\geq i). $$
The same upper bound holds for  $\Pd(\tau_1<k\;\text{and}\;\tau_2\geq k)$ since 
$$\{\tau_1<k\;\text{and}\;\tau_2\geq k)\}\subset  \bigcup_{i=1}^{k-1}\{X_j=Y_j\;\forall j\leq i-1, \; X_i\neq Y_i \; \text{and}\; \tau_2\geq k\}.$$ 
and $\{\tau_2\geq k\}\subset \{\tau_2\geq i\}$ for $i\leq k$. 
\end{proof}
From Lemma \ref{lem_couplingT} and Proposition \ref{TVlengthneig}, we obtain: 
\begin{prop}
\label{propcouplingtotalpopul}
Let $t>0$ and $n\in \NN^*$. Let $\Tn(t)$ and $\Tp(t)$ denote  the total number of individuals in a
\BGW$(1,t\beta_{n}, \nu_{n})$ and 
\BGW$(1,t\moms,\pb)$ processes respectively. 
\[
|\Pd(\Tn(nt)\geq k)-\Pd(\Tp(t)\geq
k)|\leq \frac{t}{2n}(2\mom[2]+1)\sum_{i=1}^{k-1}\Pd(\Tp(t)\geq i)\quad\text{for every}\; k\in\NN^*.
\] 
\end{prop}
\paragraph{Proof of Theorem \ref{thmblocksize}. }
Theorem \ref{thmblocksize} follows from Propositions 
\ref{propcomp1} and \ref{propcouplingtotalpopul}: 
\begin{eqnarray*}|\Pd(|\Pin^{(x)}(nt)|\leq k)-\Pd(\Tp(t)\leq k)|&\leq &
|\Pd(|\Pin^{(x)}(nt)|\leq k)-\Pd(\Tn(nt)\leq k)|\\
&& + |\Pd(\Tn(nt)\leq k)-\Pd(\Tp(t)\leq k)|\\
&\leq & \frac{kt}{2n}\big(\mom[2]^2(k-1+t)+
\mom[2](k+4)+\mom[3]+1\big).
\end{eqnarray*}\qed
\paragraph{Proof of Corollary \ref{corolblocksize}. }
To deduce Corollary \ref{corolblocksize}, we apply the above inequality 
to 
 $|\Pd(|\Pin[\enu,p_n]^{(x)}(nt_n)|\leq k)-\Pd(T^{(1)}_{p_n}(t_n)\leq k)|$. 
By Lemma \ref{lem_couplingT} and Lemma \ref{lem_couplingCP}: 
$$|\Pd(T^{(1)}_{p_n}(t_n)\leq k)-\Pd(T^{(1)}_{p}(t)\leq k)|\leq k\Big(|t_n m^{*}_{p_n,1}-t\moms|+
\max(t_n m^{*}_{p_n,1},t\moms)\TV(\tilde{p}_n,\pb) \Big).$$
Under the hypotheses of Corollary \ref{corolblocksize},  $(m_{p_n,i})_n$ 
for $i\in\{1,2,3\}$ are bounded, \linebreak[4]  ${|t_n m^{*}_{p_n,1}-t\moms|=O(\frac{1}{n})}$ 
and $\TV(\tilde{p}_n,\pb)=O(\frac{1}{n})$. 
Therefore, there exists $C(t)>0$ such that for every $k\in\NN^*$, 
\[
|\Pd(|\Pin^{(x)}(nt_n)|\leq k)-\Pd(\Tp(t)\leq k)|\leq \frac{C(t)k^2}{n}.
\]\qed
\paragraph{Proof of Corollary \ref{corolrestrictblocksize}.}
Let us now consider $\Pin[\llbracket a_n \rrbracket, p_n]^{(1)}(nt\Gp(\frac{a_n}{n}))$,
 where $a_n=\lfloor an\rfloor$ and $p_n$ is the probability distribution on $\NN^*$ defined by:
\[
p_n(k)= \left(\frac{a_n}{n}\right)^k p(k)\frac{1}{\Gp(\frac{a_n}{n})}\; \forall k\in\NN^*.
\] 
Set $t_n=t\frac{n}{a_n}\Gp(\frac{a_n}{n})$ for $n\in\NN^*$. 
To prove that there exists $C_a(t)>0$ such that for every $k,n\in \NN^*$,
\[\left|\Pd\left(|\Pin[\llbracket a_n \rrbracket, p_n]^{(1)}(nt\Gp(\frac{a_n}{n}))|\leq k\right)
-\Pd\left(T^{(1)}_{\hat{p}_a}(t\frac{\Gp(a)}{a}))\leq k\right)\right|\leq C_a(t) \frac{k^{2}}{n},
\] 
it suffices to verify that Corollary \ref{corolblocksize} applies to the sequences $(t_n)_n$ and $(p_n)_n$:
\begin{enumerate}
\item Since $|\frac{a_n}{n}-a|\leq \frac{1}{n}$ and  $\Gp'$ is bounded on 
$[0,a]$,  $t_n-t\frac{\Gp(a)}{a}=O(\frac{1}{n})$. 
\item The third moment of $p_n$ is bounded since 
$\sum_{k=1}^{+\infty}k^3 p_n(k)\leq\frac{1}{\Gp(a_n)}\sum_{k=1}^{+\infty}k^3a^kp(k)$ for every $n\in\NN^*$. 
\item The difference  $\Delta_n:=t_n m_{p_n,1}^{*}-t\frac{\Gp(a)}{a}m_{\hat{p}_a,1}^{*}$ can be split into the sum of two terms: 
\begin{eqnarray*}
\Delta_{n,1}&=&t\left(\frac{n}{a_n}-\frac{1}{n}\right)\sum_{k\geq 2}k(\frac{a_n}{n})^kp(k)=O(\frac{1}{n}),\\ 
\Delta_{n,2}&=&\frac{1}{a}\sum_{k\geq 2}((\frac{a_n}{n})^k-a^k)kp(k).
\end{eqnarray*}
By applying  the following inequality 
\begin{equation}\label{powerdiff}
|x^k-y^k|\leq k|x-y|\max(|x|,|y|)^{k-1}\; \forall x,y\in\RR, 
\end{equation}
we obtain $|\Delta_{n,2}|\leq \frac{1}{an}\sum_{k\geq 2}k^2a^{k-1}p(k)=O(\frac{1}{n})$. 
\item The last assumption of Corollary \ref{corolblocksize} concerns the total variation distance between 
the probability distributions $\tilde{p}_n$ and $\widetilde{(\hat{p}_a)}$
defined by: 
 $$\tilde{p}_n(k)=\frac{1}{S(\frac{a_n}{n})}(k+1)(\frac{a_n}{n})^k p(k+1)  \et 
 \widetilde{(\hat{p}_a)}(k)=\frac{1}{S(a)}(k+1)a^kp(k+1)\; \forall k\in\NN^*,$$
 where $S(x)=\sum_{j\geq 2}jx^j p(j)$ for $x\in [0,1[$. \\
Let us note that $d_{\TV}(\tilde{p}_n,\widetilde{(\hat{p}_a)})\leq \frac{1}{S(a)}\sum_{k\geq 2}kp(k)|(\frac{a_n}{n})^k-a^k|$. 
Therefore, by inequality \eqref{powerdiff}, $\TV(\tilde{p}_n,\widetilde{(\hat{p}_a)})=O(\frac{1}{n})$. 
\end{enumerate}
In conclusion, the four assumptions of Corollary \ref{corolblocksize} are satisfied. 

Relation \eqref{identtotalpop} between the  probability mass functions  of 
$T^{(u)}_{\hat{p}_a}(\frac{\Gp(a)}{a}t)$ and $\Tp[u](t)$  can be easily proven  
 by applying formula \eqref{lawTGW} for the probability mass function of $\Tp[u](t)$ 
(see Appendix \ref{annex:totalprogenydistr} for a proof of \eqref{lawTGW}) and by 
expressing  the probability mass function of $\CPois(\lambda G_{\tilde{p}}(a),\widehat{(\tilde{p})}_a)$  in terms of 
the probability mass function of $\CPois(\lambda,\tilde{p})$ (see Lemma \ref{annex:changeofmeasure}). \qed
\subsection{Asymptotic distribution of two block sizes}
Let us prove Corollary \ref{coroltwoblocks} stating that under the assumptions 
of Theorem \ref{thmblocksize},  the block sizes of two elements converge in law to the total population sizes of 
two independent \BGW$(1,t\moms,\tilde{p})$ processes. 
\paragraph{Proof of Corollary \ref{coroltwoblocks}.}
 The proof is similar to the proof presented in \cite{Bertoinbookfrag} in order to study the joint limit of the 
component sizes of two vertices in  the Erd\"os-R\'enyi random graph 
process. It is based on the properties  \ref{itm:proprI} and 
\ref{itm:proprII} stated in Subsection \ref{subsect:restriction}. \\ 
Let $x$ and $y$ be two distinct vertices and let $j,k$ be two nonnegative
integers.  
 We have to study the convergence of $\Pd(|\Pin^{(x)}(nt)|=j \et 
|\Pin^{(y)}(nt)|=k)$. 
First, let us note that by \ref{itm:proprI}, for every $n\geq j$, $\Pd(y\in
\Pin^{(x)}(nt) \mid |\Pin^{(x)}(nt)|=j)= \frac{j-1}{n-1}$.
Therefore, ${\Pd(y\in \Pin^{(x)}(nt) \et |\Pin^{(x)}(nt)|=j)}$
converges to $0$ as $n$ tends to $+\infty$. \\
It remains to study $\Pd(y\not \in \Pin^{(x)}(nt) \et |\Pin^{(x)}(nt)|=j \et 
|\Pin^{(y)}(nt)|=k)$ which can be written: 
\[
 \Pd(|\Pin^{(y)}(nt)|=k \mid y\not \in \Pin^{(x)}(nt) \et
|\Pin^{(x)}(nt)|=j) \Pd(y\not\in \Pin^{(x)}(nt) \et |\Pin^{(x)}(nt)|=j).
\]
Since $\Pd(y\not\in \Pin^{(x)}(nt) \et |\Pin^{(x)}(nt)|=j)=(1-\frac{j-1}{n-1})\Pd(|\Pin^{(x)}(nt)|=j)$, 
it  converges to $\Pd(\Tp(t)=j)$ by Theorem \ref{thmblocksize}. \\
By \ref{itm:proprII},  
\[
\Pd(|\Pin^{(y)}(nt)|=k \mid y\not \in \Pin^{(x)}(nt) \et
|\Pin^{(x)}(nt)|=j)=\Pd(|\Pi_{\enu[n-j],p_{n,j}}^{(y)}((n-j)t_{n,j})|=k )
\]
 where $t_{n,j}=\frac{tn}{n-j}\Gp(1-\frac{j}{n})$ and $p_{n,j}=p_{\enu|\enu[n-j]}$ 
(\textit{i.e. }$p_{n,j}(k)=(1-\frac{j}{n})^k\frac{p(k)}{\Gp(1-\frac{j}{n})}$ for every $k\in\NN^*$). 
   
Let us verify that Corollary \ref{corolblocksize} can be applied to the sequences 
$(t_{n,j})_n$  and  $(p_{n,j})_n$.
\begin{itemize} 
\item First, $(t_{n,j})_n$ converges to $t$, $(p_{n,j})_n$ converges weakly 
to $p$,  and  $(m_{p_{n,j},3})_n$ converges to $\mom[3]$.  
\item By inequality \eqref{ineqpuis},  $0\leq t\moms-t_{n,j}m^{*}_{p_{n,j},1}\leq \frac{t j}{n}\mom[2]$. 
\item Finally, let us show that $\TV(\tilde{p}_{n,j},\pb)=O(\frac{1}{n})$.
 For $k\in\NN$, $$\pb(k)-\tilde{p}_{n,j}(k)=(k+1)p(k+1)\frac{V_n(k)}{\mom[1]S_n}$$ with 
$$V_n(k)=\sum_{\ell\geq 1}(\ell+1)p(\ell+1)\Big((1-\frac{j}{n})^{\ell+1} - (1-\frac{j}{n})^{k+1}\Big)\;\text{and}\; 
S_n=\sum_{\ell\geq 1}(\ell+1)p(\ell+1)(1-\frac{j}{n})^{\ell+1}.$$  
Using that the first $k$ terms in  $V_n(k)$ are positive and the others are  
nonpositive, we obtain 
$|V_n(k)|\leq \frac{j}{n}\max(k\moms,\mom[2])$ for every $k\in\NN$.\\ 
As $(1-x)^{\ell}\geq 1-\ell x$ for every $x\geq 0$ and $\ell\in\NN^*$, 
$S_n\geq  \moms-\frac{j}{n}(\moms+\mom[2])$. Therefore, 
$\TV(\tilde{p}_{n,j},\pb)\leq \frac{j}{n}\mom[2]\big(\moms-\frac{j}{n}(\moms+\mom[2])\big)^{-1}=O(\frac{1}{n})$. 
\end{itemize}
Consequently, $\Pd(|\Pin^{(y)}(nt)|=k \mid y\not \in \Pin^{(x)}(nt) \et
|\Pin^{(x)}(nt)|=j)$ converges to  $\Pd(\Tp=k)$, which completes the proof. \qed
\section{Hydrodynamic behavior of the coalescent process\label{sect:coageq}}
This section is devoted to the proof of Theorem~\ref{thm:hydrodyn} describing 
the asymptotic limit of the average number of blocks having the same size. 
\begin{enumerate}
\item Let $t>0$ and $k\in \NN^*$.  First, we prove that  
$\rho_{n,k}(t) =\frac{1}{nk}\sum_{x=1}^{n}\un_{\{|\Pin^{(x)}(nt)| = k\}}$ converges in $L^2$ to  
$\rho_{k}(t) =\frac{1}{k}\Pd(\Tp(t) = k)$.
Theorem \ref{thmblocksize} and Corollary \ref{coroltwoblocks} imply the
convergence of the first two moments of $\rho_{n,k}(t)$ to 
$\rho_{k}(t)$ and $(\rho_{k}(t))^2$ respectively and thus the
$L^2$ convergence of $(\rho_{n,k}(t))_n$.  Indeed, 
$\Ed(\rho_{n,k}(t))=\frac{1}{k}\Pd(|\Pin^{(1)}(nt)| = k)$
converges to $\rho_{k}(t)$. The second moment is
$$\Ed((\rho_{n,k}(t))^2)=\frac{1}{nk^2}\Pd(|\Pin^{(1)}(nt)|=k)+
(1-\frac{1}{n})\frac{1}{k^2}\Pd(|\Pin^{(1)}(nt)|=k\; \et\;|\Pin^{(2)}(nt)|=k).$$ 
The first term converges to 0 and the second term converges to
$(\rho_{k}(t))^2$. 
\item 
It remains to show  that $\{\rho(t),\; t\in\RR_+\}$ is solution of the coagulation equations \eqref{coageq}:  
\[
\frac{d}{dt}\rho_k(t) =
\sum_{j=2}^{+\infty}p(j)\mK{j}(\rho(t),k)  
\]
where 
$$\mK{j}(\rho(t),k)=
\Big(\sum_{\substack{(i_1,\ldots,i_{j})\in(\NN^*)^{j}\\  i_1+\cdots
+i_j=k}}\prod_{u=1}^{j}i_u\rho_{i_u}(t)\Big)\un_{\{j\leq
k\}}-kj\rho_k(t).$$
By definition of $\rho(t)$,  for $j\in\NN\setminus\{0,1\}$, 
$$\mK{j}(\rho(t),k)=\Pd(\Tp[j](t)=k)\un_{\{j\leq
k\}} - j\Pd(\Tp[1](t)=k),$$ where $\Tp[\ell](t)$ denotes 
the total progeny of a BGW$(\ell,t\moms,\pb)$ process for every $\ell\in\NN^*$. \\
The probability distribution of $\Tp[\ell](t)$ is computed in the appendix (Lemma \ref{lemTdistr}):
\[
\left\{\begin{array}{l}
\Pd(\Tp[\ell](t) = \ell) = 
\displaystyle{e^{-t\ell\moms}}\\
\Pd(\Tp[\ell](t) = k) = \displaystyle{\frac{\ell}{k}e^{-kt\moms}
\sum_{h=1}^{k-\ell}\frac{(t\moms k)^h}{h!}(\pb)^{\star h}(k-\ell)}\quad \forall k\geq \ell+1.  
\end{array}\right.
\]
For $k=1$, $\rho_1$ is solution of the equation 
$\frac{d}{dt}\rho_1(t)=-\moms \rho_1(t)$ and the right hand side term is equal to  $\sum_{j=2}^{+\infty}p(j)\mK{j}(\rho(t),1)$.\\ 
Let us assume now that $k\geq 2$.  
\begin{multline*}\sum_{j=2}^{+\infty}p(j)\mK{j}(\rho(t),k)=\frac{e^{-tk\moms}}{k}\left(kp(k)+
\sum_{j=2}^{k-1}\sum_{h=1}^{k-j}\frac{(t\moms k)^h}{h!}jp(j)(\pb)^{\star h}(k-j)\right)
\\-\moms \Pd(\Tp[1](t)=k).
\end{multline*}
By using that $jp(j)=\moms\pb(j-1)$ for every $j\geq 2$ and by inverting 
the two sums we obtain
$$\sum_{j=2}^{+\infty}p(j)\mK{j}(\rho(t),k)=\frac{\moms}{k}e^{-tk\moms}
\sum_{h=1}^{k-1}\frac{(t\moms k)^{h-1}}{(h-1)!}(\pb)^{\star h}(k-1)
-\moms\Pd(\Tp[1](t)=k).$$
Since the right-hand side of the last formula is equal to $\frac{1}{k}\frac{d}{dt}\Pd(\Tp(t) = k)$, 
$$
\frac{d}{dt}\rho_k(t) =
\sum_{j=2}^{+\infty}p(j)\mK{j}(\rho(t),k)
$$
This completes the proof of Theorem \ref{thm:hydrodyn}.

\end{enumerate}
\section{Phase transition\label{sect:transition}}
The expectation of the compound Poisson distribution $\CPois(t\moms,\pb)$
is $t\mom[2]$. Thus the limiting BGW process associated with a block 
is subcritical, critical or supercritical depending on whether $t$ is smaller, 
equal or larger than $\frac{1}{\mom[2]}$.
This section is devoted to the proofs of Theorems \ref{thmtransition} and \ref{subcritlowerbound}, which provide
some results on the size of the largest block at time $nt$ in these three cases. 
\subsection{The subcritical regime}
Let us assume that $t< \frac{1}{\mom[2]}$. 
An application of the block exploration procedure and Fuk-Nagaev  inequality 
allows to prove 
that, if the  moment of $p$ of  order $u$ is finite for some $u\geq 3$, then the largest block size at time $nt$ is not greater than  $n^{1/(u-1)+\epsilon}$ for any $\epsilon>0$  
 with probability that converges to $1$. 
 If the probability 
 generating function of $p$ is assumed to be finite for some real greater than 1, then 
 it can be shown using a Chernoff bound that the largest block size at time $nt$ 
 is at most of order $\log(n)$
 with probability that converges to~$1$. 
\begin{bibli}[\ref{thmtransition}.(1)]
Let $0<t<\frac{1}{\mom[2]}$. 
\begin{enumerate}
\item[(a)] Assume that $p$ has a finite moment of order $u$ for some $u\geq 3$.
If   $(a_n)_n$ is a sequence of reals that tends to $+\infty$, 
then
$\Pd(\max_{x\in \enu}|\Pin^{(x)}(nt)|> a_n n^{\frac{1}{u-1}})$ converges 
to $0$ as $n$ tends to $+\infty$.
\item[(b)]  
Assume that $\Gp$ is finite on $[0,r]$ for some $r>1$. 
Set $h(t)=\sup_{\theta>0}(\theta-\log(L_{t}(\theta)))$
where $L_{t}$ is the moment-generating function of the compound Poisson
distribution $\CPois(t\moms,\pb)$.\footnote{$h(t)$ is
the value of the Cram\'er function at 1 of
$\CPois(t\moms,\pb)$.} \\
Then $h(t)>0$ and for every $a\geq (h(t))^{-1}$, 
$\Pd(\max_{x\in \enu}|\Pin^{(x)}(nt)|> a\log(n))$ converges to $0$ as $n$ tends
to $+\infty$.
\end{enumerate}
\end{bibli}
\begin{proof}
For $k\in \NN^*$, let $Z_k(t)$ denote the number of blocks of size greater than $k$ at time
$t$. Since each element of $\enu$ plays the same role, 
\begin{equation}
\label{ineqtailmaxblock}
\Pd(\max_{x\in \enu}|\Pin^{(x)}(t)|>k)\leq \Pd(Z_k(t)>k)\leq \frac{\Ed(Z_k(t))}{k}=
\frac{n}{k} \Pd(|\Pin^{(1)}(t)|>k). 
\end{equation} 
By construction of the random variables $\xi_{n,j}(t)$ and
$\bar{\zeta}_{n,j}(t)$,  $$\Pd(|\Pin^{(1)}(t)|> k)\leq
\Pd(\sum_{i=1}^{k}\xi_{n,i}(t)\geq k)\leq
\Pd(\sum_{i=1}^{k}\bar{\zeta}_{n,i}(t)\geq k).$$ 
\begin{enumerate}
 \item[(i)] First, let us assume that $p$ has a finite moment of order $u\geq 3$. 
 Set $e_n(t)=\Ed(\bar{\zeta}_{n,1}(t))$ and $X_{i,n}= \bar{\zeta}_{n,i}(nt)-e_n(nt)$
 for $i\in\enu$. \\
 Let us recall the Fuk-Nagaev inegality, we shall apply to the sequence $(X_{i,n})_{i=1\ldots,k}$:
 \begin{bibli}[Corollary 1.8 of \cite{Nagaev79}]
 Let $s\geq 2$ and let $Y_1,\ldots, Y_k$ be independent random variables such that  $\Ed(\max(Y_i,0)^s)<+\infty$ and $\Ed(Y_i)=0$ for every ${i\in \{1,\ldots,k\}}$. Set 
$ A_{s,k}^+=\sum_{i=1}^{k}\Ed(\max(Y_i,0)^s)$ and $B_k=\sum_{i=1}^{k}\var(Y_i)$. \\
 For every $x>0$,
 \[P(\sum_{i=1}^{k}Y_i\geq x)\leq x^{-s}c^{(1)}_s \sum_{i=1}^{k}\Ed(\max(Y_i,0)^s) +\exp(-c^{(2)}_s \frac{x^2}{B_k}),\]
  where $c^{(1)}_s= (1+2/s)^s$ and   $c^{(2)}_s= 2(s+2)^{-2}e^{-s}$.  
 \end{bibli}
 We begin by proving that $\Ed|X_{1,n}|^{u-1}$ is uniformly bounded in $n$. 
 Let us recall that the law of $\bar{\zeta}_{n,j}(nt)$ is $\CPois(nt\beta_n,\nu_n)$ 
 with 
 \[\nu_n(j)=\frac{1}{\beta_n}\left(1-(1-\frac{1}{n})^{j+1}-(\frac{1}{n})^{j+1}\right)p(j+1)
 \leq \frac{\moms}{n\beta_n}\pb(j)\quad \forall j\in\NN.
 \]
 Thus we  can apply the following property of compound Poisson distributions, 
 the proof of which is straightforward: 
 \begin{lem}
 Let $p_1$ and $p_2$ be two probability measures on $\NN^*$ and let $\lambda_1, \lambda_2$ be two positive reals 
 such that $p_1(j)\leq \frac{\lambda_2}{\lambda_1}p_2(j)$ $\forall j\in\NN^*$. 
 Let $X_1$ and $X_2$ be two random variables with compound Poisson distribution $\CPois(\lambda_1,p_1)$ 
 and $\CPois(\lambda_2, p_2)$ respectively. \\
 For every positive function $f$, $\Ed(f(X_1))\leq \Ed(f(X_2))\exp(\lambda_2-\lambda_1)$. 
 \end{lem}
This shows that 
\[\Ed\big(|X_{1,n}|^{u-1}\big)\leq e^{t(\moms-n\beta_n)}\Ed(|Y-e_n(nt)|^{u-1}),\] where $Y$ is a 
${\CPois(t\moms,\tilde{p})}$-distributed random variable.
Since $p$ has a finite moment of order $u$, $\tilde{p}$ has a finite moment of order $u-1$. 
Consequently, $\Ed(|Y|^{u-1})$ is finite. As $e_n(nt)$ converges to $t\mom[2]$ and $n\beta_n$ 
converges to $\moms$, we deduce that $\Ed|X_{1,n}|^{u-1}$ is uniformly bounded. \\
Therefore, by the Fuk-Nagaev inequality,  for every $k,n\in\NN^*$, 
\begin{equation}
\label{upperboundtail}
\Pd(|\Pin^{(1)}(nt)|> k)\leq 
 k^{2-u}M^{(1)}_{u,t}+\exp(-k M^{(2)}_{u,t})
\end{equation}
 where $M^{(1)}_{u,t}=c^{(1)}_{u-1}(1-t\mom[2])^{1-u}\sup_{n}\Ed|X_{1,n}|^{u-1}$ and 
$M^{(2)}_{u,t}=\frac{c^{(2)}_{u-1}(1-t\mom[2])^{2}}{t(\mom[3]+\mom[2])}$. 
 
In conclusion, there exists a constant $C_{t,u}>0$  such that if $(b_n)_n$ is a positive sequence that converges to $+\infty$, $\Pd(\max_{x\in\enu} |\Pin^{(x)}(t)|> b_n)\leq C_{t,u}\frac{n}{b_{n}^{u-1}}$ for every $n\in\NN$,  which completes the proof of assertion (a). 
\item[(ii)]
Let us assume now that $\Gp$ is finite on $[0,r]$ for some $r>1$. 
 The moment-generating function  of $\bar{\zeta}_{n,i}(nt)$ is finite on 
$[0,\log(r)[$ and is equal to $$
\Ed(e^{\theta\bar{\zeta}_{n,i}(nt)})=
\exp\Big(nt \sum_{j=1}^{+\infty}(e^{\theta j}-1)p(j+1)(1-(1-\frac{1}{n})^{j+1}-(\frac{1}{n})^{j+1}) \Big).$$ 
It is smaller than $L_t(\theta)=\exp\Big(t \sum_{j=1}^{+\infty}(e^{\theta j}-1)(j+1)p(j+1)\Big)$. 
 By Markov's inequality:
$$\Pd(|\Pin^{(1)}(nt)|> k)\leq
\Ed(e^{\theta\bar{\zeta}_{n,1}(nt)})^ke^{-k\theta}\leq
\exp\Big(-k\big(\theta-\log(L_{t}(\theta))
\big)\Big)\quad \forall 0<\theta<\log(r). 
$$ 
Since the expectation of the $\CPois(t\moms,\pb)$-distribution is assumed to be smaller than 1, 
$h(t)=\sup_{\theta>0}(\theta-\log(L_t(\theta)))$ is positive. 
We deduce that for every $ k\in \NN^*$, $\Pd(|\Pin^{(1)}(nt)|> k)\leq \exp(-kh(t))$. 
In particular, for every $a>0$,  
\[\Pd\left(\max_{x\in[n]} |\Pin^{(x)}(nt)|>
a\log(n)\right)\leq \frac{n^{1-ah(t)}}{\lfloor a\log(n)\rfloor}\exp(h(t)),\] 
which completes the proof of assertion (b). 
\end{enumerate}
\end{proof}
 Let us now prove the lower bound for the largest block stated in Theorem \ref{subcritlowerbound}:
\begin{bibli}[Theorem \ref{subcritlowerbound}]
Set $0<t<\frac{1}{\mom[2]}$. Assume  that $p$ is regularly varying  with index $-\alpha<-3$. \\ 
For every $\alpha'>\alpha$, $\Pd(\max_{x\in \enu}|\Pin^{(x)}(nt)|\leq n^{\frac{1}{1+\alpha'}})$ converges to $0$ as $n$ tends to $+\infty$.
\end{bibli} 
\begin{proof}[Proof of Theorem \ref{subcritlowerbound}]
 To prove this lower bound, we use a second moment method with the random variable $Z_k(nt)$ (which is the number of elements that belong to a block of size greater than $k$ at time $nt$).
\[\Pd(\max_{x\in \enu}|\Pin^{(x)}(nt)|\leq k)=\Pd(Z_k(nt)=0)\leq \frac{\var(Z_k(nt))}{\Ed(Z_k(nt))^2}. \]
Let us first give an upper bound for the variance $\var(Z_k(nt))$. 
Using properties \ref{itm:proprI} and \ref{itm:proprII} of $\mP_{n}(t)$ stated in 
Subsection \ref{subsect:restriction}, one can proceed as in (\cite{book1vanderHofstad}, Proposition 4.7) to obtain the following inequality:
\begin{equation} 
\label{varineqsub}
\var(Z_k(nt))\leq n\Ed(|\Pin^{(1)}(nt)|\un_{|\Pin^{(1)}(nt)|>k}).
\end{equation}
Let us  continue the proof of Theorem \ref{subcritlowerbound} before showing \eqref{varineqsub}. 
The right-hand side of \eqref{varineqsub} can be expressed by means of the tail distribution of $|\Pin^{(1)}(nt)|$:
\begin{equation}
\label{truncexpectation}
\Ed(|\Pin^{(1)}(nt)|\un_{|\Pin^{(1)}(nt)|>k})=k\Pd(|\Pin^{(1)}(nt)|>k)+\int_{[k,+\infty[}\Pd(|\Pin^{(1)}(nt)|>s)ds
\end{equation}
As  $p$ has a finite moment of order $\alpha_1$ for every $0<\alpha_1<\alpha$, an application of inequality \eqref{upperboundtail} deduced from the Fuk-Nagaev inequality yieds the following upper bound:  
 for every  $\alpha_1\in]3,\alpha[$. there exists $A_{t,\alpha_1}>0$ such that 
\begin{equation}
 \label{upperboundtruncexp}
 \Ed(|\Pin^{(1)}(nt)|\un_{|\Pin^{(1)}(nt)|>k})\leq A_{t,\alpha_1}k^{3-\alpha_1}.
\end{equation} \medskip

Let us now establish a lower bound for  $\Ed(Z_k(nt))=n\Pd(\Pin^{(1)}(nt)>k)$. 
By  Theorem \ref{thmblocksize}, there exists $C(t)>0$ such that  \[\Ed(Z_k(nt))\geq n(\Pd(\Tp(t)>k)-C(t)\frac{k^2}{n})\; \text{for every }k,n\in\NN^*.\]
To obtain a lower bound for $\Pd(\Tp(t)>k)$, we shall apply  several results on regularly varying distributions.  Let us first introduce a notation: for a nonnegative random variable $X$ with probability distribution $\nu$, let $\bar{F}_{X}(t)$ or $\bar{F}_{\nu}$  denote its tail distribution: $\bar{F}_{\nu}(t)=P(X>t)$ $\forall t\in\RR$. The following Lemma is an application of  a more general result on the solution of a fixed-point problem  proven in \cite{Asmussen18}: 
\begin{lem}
\label{totalprogenytail}
Let $\nu$ be a probability distribution on $\NN$ such that its expectation $m$ is smaller than $1$. 
Let $T$ be the total population size of a BGW process with offspring distribution $\nu$ and one ancestor.  \\
If $\nu$ is a regular varying distribution then $T$ has also  a regular varying distribution and $\bar{F}_T(x)\underset{x\rightarrow +\infty}{\sim}\frac{1}{1-m}\bar{F}_{\nu}((x-1)(1-m))$.  
\end{lem}
To apply this Lemma when the offspring distribution is a compound Poisson distribution, we can use the following result proven in (\cite{Embrechts79}, Theorem 3):
\begin{lem}
\label{cpoistail}
Let $\nu$ be  a regularly varying distribution on $\RR_+$ and $\lambda>0$. \\
Then, $\CPois(\lambda,\nu)$ is a regularly varying distribution on $\RR_+$ with the same index as $\nu$ and $\bar{F}_{\CPois(\lambda,\nu)}(x)\underset{x\rightarrow +\infty}{\sim} \lambda \bar{F}_{\nu}(x)$.
\end{lem}
As $\Tp(t)$ is the total population of a \BGW process with $\CPois(t\moms,\tilde{p})$-offspring distribution, it remains to show that $\bar{F}_{\tilde{p}}$ is a regularly varying function with index $-\alpha+1$. 
Let us note that for $k\in\NN$,  
\begin{equation}
\label{tailptilde}\moms\bar{F}_{\tilde{p}}(k)=(k +1)\bar{F}_{p}(k+1)+\int_{[k+1,+\infty[}\bar{F}_p(u)du.
\end{equation}
 The following result known as `Karamata Theorem for distributions' yields an asymptotic result for the last term in \eqref{tailptilde}: 
\begin{biblil}(see Theorem 2.45 in \cite{FossBook} for instance).
\label{integratedtail}
Let $F$ be a cumulative distribution function on $\RR_+$. \\
If $F$ is a regularly varying function with index $-\alpha<-1$ then the integrated tail distribution $F_I:x \mapsto \int_{x}^{+\infty}(1-F(u))du$ is a regularly varying function with index $-\alpha+1$ and 
$F_I(x)\underset{x\rightarrow +\infty}{\sim} (\alpha-1)^{-1}x(1-F(x))$.  
\end{biblil}
By this lemma we obtain
\begin{equation}
\label{tildeptail}
  \moms\bar{F}_{\tilde{p}}(x)\underset{x\rightarrow +\infty}\sim \frac{\alpha}{\alpha-1} \frac{\ell(\lfloor x \rfloor +1)}{(\lfloor x \rfloor+1)^{\alpha-1}}. 
\end{equation}
We deduce from  Lemma \ref{cpoistail} and Lemma  \ref{totalprogenytail} that 
\begin{equation}
\bar{F}_{\Tp(t)}(x)\underset{x\rightarrow +\infty}{\sim}\frac{t\alpha}{(1-t\mom[2])^{\alpha}(\alpha-1)}\frac{\ell\left(\lfloor x \rfloor(1-t\mom[2])\right)}{\lfloor x \rfloor^{\alpha-1}}.
\end{equation}
In summary, we have shown that there exists a slowly varying function $\tilde{\ell}$ such that for every ${k,n\in\NN^*}$, 
\begin{equation}
\label{lowerboundexpsubcrit}
\Ed(Z_k(nt))\geq n k^{-\alpha+1}(A^{(2)}_{\alpha,t}\tilde{\ell}(k)-C(t)\frac{k^{1+\alpha}}{n})
\end{equation} 
where $A^{(2)}_{\alpha,t}=\frac{t\alpha}{\alpha-1}(1-t\mom[2])^{-\alpha}$ and $C(t)$ is the constant defined in Theorem \ref{thmblocksize}. 
Set $k_n=n^{\frac{1}{1+\alpha'}}$ with $\alpha'>\alpha$. For $n$ large enough, the lower bound \eqref{lowerboundexpsubcrit} for
$\Ed(Z_{k_n}(nt))$ is positive. Using  \eqref{lowerboundexpsubcrit} and the upper bound \eqref{upperboundtruncexp} for $\var(Z_{k_n}(nt))$, we obtain that for every $3<\alpha_1<\alpha<\alpha'$ and $n$ large enough, 
\begin{equation}
\Pd(\max_{x\in \enu}|\Pin^{(x)}(nt)|\leq n^{\frac{1}{1+\alpha'}})\leq n^{2\alpha-\alpha_1-\alpha'}\frac{A^{(1)}_{t,\alpha_1}}{(A^{(2)}_{\alpha,t}\tilde{\ell}(n^{\frac{1}{1+\alpha'}})-C(t)n^{\frac{\alpha-\alpha'}{1+\alpha}})^2}
\end{equation}
If we take $\alpha_1\in]\max(3,\alpha-(\alpha'-\alpha)), \alpha[$, the upper bound converges to $0$ as $n$ tends to $+\infty$. This ends the proof of Theorem \ref{subcritlowerbound}.\\

It remains to show the upper bound for $\var(Z_k(nt))$ given by \eqref{varineqsub}.
We expand the value of $\var(Z_k(nt))$ by using that $Z_k(nt)$ is a sum of $n$ indicator functions and by splitting $\Pd(|\Pin^{(x)}(nt)|> k\;\text{and}\;|\Pin^{(y)}(nt)|> k)$ into two terms depending on whether  $x$ and $y$ belong to a same block or not:  
$\var(Z_k(nt))=S^{(1)}_{n}(k)+S^{(2)}_{n}(k)$, where 
\begin{alignat*}{2}
S^{(1)}_{n}(k)=\sum_{x,y\in\enu}&\Pd\big(|\Pin^{(x)}(nt)|> k\;\text{and}\; y\in \Pin^{(x)}(nt)\big)\\
S^{(2)}_{n}(k)=\sum_{x,y\in\enu}&\Big(\Pd\big(|\Pin^{(x)}(nt)|> k,\ |\Pin^{(y)}(nt)|> k \et 
y\not\in \Pin^{(x)}(nt)\big)\\ 
&-\Pd\big(|\Pin^{(x)}(nt)|> k\big)\Pd\big(|\Pin^{(y)}(nt)|> k\big)\Big).
\end{alignat*}
First,  $S^{(1)}_{n}(k)=n\Ed(|\Pin^{(1)}(nt)|\un_{\{|\Pin^{(1)}(nt)|>k\}}).$\\
We consider now the following term in $S^{(2)}_{n}(k)$: 
\begin{multline*}\Pd\big(|\Pin^{(x)}(nt)|> k,\ |\Pin^{(y)}(nt)|> k \et y\not\in \Pin^{(x)}(nt)\big)
\\  =\sum_{h=k+1}^{n-k}\Pd\big(|\Pin^{(y)}(nt)|> k\mid |\Pin^{(x)}(nt)|=h \et y\not\in \Pin^{(x)}(nt)\big)\Pd\big(|\Pin^{(x)}(nt)|=h \et y\not\in \Pin^{(x)}(nt)\big). 
\end{multline*}
Let $\Pin[n,h](nt)$ denote 
the partition  generated 
by  tuples the elements of which are in  $\enu[n-h]$ at time $nt$ and 
let $\Pin[n,h]^{(1)}(nt)$ denote the block of $\Pin[n,h](nt)$ that contains $1$. 
By the properties of the Poisson tuple set, for $h\in\{k+1,\ldots,n\}$
$$\Pd\big(|\Pin^{(y)}(nt)|> k \mid  y\not\in \Pin^{(x)}(nt) \et |\Pin^{(x)}(nt)|= h\big)=
\Pd\big(|\Pin[n,h]^{(1)}(nt)|> k \big)\leq \Pd\big(|\Pin^{(1)}(nt)|> k \big).$$
Thus $S^{(2)}_{n}(k)\leq 0$ which ends the proof of \eqref{varineqsub}. 
 \end{proof}

\subsection{The supercritical regime}
When $t>\frac{1}{\mom[2]}$,  BGW processes with 
family size distribution $\CPois(t\moms,\pb)$ are supercritical. 
We  show that there is a constant $c>0$ such that with high probability there is only one 
block with more than $c\log(n)$ elements and the size of this block is of order $n$. Let us recall the precise statement: 
\begin{bibli}[\ref{thmtransition}.(ii)]
\label{propgiantblock}
Let $B_{n,1}(nt)$ and $B_{n,2}(nt)$ denote the first and second largest blocks of  $\Pin(nt)$. 
Assume that $p$ has a finite moment of order three, $p(1)<1$ and  $t>\frac{1}{\mom[2]}$.  
Let $q_{t}$ denote the extinction probability of the 
BGW$(1,t\moms,\pb)$ process.\\
For every $a\in]1/2,1[$, there exist $b>0$ and $c>0$ such that 
$$\Pd[||B_{n,1}(nt)|-(1-q_{t})n|\geq n^a]+\Pd[|B_{n,2}(nt)|\geq c\log(n)]=O(n^{-b}).$$ 
\end{bibli}
In this section, we always assume the following hypothesis: 
\begin{quote}
$(\text{Hyp}_{p,t})$ :  $p$ has a finite moment of order three, $p(1)<1$ and  $t>\frac{1}{\mom[2]}$. 
\end{quote}
Let $h(t)$ be the Cram\'er function of the $\CPois(t\moms,\pb)$-distribution at point $1$: \\
 ${h(t)=\sup_{\theta\leq 0}(\theta-\log  \Ed(e^{\theta X}))}$, if $X$ denotes
 a $\CPois(t\moms,\pb)$-distributed random variable.\\
The proof of Theorem \ref{thmtransition}.(ii) consists of four steps: 
\begin{enumerate}
 \item In the first step, we show that the block of an element has a size greater than $c\log(n)$ 
 with a probability equivalent to the BGW process survival probability $1-q_{t}$. 
 \begin{prop}
 \label{surcritcompsize} Under assumption  $(\text{Hyp}_{p,t})$,  
  \[h(t)>0 \et 
  \forall a>h(t)^{-1},\; \Pd(|\Pin^{(x)}(nt)|\geq a\log(n))=1-q_{t}+O(\frac{\log^2(n)}{n}).\] 
 \end{prop}
 \item For $k\in\NN$, let $Z_k(nt)$ denote the number of elements that belong to a block 
 of size greater than $k$  at time $nt$. 
In the second step,  we study the first two moments of $Z_k(nt)$  in order to prove:
 \begin{prop}
 \label{surcritnblargecomp}
 Under assumption  $(\text{Hyp}_{p,t})$,
  for every $b\in]1/2;1[$, there exists $\delta>0$ such that if $a>h(t)^{-1}$ then 
  $\Pd(|Z_{a\log(n)}(nt)-n(1-q_{t})| > n^{b})=O(n^{-\delta})$.
 \end{prop}
 \item The aim of the third step is to prove that with high probability, there is no block  
 of size between $c_1\log(n)$ and $c_2n^\beta$ for any constant $\beta\in]1/2,1[$. More precisely, 
 we show the following result on the set of active elements in step $k$, denoted $A_{k}(x)$:
 \begin{prop}
  \label{propintermedsize}
  Let $\beta\in]1/2,1[$. Assume that $(\text{Hyp}_{p,t})$ holds.\\  
  For every ${0<c_2<\min(1,t\mom[2]-1)}$, there exists $\delta(c_2)>0$ such that for 
  ${c_1>\delta^{-1}(c_2)}$, 
  $$\Pd\big(\exists x\in\enu,\ A_{c_1\log(n)}(x)\neq \emptyset\text{ and } 
  \exists k\in[c_1\log(n),n^{\beta}],\ |A_k(x)|\leq c_2k\big)=O(n^{1-c_1\delta(c_2)}).$$
 \end{prop}
\item In the fourth step, we deduce from Proposition \ref{propintermedsize} that with high 
probability there  exists at most one  block of size greater than $a\log(n)$:
\begin{prop}
  \label{uniquegiantcomp}Assume that $(\text{Hyp}_{p,t})$ holds.
  For every $0<c_2<\min(1,t\mom[2]-1)$, there exists $\delta(c_2)>0$ such that 
  for $c_1>\delta^{-1}(c_2)$, 
  $$\Pd\big(\text{there exist two distinct blocks of size greater than}\; 
  c_1\log(n)\big)=O(n^{1-c_1\delta(c_2)}).$$
 \end{prop}
  \end{enumerate}
 Assertion (ii) of Theorem \ref{thmtransition} is then a direct consequence of Proposition
\ref{surcritnblargecomp}  and Proposition \ref{uniquegiantcomp}, since 
$Z_{c_1\log(n)}(nt)$ is equal to the size of the largest block on the event: 
$$\{|Z_{c_1\log(n)}(nt)-n(1-q_{t})|\leq n^{b}\}\cap 
\{\text{there is at most one block of size greater than}\;c_1\log(n)\}.$$
 The first two steps of the proof of assertion (ii) of Theorem \ref{thmtransition} are similar 
 to the  first two steps detailed in \cite{book1vanderHofstad} for 
 the Erd\"os-R\'enyi random graph. The last two steps follow the proof described in 
 \cite{BordenaveNotes} for the Erd\"os-R\'enyi random graph.  
 \begin{proof}[Proof of Proposition \ref{surcritcompsize}]
  Let $x\in\enu$. By Theorem \ref{thmblocksize}, for every $c>0$, 
  $$\Pd(|\Pin^{(x)}(nt)|> c\log(n))=\Pd(\Tp> c\log(n))+O(\frac{\log^2(n)}{n}).$$
  Moreover, $\Pd(\Tp=+\infty)=1-q_{t}$.
  To complete the proof, we use the following result on the total progeny of a 
  supercritical BGW process stated in \cite{book1vanderHofstad}: 
  \begin{bibli}[3.8 in \cite{book1vanderHofstad}]
   Let $T$ denote the total progeny of a BGW process with offspring distribution 
   $\nu$. 
   Assume that $\sum_{k\in\NN}k\nu(k)>1$. 
   Then, 
    \[
    I=\sup_{\theta\leq 0}\left(\theta-\log\left(\sum_{k=0}^{+\infty} e^{\theta x}\nu(k)\right)\right)>0\;\et\; 
   \Pd(k\leq T<+\infty)\leq \frac{e^{-kI}}{1-e^{-I}}.
   \]  
  \end{bibli}
  This theorem shows that for every $c>h(t)^{-1}$, $\Pd(c\log(n)< \Tp<+\infty)=O(n^{-1})$ and 
  $$\Pd(|\Pin^{(x)}(nt)|> c\log(n))=1-q_{t}+O(\frac{\log^2(n)}{n}).$$
 \end{proof}
\begin{proof}[Proof of Proposition \ref{surcritnblargecomp}]
In order to  apply Bienaym\'e-Chebyshev inequality to obtain an upper bound for   
$\Pd(|Z_{a\log(n)}(nt)-n(1-q_{t})| > n^{b})$,
we compute the expectation and the variance of 
$Z_k(nt)=\sum_{x\in\enu}\un_{\{|\Pin^{(x)}(nt)|> k\}}$.  
First, we deduce from Proposition \ref{surcritcompsize} that 
if  $a>h(t)^{-1}$ then $$\Ed(Z_{a\log(n)}(nt))=n(1-q_{t})+O(\log^2(n)).$$ 
We  proceed as in (\cite{book1vanderHofstad}, Proposition 4.10) to prove the following upper bound for the variance of $Z_k(nt)$:
\begin{equation}
\label{boundvarsup}
\var(Z_k(nt))\leq n(1+kt\mom[2])\Ed(|\Pin^{(1)}(nt)|\un_{|\Pin^{(1)}(nt)|\leq k})
\end{equation}
The beginning of the calculation is similar to the one used to prove inequality \eqref{varineqsub}:
the variance of $Z_k(nt)$ which is equal to the variance of  
$\sum_{x\in\enu}\un_{\{|\Pin^{(x)}(nt)|\leq k\}}$ can be written as the sum of the following two terms: 
\begin{alignat*}{2}
\tilde{S}^{(1)}_{n}(k)=\sum_{x,y\in\enu}&\Pd\big(|\Pin^{(x)}(nt)| \leq k\;\text{and}\; y\in \Pin^{(x)}(nt)\big)=n\Ed(|\Pin^{(1)}(nt)|\un_{\{|\Pin^{(1)}(nt)|\leq k\}})\\
\tilde{S}^{(2)}_{n}(k)=\sum_{x,y\in\enu}&\Big(\Pd\big(|\Pin^{(x)}(nt)| \leq k,\ |\Pin^{(y)}(nt)|\leq  k \et 
y\not\in \Pin^{(x)}(nt)\big)\\ 
&-\Pd\big(|\Pin^{(x)}(nt)| \leq k\big)\Pd\big(|\Pin^{(y)}(nt)|\leq  k\big)\Big).
\end{alignat*}
We consider the following term in $\tilde{S}^{(2)}_{n}(k)$: 
\begin{alignat*}{2}\Pd\big(|\Pin^{(x)}(nt)|\leq  k,& |\Pin^{(y)}(nt)|\leq  k \et y\not\in \Pin^{(x)}(nt)\big)
\\  &=\sum_{h=1}^{k}\Pd\big(|\Pin^{(x)}(nt)|=h,\ |\Pin^{(y)}(nt)| \leq k \et y\not\in \Pin^{(x)}(nt)\big).
\\ &\leq \sum_{h=1}^{k} \Pd\big(|\Pin^{(x)}(nt)|= h\big)
\Pd\big(|\Pin^{(y)}(nt)|\leq  k \mid  y\not\in \Pin^{(x)}(nt) \et |\Pin^{(x)}(nt)|= h\big).
\end{alignat*}
By the properties of the Poisson tuple set, 
$$\Pd\big(|\Pin^{(y)}(nt)|\leq  k \mid  y\not\in \Pin^{(x)}(nt) \et |\Pin^{(x)}(nt)|= h\big)=
\Pd\big(|\Pin[n,h]^{(1)}(nt)|\leq  k \big)$$
where   $\Pin[n,h](nt)$ denotes 
the partition  generated 
by  tuples the elements of which are in  $\enu[n-h]$ at time $nt$ and $\Pin[n,h]^{(1)}(nt)$ denotes the block of $\Pin[n,h](nt)$ that contains $1$. 
 
We can couple  $\mP_n(nt,\enu[n-h])$ and $\mP_n(nt)$ 
by adding to  $\mP_n(nt,\enu[n-h])$ tuples of an independent 
Poisson point process on $\RR^{+}\otimes \mW$ at time $nt$ that are not included 
in $\enu[n-h]$. 
Therefore, 
$\Pd(|\Pin[n,h]^{(1)}(nt)| \leq k )-\Pd(|\Pin^{(1)}(nt)| \leq k )$ is equal to the probability that 
$|\Pin[n,h]^{(1)}(nt)|$ is smaller than or equal to $k$ 
and that $|\Pin^{(1)}(nt)|$  is greater than $k$. This probability is bounded above by the 
probability that there exists  $w\in \mP_{n}(nt)$ that contains both elements 
of $\{1,\ldots,k\}$ and elements of $\{n-h+1,\ldots,n\}$.
Therefore, 
$$
\Pd(|\Pin[n,h]^{(1)}(nt)|\leq k )-\Pd(|\Pin^{(1)}(nt)|\leq k )\\\leq 
1-e^{-nt I_{n}(k,h)}$$
with 
\begin{eqnarray*}I_{n}(k,h)&=&\mu_n\big(w\in \mW,\ w\cap\enu[k]\neq \emptyset \et  
w\cap\{k+1,\ldots,k+h\}\neq \emptyset\big)\\
&=&\mu_n(\mW)-\mu_n(\mW[\{k+1,\ldots,n\}])\\
&&-\mu_n(\mW[\enu\setminus\{k+1,\ldots,k+h\}])+
\mu_n(\mW[\{k+1+h,\ldots,n\}])\\
&=&1-\Gp(1-\frac{k}{n})-\Gp(1-\frac{h}{n})+\Gp(1-\frac{k+h}{n})\\
&\leq& \frac{k h}{n^2}\mom[2].
\end{eqnarray*}
We deduce that 
\[\tilde{S}^{(2)}_{n}(k)\leq
\sum_{x,y\in\enu}\Big(\sum_{h=1}^{k}\Pd(|\Pin^{(x)}(nt)|=h)\frac{tkh}{n}\mom[2]\Big)= ntk\mom[2]\Ed(|\Pin^{(1)}(nt)|\un_{\{|\Pin^{(1)}(nt)|\leq k\}})
\]
which yields \eqref{boundvarsup}. 
 
Let us note that for every $\delta>0$, $\dfrac{\var(Z_{a\log(n)}(nt))}{n^{1+\delta}}$
 converges to $0$ as $n$ tends to $+\infty$. Therefore,  
 Bienaym\'e-Chebyshev inequality is sufficient
 to complete the proof. 
\end{proof}
\begin{proof}[Proof of Proposition \ref{propintermedsize}]
Let $\alpha\in]1/2,1[$. 
The idea of the proof is to lower bound the number of new active elements at the first steps 
of the block exploration procedure by considering only tuples inside a subset of 
$m_n=n-\lceil 2n^\alpha\rceil$ elements. For large $n$, the BGW process associated with 
this block exploration procedure is still supercritical. \\
Let $\tau=T^{(n)}_t\wedge \min(k\in\NN^*,\ \sum_{i=1}^{k}\xi_{n,i}(t)\geq 2n^{\alpha})$.
 On the event $\{k\leq \tau\}$, the number of neutral elements at step $k$ is greater than $m_n$. 
 Let $U_k$ denote the set of the $m_n$ first neutral elements at step $k$ and 
 let $Y_{n,k+1}(t)$ denote the number of $y\in U_k$ which are contained in a tuple
 $w\in \mP_{n,x_k}(t,U_k\cup \{x_k\})$.
 On the event $\{k\leq \tau\}$, $Y^{(n)}_{t,k+1}\leq \xi_{n,k+1}(t)$. Therefore, 
 $\sum_{i=1}^{k\wedge \tau}Y_{n,i}(t)\leq \sum_{i=1}^{k\wedge \tau}\xi_{n,i}(t)$.
\\
For  $x\in\enu$, set  
$$\Omega^{(n)}_{c_1,c_2}(x)=\{A_{c_1\log(n)}(x)\neq \emptyset\text{ and } 
  \exists k\in[c_1\log(n),n^{\alpha}],\ |A_k(x)|\leq c_2k\}. $$
  On the event $\{k\leq \tau \et  |A_k(x)|\leq c_2k\}$, 
  $\sum_{i=1}^{k}Y_{n,i}(t)$ is bounded above by $(c_2+1)k-1$. 
  Thus, 
\begin{eqnarray*}\Pd(\Omega^{(n)}_{c_1,c_2}(x))&\leq &\sum_{k=c_1\log(n)}^{n^{\alpha}}
\Ed\left(\Pd( A_{c_1\log(n)}(x)\neq \emptyset\text{ and }  
|A_k(x)|\leq c_2k \mid \mathcal{F}_{k-1})\un_{\{k\leq \tau\}}\right)\\
&\leq& \sum_{k=c_1\log(n)}^{n^{\alpha}}\Pd\left(\sum_{i=1}^{k}\tilde{Y}_{n,i}(t)\leq (c_2+1)k-1\right).
\end{eqnarray*}
where $(\tilde{Y}_{n,i}(t))_i$ denotes a sequence of independent random variables distributed 
as $|\Ne_{1}(t,\enu[m_n+1])|$. 
The last step consists in  establishing an exponential bound for
$$p_{n,k}:=\Pd\Big(\sum_{i=1}^{k}\tilde{Y}_{n,i}(t)\leq (c_2+1)k-1\Big)$$ uniformly on $n$. 
A such exponential bound is an easy consequence of the following two facts:
\begin{itemize}
\item[(i)]  $c_2+1$ is smaller than the expectation of the 
$\CPois(t\moms,\pb)$-distribution. 
\item[(ii)] $(\tilde{Y}_{n,1}(t))_n$ converges in law to the 
$\CPois(t\moms,\pb)$-distribution by Proposition \ref{convvoisins}. 
\end{itemize}
For every $\theta>0$, $p_{n,k}\leq \exp(k\Lambda_n(-\theta))$ where 
$\Lambda_n(\theta)=\log\Big(\Ed(e^{\theta(\tilde{Y}_{n,1}(t)-(c_2+1))})\Big)$. 
Let $Y$ be a $\CPois(t\moms,\pb)$-distributed random variable.  
Set  $\Lambda(\theta)=\log\big(\Ed(e^{\theta( Y-(c_2+1))})\big)$  for $\theta\leq 0$. 
Since $\Ed(Y)=t\mom[2]$ is finite, $\Lambda'(-\theta)$ converges to $-\Ed(Y)+c_2+1$ 
which is negative as $\theta$ converges to $0$.   
Therefore,  there exists $u^*<0$ such
that $\Lambda(u^*)<0$. Set $\delta=-\frac12\Lambda(u^*)$. 
By assertion (ii),  $\Lambda_n(u^*)$ converges to $\Lambda(u^*)$, hence there exists $n^*$ such that 
for every $n\geq n^*$ and $k\in\NN^*$, $p_{n,k}\leq \exp(-k\delta)$. 
We deduce that  for  $n\geq n^*$,
$$\Pd\Big(\underset{x\in\enu}{\cup}\Omega^{(n)}_{c_1,c_2}(x)\Big)\leq n\Pd(\Omega^{(n)}_{c_1,c_2}(1))\leq n^{1-c_1\delta}(1-e^{-\delta})^{-1}$$
 which converges to $0$ if  $c_1>\delta^{-1}$. \\
\end{proof}
\begin{proof}[Proof of Proposition \ref{uniquegiantcomp}]
For $0<c_1<1$ and $c_2>0$, let  $\Omega^{(n)}_{c_1,c_2}$ denote the event 
$$\{\exists  x\in\enu\;\text{such that}\;
A_{c_1\log(n)}(x)\neq \emptyset \et \exists k\in[c_1\log(n), n^{\alpha}]\;\text{such that}\;
|A_k(x)|\leq c_2k\}.$$
 It occurs  with probability $O(n^{1-c_1\delta(c_2)})$ by Proposition \ref{propintermedsize}. \\
Assume that $\Omega^{(n)}_{c_1,c_2}$ does not hold and  that there exist two elements 
$x_1$ and $x_2$ in $\enu$ contained in two different blocks both of size greater than $c_1\log(n)$. 
The subsets of active elements in step $n^{\alpha}$, $A_{n^\alpha}(x_1)$ and $A_{n^\alpha}(x_2)$, 
are disjoint and both of size greater than 
$c_2n^\alpha$. It means that no tuple $w\in\mP_{n}(nt)$  contains both elements of 
$A_{n^\alpha}(x_1)$ and  $A_{n^\alpha}(x_2)$. 
Let us note that if $F_1$ and $F_2$ are two disjoint subsets of $\enu$ then 
\begin{alignat*}{1}
\Pd(\nexists w \in\mP_{n}(nt),\;& w \cap F_1\neq\emptyset \et w \cap F_2\neq\emptyset)\\
&=
\exp\Big(-nt\mu(w \in\mW,\; w \cap F_1\neq\emptyset \et w \cap F_2\neq\emptyset)\Big)\\
&=\exp\Big(-nt\big(1-\Gp(1-\frac{|F_1|}{n})-\Gp(1-\frac{|F_2|}{n})+\Gp(1-\frac{|F_1|+|F_2|}{n})\big)\Big). 
 \end{alignat*}
Therefore if $F_1$ and $F_2$ are two disjoint subsets of $\enu$ of size greater than $c_2n^\alpha$ 
with  $n$ large enough, 
\begin{multline*}
\Pd(\nexists w \in\mP_{n}(nt),\; w \cap F_1\neq\emptyset \et w \cap F_2\neq\emptyset) \\\leq 
\exp\Big(-nt\big(1-2\Gp(1-c_2n^{\alpha-1})+\Gp(1-2c_2n^{\alpha-1})\big)\Big)\\\leq 
\exp(-\frac{1}{2}c^2_2t\mom[2]n^{2\alpha-1}),
\end{multline*}
since $x\mapsto 1-\Gp(1-x)-\Gp(1-x-y)$ is an increasing function on $]0,1-y[$ for every $y\in]0,1[$ and for $x>0$  small enough, 
 $1-2\Gp(1-x)+\Gp(1-2x)\geq \frac{x^2}{2}\mom[2]$. 
\\
Set $J_{n,\alpha}=\{(x_1,x_2)\in \enu^2,\ A_{n^{\alpha}}(x_1)\cap A_{n^{\alpha}}(x_2)= \emptyset,  
|A_{n^{\alpha}}(x_1)| > c_2 n^{\alpha},\ |A_{n^{\alpha}}(x_2)| > c_2 n^{\alpha}\}$. 
It follows from the last inequality that  there exists two different blocks of size greater than $c_1\log(n)$ 
with a probability smaller than the sum of $\Pd(\Omega^{(n)}_{c_1,c_2})$ and 
\begin{multline*}
\Ed\left(\sum_{(x_1,x_2)\in J_{n,\alpha}}
\Pd(\nexists w \in\mP_{n}(nt),\; w \cap  A_{n^{\alpha}}(x_1)\neq\emptyset 
\et w \cap  A_{n^{\alpha}}(x_2)\neq\emptyset \mid \mathcal{F}_{n^\alpha})\right)\\
\leq n^2\exp\left(-\frac{1}{2}tc_{2}^{2}\mom[2]n^{2\alpha-1}\right). 
\end{multline*}
As $\alpha\in]\frac12,1[$,  this probability is of order $O(n^{1-c_1\delta(c_2)})$. 
\end{proof}
\subsection{The critical regime}
Let us now study  block sizes at time  
 $t_n=\frac{1}{\mom[2]}(1+\theta \epsilon_n)$ where $\theta>0$ and $(\epsilon_n)_n$ is a sequence of positive reals that converges to $0$. 
 The aim of this section is to prove the third statement of Theorem \ref{thmtransition}: 
\begin{bibli}[\ref{thmtransition}.(iii)]
Assume that $p$ is a probability measure on $\NN^*$ with $p(1)<1$ and a finite 
third moment.  For every $\theta>0$, there exists a constant $b>0$ such that 
for every $c>1$ and $n\in \NN^*$ 
\begin{equation}
\Pd\left(\max_{x\in \enu}\Big\lvert \Pin^{(x)}\left(\frac{n}{\mom[2]}(1+\theta n^{-1/3})\right)\Big\rvert  > 
cn^{2/3}\right)\leq \frac{c}{b}.\tag{\ref{critupperbound}}
\end{equation}
\end{bibli}
Let us recall  that the size of a block is smaller than the total population 
size of a \BGW$(1,n\beta_n,\nu_n)$ process, which is itself close to the total 
population size of a \linebreak[4]{\BGW$(1,t_n\moms,\tilde{p})$} process. 
Therefore the strategy of proof used  to establish the same result for 
the Erd\"os-R\'enyi random graph in (\cite{book1vanderHofstad}, Theorem 5.1) 
can be followed if we are able to show the following two properties for the total population size $T^{(1)}_p(t_n)$ of the \BGW$(1,t_n\moms,\tilde{p})$ process: 
\begin{itemize}
\item the survival probability of a \BGW$(1,t_n\moms,\tilde{p})$ process is of order $O(\epsilon_n)$. 
\item  There exists a constant $c>0$ such that 
$\Pd(\Tp(t_n)=k)\leq ck^{-3/2}$ for every $n\in\NN$ and $k\in\NN^*.$  
\end{itemize}
Let us now detail the proof of  \eqref{critupperbound}. 
As in the study of the subcritical phase, we reduce the proof to the study of
 $\Pd(|\Pin^{(1)}(nt_n)|\geq k)$, by noting that for every $k\in\enu[n]$, 
$\Pd(\max_{x\in \enu}|\Pin^{(x)}(nt_n)|\geq k)\leq \frac{n}{k}\Pd(|\Pin^{(1)}(nt_n)|\geq k)$ (see \eqref{ineqtailmaxblock}).\\
Since $|\Pin^{(1)}(nt_n)|$ is smaller than the total population size of a BGW$(1,nt_n\beta_n,\nu_n)$ process, 
by Proposition \ref{propcouplingtotalpopul}, for every $k\geq 1$, 
\begin{equation}
\label{critupperboundblock}
\Pd(|\Pin^{(1)}(nt_n)|\geq k)\leq \Pd(\Tp(t_n)\geq k)+\frac{t_n}{2n}(2\mom[2]+1)\sum_{i=1}^{k-1}\Pd(\Tp(t_n)\geq i)
\end{equation}
and $\Pd(\Tp(t_n)\geq k)=\sum_{i=k}^{\infty}\Pd(\Tp(t_n)= k)+1-q_{t_n}$, where 
$q_{t_n}$ is the extinction probability of the \BGW$(1,t_n\moms,\tilde{p})$ process. 
 
To estimate the survival probability $1-q_{t_n}$, we use the following inequalities: 
\begin{lem}
\label{lemsurvineq}
Let $\nu$ be a probability measure on $\NN$. Assume that $\nu$ has a finite second moment, 
$\nu(0)+\nu(1)<1$ and the first moment $m_{\nu,1}$ is greater than 1.  Let  $m_{\nu,2}$ denote the 
second factorial moment of $\nu$. \\ The survival probability $\alpha$
of a BGW process with offspring distribution $\nu$ and one ancestor satisfies:  
\begin{equation}
\label{survineq}
2\frac{m_{\nu,1}-1}{m_{\nu,2}}\leq \alpha \leq \frac{m_{\nu,1}-1}{m_{\nu,1}-1+\nu(0)}.
\end{equation}
\end{lem}
The lower bound was proved by Quine in \cite{Quine}. A simple 
proof of this lemma is given in Appendix \ref{annex:generfunct}. 
By Lemma \ref{lemsurvineq},
\begin{equation}
\label{survcpoisgw}
\frac{2\theta \epsilon_n}{(t_n\mom[2])^2+t_n\mom[3]}\leq 1-q_{t_n} \leq   \frac{\theta \epsilon_n}{\theta\epsilon_n+\exp(-t_n\moms)}. 
\end{equation}
To estimate $\Pd(\Tp(t_n)= k)$, we first rewrite it by the mean of  Dwass identity:  
\[\forall k\in\NN^*,\; \Pd(\Tp(t_n)= k)=\frac{1}{k}\Pd(\sum_{i=1}^{k}\xi_{i}= k-1),\] 
where $(\xi_{i})_i$ denotes a sequence of independent random variables with $\CPois(t_n\moms,\tilde{p})$-distribution 
(the statement of Dwass's theorem is recalled in Appendix \ref{annex:totalprogenydistr}). 
The local central limit theorem applied to the sequence $(\xi_{i})_i$  yields  
$\Pd(\Tp(t_n)=k)=O(k^{2/3})$ at a fixed time $t_n$. But we need a bound uniform in $n$. 
A careful study of the  local central limit theorem proof shows  that the convergence is uniform if it is applied to  well-chosen 
families of probability  distributions. In particular, in our setting:   
\begin{lem}
\label{lclt}
Let $\nu$ be a probability distribution on $\NN$ with a finite second moment. 
Let $m_{\nu,1}$ and $m_{\nu,2}$ denote the first two factorial moments of 
$\nu$.  For $\lambda>0$, let $(X_{\lambda,n})_n$ be a sequence of independent 
$\CPois(\lambda,\nu)$-distributed random variables.   
Let $r$ be the largest positive integer such that the support of the 
$\CPois(1,\nu)$-distribution  is included in $r\NN$. For every $0<a<b$,
\[
\sup_{\lambda\in[a,b],k\in \NN}\!\!\!\sqrt{n}\left|\Pd(\sum_{i=1}^{n}X_{\lambda,i}=kr)-
\frac{r}{\sqrt{2\pi n\lambda (m_{\nu,2}+m_{\nu,1})}}
\exp\left(-\frac{(kr-n\lambda m_{\nu,1})^2}{2n\lambda (m_{\nu,2}+m_{\nu,1})}\right)\right|
\underset{n\rightarrow +\infty}\longrightarrow 0.
\]
\end{lem} 
\noindent(See Appendix \ref{annex:lclt} for a proof of Lemma \ref{lclt})\\ 
Lemma \ref{lclt} implies that there exists a sequence $(\delta_k)_k$ 
(depending on $\theta$) that converges to $0$  such that for every 
$n\in \NN^*$ and $k\in\NN^*$,  $\Pd(\Tp(t_n)=k)\leq\frac{c+\delta_k}{k^{3/2}}$,
 where ${c=\frac{r}{\sqrt{2\pi a (\mom[2]+\mom[3])}}}$ and $r$ is the 
 largest integer such that 
the support of $\CPois(1,\tilde{p})$ is included in $r\NN$. 
From this and the upper bound of  \eqref{survcpoisgw}, we deduce that
there exists $c_1(\theta)>0$ such that for every $n,k\in\NN^*$, 
$\Pd(\Tp(t_n)\geq k)\leq c_1(\theta)(\frac{1}{\sqrt{k}}+\theta\epsilon_n)$.
Thus by \eqref{critupperboundblock}, there exists $c_1(\theta)>0$ such that 
for every $n\in\NN^*$ and $k\in \enu$,  
\[
\Pd(|\Pin^{(1)}(nt_n)|\geq k)\leq c_1(\theta)(\frac{1}{\sqrt{k}}+\theta\epsilon_n)
+c_2(\theta)(\frac{\sqrt{k}}{n}+\frac{k}{n}\theta\epsilon_n)
\leq (c_1(\theta)+c_2(\theta))(\frac{1}{\sqrt{k}}+\theta\epsilon_n). 
\]
To complete the proof of the statement (iii) of Theorem \ref{thmtransition}, 
it suffices to apply inequality \eqref{ineqtailmaxblock}. 
\appendix

\section{Some properties of BGW processes with a compound Poisson  offspring
distribution \label{annex:propGW}}
\subsection{Probability generating function \label{annex:generfunct}}
First, let us establish inequalities for the probability generating 
function of a distribution having a finite second moment; This provides  
 a simple proof for Lemma \ref{lemsurvineq}:
\begin{lem}
\label{pgf}
Let $\nu$ be a probability measure on $\NN$. Let us assume 
that $\nu$ has a finite second moment and that $\nu(0)+\nu(1)<1$.   
Let  $m_{\nu,1}$ and $m_{\nu,2}$ denote the first two factorial moments of $\nu$,  
\begin{enumerate}
 \item The probability generating
function of $\nu$ denoted by $G_{\nu}$ satisfies:
\begin{equation}
\label{pgfineq}
(m_{\nu,1}-1+\nu(0))(s-1)^2\leq G_{\nu}(s)-1-(s-1)m_{\nu,1}
\leq \frac12 m_{\nu,2}(s-1)^2\; \forall s\in[0,1].
\end{equation}
\item Assume that $m_{\nu,1}$ is greater than $1$. The survival probability $\alpha$
of a BGW process with offspring distribution $\nu$ and one ancestor satisfies: 
\[
2\frac{m_{\nu,1}-1}{m_{\nu,2}}\leq \alpha \leq \frac{m_{\nu,1}-1}{m_{\nu,1}-1+\nu(0)} \tag{\ref{survineq}}
\]
\end{enumerate}
\end{lem}
\begin{proof}
\begin{enumerate}
 \item
Let us note that 
$1-G_{\nu}(s)=(1-s)m_{\nu,1}H(s)$ where $H$ is the generating function of 
the probability $\eta$ defined by 
$$\eta(k)=\frac{1}{m_{\nu,1}}\sum_{\ell\geq k+1}\nu(\ell)\; \forall k\in\NN.$$ 
By writing a similar formula for $1-H$ (it is possible since $H$  has a finite expectation), 
we obtain:
$$G_{\nu}(s)=1+(s-1)m_{\nu,1}+\frac12 m_{\nu,2}(s-1)^2K(s)$$ where 
$K$  is the generating function of the probability $\rho$ defined by 
 $$\rho(k)=\frac{2}{m_{\nu,2}}\sum_{\ell\geq k+2}(\ell-1-k)\nu(\ell)\;\forall\ell\in\NN. $$ 
 In particular, for every $s\in[0,1]$, $\frac{2}{m_{\nu,2}}(m_{\nu,1}-1+\nu(0))\leq K(s)\leq 1$.
 \item The extinction probability $q=1-\alpha$ is smaller than $1$ and satisfies $G_{\nu}(q)=q$. 
  The second assertion is obtained by taking  $s=q$ in \eqref{pgfineq}. 
 \end{enumerate}
\end{proof}

\subsection{Total progeny distribution \label{annex:totalprogenydistr}}
Let us turn to the total population size of a BGW process. 
 A useful tool to study its distribution is the following formula known as Dwass identity: 
\begin{bibli}(\cite{Dwass})
 Consider a BGW process with offspring distribution $\nu$ and $u\geq 1$
ancestors. Let 
$T$ denote its total progeny and let $(X_n)_n$ be a sequence of independent
random variables with distribution $\nu$. 
 \begin{equation}
 \forall k\geq u,\ \Pd(T=k)=\frac{u}{k}P(X_1+\ldots+X_k=k-u). 
 \end{equation}
\end{bibli}
Recall that in the supercritical case (i.e. $\sum_{k}\nu(k)>1$),  $P(T<+\infty)=q^u<1$ 
if $q$ denotes the extinction probability of the BGW process starting from one ancestor. \\ 
Using Dwass identity, we obtain:
\begin{lem}\label{lemTdistr} 
Let $T^{(u)}$ denote the total progeny  of a BGW process with $u$
ancestors and with offspring distribution 
$\CPois(\lambda,\nu)$. 
Then,  
\begin{equation}
\label{TGWdistrib}
\left\{\begin{array}{l}P(T^{(u)}= u) = \displaystyle{e^{-u \lambda}}\\
P(T^{(u)} = k) = \displaystyle{\frac{u}{k}e^{-k
\lambda}\sum_{j=1}^{k-u}\frac{(\lambda k)^j}{j!}\nu^{\star j}(k-u)}\quad \forall k\geq u+1.  \\
\end{array}\right.
\end{equation}
where $\nu^{\star j}$ denotes the $j$-th convolution power of $\nu$.
\end{lem}
\begin{proof}
It suffices to observe that  the  sum $X_1+\ldots+X_k$ appearing 
in  Dwass's theorem   has the $\CPois(k\lambda,\nu)$-distribution. 
\end{proof}
\subsection{Local central limit theorem for a family  of compound Poisson distributions
\label{annex:lclt}}
Let $\nu$ be a probability distribution on  $\NN$ with a finite second moment. 
For $\lambda>0$, let $(X_{\lambda,n})_n$ denote a sequence of independent 
random variables with $\CPois(\lambda,\nu)$ distribution. 
The aim of this paragraph is to prove Lemma \ref{lclt}, which states that 
for every $0<a<b$,  the speed of convergence in the local limit theorem 
for $(X_{\lambda,n})_n$ can be bounded uniformly for $\lambda\in[a,b]$ by $o(\frac{1}{\sqrt{n}})$. 
Without loss of generality,  we can assume that there is no $r>1$ such that 
$\Pd(X_{1,1}\in a+r\ZZ)=1$ for some $a\in\ZZ$ 
(otherwise, it suffices to consider $\frac{1}{r_*}X_{\lambda,n}$ instead of $X_{\lambda,n}$, 
where $r_*$ is the largest 
$r\in\NN^*$ such that  $\Pd(X_{\lambda,1}\in a+r\ZZ)=1$ for some $a\in\ZZ$;  
and to note that $r_{*}$ does not depend on $\lambda$). We follow the presentation 
of the local limit theorem proof proposed in (\cite{lawlerlimic}, Theorem 2.3.9). 
Let $m_{\lambda,\nu}$ 
and  $\sigma^{2}_{\lambda,\nu}$ denote the 
expectation and variance of  $X_{\lambda,n}$ respectively. 
We have to prove that for every $k\in \NN$,  
\[R_{\lambda,n}(k)=2\pi\left(\sqrt{n}\Pd(\sum_{i=1}^{n}X_{\lambda,i}=k)-
\frac{1}{\sqrt{2\pi \sigma^{2}_{\lambda,\nu}}}
\exp\left(-\frac{(k-n m_{\lambda,\nu})^2}{2\sigma^{2}_{\lambda,\nu}}\right)\right)
\]
converges to 0 uniformly for $\lambda\in[a,b]$ as $n$ tends to $+\infty$. 
 
Let $\varphi_{\lambda}$ denote the characteristic function of 
$X_{\lambda,1}-\Ed(X_{\lambda,1})$. 
The first term of  $R_{\lambda,n}(k)$ can be rewritten: 
\[ 
2\pi\sqrt{n}\Pd(\sum_{i=1}^{n}X_{\lambda,i}=k) =
\int_{-\pi \sqrt{n}}^{\pi\sqrt{n}}\varphi_{\lambda}(\frac{x}{\sqrt{n}})^n h_{\lambda,n,k}(x) dx,
\]
where $h_{\lambda,n,k}(x)=e^{ix(m_{\lambda,\nu}\sqrt{n}-\frac{k}{\sqrt{n}})}$. 
For the second term, the Fourier inversion theorem yields: 
$$\frac{\sqrt{2\pi}}{\sigma_{\lambda,\nu}}\exp\left(-\frac{(k-nm_{\lambda,\nu})^2}{2\sigma^{2}_{\lambda,\nu}}\right)
=\int_{\RR}e^{-\frac{x^2}{2}\sigma^{2}_{\lambda,\nu}}h_{\lambda,n,k}(x)dx.$$
Thus, $|R_{\lambda,n}(k)|$ is bounded by the sum of three terms: 
\begin{eqnarray*}
I_{1,\lambda,\epsilon}(n) & = & \int_{|x|< \epsilon \sqrt{n}}|\varphi_{\lambda}(\frac{x}{\sqrt{n}})^n-e^{-\frac{x^2}{2}\sigma^{2}_{\lambda,\nu}}|dx\\
I_{2,\lambda,\epsilon}(n) & = & \int_{\epsilon\sqrt{n}\leq |x|\leq \pi \sqrt{n}}|\varphi_{\lambda}(\frac{x}{\sqrt{n}})^n|dx\\
I_{3,\lambda,\epsilon}(n) & = &  \int_{|x|\geq \epsilon \sqrt{n}}e^{-\frac{x^2}{2}\sigma^{2}_{\lambda,\nu}}dx
\end{eqnarray*}
where $\epsilon\in]0,\pi[$. \\
Let us now use that $X_{\lambda,1}$ has the $\CPois(\lambda,\nu)$-distribution: 
\begin{itemize}
 \item $\varphi_{\lambda}(x)=\exp(-i\lambda x+ \lambda(\phi_{\nu}(x)-1))$,
 where $\phi_{\nu}$ denotes the characteristic function of~$\nu$; 
 \item the expectation of $X_{\lambda,1}$ is 
 $m_{\lambda,\nu}=\lambda m_{\nu,1}$ and its variance is
 $\sigma^{2}_{\lambda,\nu}=\lambda(m_{\nu,2}+m_{\nu,1})$.
\end{itemize} 
  Therefore,    
\[
\varphi_{\lambda}(\frac{x}{\sqrt{n}})^n=e^{\psi_{n,\lambda}(x)}e^{-\frac{x^2}{2}\sigma^{2}_{\lambda,\nu}}\;\text{where}\;
\psi_{n,\lambda}(x)=n\lambda(\phi_{\nu}(\frac{x}{\sqrt{n}})-\phi_{\nu}(0)-\phi'_{\nu}(0)\frac{x}{\sqrt{n}}-\phi''_{\nu}(0)\frac{x^2}{2n}).
\] 
The study of  the remainder in  the Taylor expansion of $\phi_\nu$ yields: 
\begin{equation}
\label{taylorremainder}
|\psi_{n,\lambda}(x)|\leq 
\frac{1}{2}\lambda x^2\sup_{u\leq \frac{|x|}{\sqrt{n}}}|\phi''_{\nu}(u)-\phi''_{\nu}(0)|.
\end{equation}
Accordingly, there exists $\epsilon_0>0$ such that for $|x|\leq \epsilon_{0} \sqrt{n}$, $|e^{\psi_{n,\lambda}(x)}-1|\leq e^{\frac{x^2}{4}\sigma^2_{\lambda,\nu}}+1$. 
Let us split $I_{1,\lambda,\epsilon_0}(n) $ into the integral on $[-B,B]$ 
denoted by $J_{1,\lambda,B}(n)$ and the integral on 
$]-\epsilon_0 \sqrt{n},-B[\cup ]B, \epsilon_0 \sqrt{n}[$ denoted by $J_{2,\lambda,B,\epsilon_0}(n)$.
For every $B>0$ and $\lambda\in[a,b]$, $J_{2,\lambda,B,\epsilon_0}(n)$ is bounded by 
\[\int_{B<|x|<\epsilon_0 \sqrt{n}}2\exp\left(-\frac{1}{4}x^2\sigma^{2}_{\lambda,\nu}\right)dx\leq 
\frac{2}{Ba(m_{\nu,1}+m_{\nu,2})}\exp\left(-\frac{B^2}{4}a(m_{\nu,1}+m_{\nu,2})\right).
\]   
Since  $\sup_{\lambda\in[a,b],|x|\leq B}|\psi_{n,\lambda}(x)|$ converges to $0$ as $n$ tends to $+\infty$ (by \eqref{taylorremainder}), 
$J_{1,\lambda,B}(n)$ converges to 0 uniformly for $\lambda\in[a,b]$. 
Therefore, $I_{1,\lambda,\epsilon_0}(n) $ converges to $0$ uniformly for $\lambda\in[a,b]$. 

Let us now consider $I_{2,\lambda,\epsilon_0}(n)$. By assumption, 
$|\varphi_{1}(x)|<1$ for every ${x\in[-\pi,\pi]\setminus\{0\}}$. 
Thus $\beta=\sup_{\epsilon_0\leq |x|\leq \pi}|\varphi_{1}(x)|<1$. 
Since  $\varphi_{\lambda}=(\varphi_{1})^{\lambda}$ for every $\lambda>0$, 
\[I_{2,\lambda,\epsilon_0}(n)\leq 2\pi\sqrt{n}(\beta)^{\lambda n}\;
\text{for every}\;\lambda>0 \et n\in\NN^*.
\]
 This shows that
 $(I_{2,\lambda,\epsilon_0}(n))_n$ converges to 0 uniformly for $\lambda\in[a,b]$.  

Finally, $I_{3,\lambda,\epsilon_0}(n)\leq \frac{2}{\epsilon_0\lambda(m_{\nu,1}+m_{\nu,2})\sqrt{n}}e^{-\frac{n}{2}\epsilon^2_0\lambda(m_{\nu,1}+m_{\nu,2})}$. 
It converges also to 0 uniformly for $\lambda\in[a,b]$, which completes the proof. 
\subsection{Dual BGW process \label{annex:dual}}
A supercritical BGW process conditioned to become extinct 
is a subcritical BGW process: 
\begin{bibli}[\cite{AthreyaNey72}, Theorem 3, p. 52]
 Let $(Z_n)_n$ be a supercritical BGW process with one ancestor. 
 Let $\phi$ denote the generating function of its offspring distribution and let 
 $q$ denote its extinction probability. 
 Assume that $\phi(0)>0$. Then,  $(Z_n)_n$ conditioned to become extinct has the same law
as a subcritical BGW process with one ancestor and offspring generating function
$s\mapsto\frac{1}{q}\phi(qs)$. 
\end{bibli}
As a consequence of this theorem, if the offspring distribution is a compound Poisson distribution, the offspring distribution 
of the dual BGW process is also a compound Poisson distribution:
\begin{lem}
 \label{dualproces}Let $\lambda$ be a positve real  and let $\nu$ be a probability measure on $\NN$ 
 with a finite expectation $m$  and generating function $G_{\nu}$. 
 Let $Z$ be a BGW process with offspring distribution 
$\CPois(\lambda,\nu)$. 
Assume that $\lambda m > 1$ and let $q$ denote the extinction
probability  of $Z$ 
(that is the smallest positive solution of the equation $\exp(\lambda (G_{\nu}(x)-1))=x$). 
Then $Z$ conditioned to become extinct has the same law
as the subcritical BGW process with offspring distribution 
$\CPois(\lambda G_{\nu}(q) ,\hat{\nu}_q)$ where
$\hat{\nu}_q(k)=\frac{q^k}{G_{\nu}(q)}\nu(k)$ for every $k\in\NN$. 
\end{lem}

\begin{rem}
Let us note that if $\nu$ is an heavy-tailed distribution (that is the couvergence radius of $G_{\nu}$
is equal to 1) then it is not the case for  $\hat{\nu}_{q}$. 
\end{rem}
 
More generally, let us write out some properties of a 
\BGW$(u,\lambda G_{\nu}(a) ,\hat{\nu}_a)$ process for any $a\in]0,1[$ (such a process 
appears in Corollary \ref{corolrestrictblocksize} dealing with the restriction 
of the Poisson point process $\mP_n$ 
to tuples in $\mW[\{1,\ldots,\lfloor a n\rfloor\}]$). 
\begin{lem} 
\label{annex:changeofmeasure}
Let $\lambda>0$, let $\nu$ be a probability measure on $\NN$ and let $G_{\nu}$ denote its generating function. 
For $a\in]0,1[$, set $\hat{\nu}_a(k)=\frac{a^k}{G_{\nu}(a)}\nu(k)$ for every $k\in\NN$. 
\begin{enumerate}
\item Let $G_{\lambda,\nu,a}$ denote the generating function of the $\CPois(\lambda,\hat{\nu}_a)$-distribution. Then, 
\[G_{\lambda G_{\nu}(a),\nu,a}(s)= G_{\lambda,\nu,1}(as)\exp(\lambda (1-G_{\nu}(a)) )
\text{ for every }s\text{ in the domain of  }G_{\lambda,\nu,1}.
 \] 
\item \emph{Mass-function distribution}: for every $k\in\NN$, 
\begin{equation}
\label{massfunctCPa}
\CPois(\lambda G_{\nu}(a),\hat{\nu}_a)(\{k\})=a^ke^{\lambda (1-G_{\nu}(a))}\CPois(\lambda,\nu)(\{k\})\quad \forall k\in\NN.
\end{equation}
\item The expectation of the $\CPois(\lambda,\hat{\nu}_a)$-distribution is 
an analytic and  increasing function of $a\in]0,1[$. In particular, the maximal value of 
this function is greater than $1$ on $]0,1[$ if and only if the expectation of the 
$\CPois(\lambda,\nu)$-distribution  is greater than~1. 
\item Let $\hat{T}^{(u)}_a$ denote the total population size of a 
\BGW$(u,\lambda G_{\nu}(a) ,\hat{\nu}_a)$ process. 
For every $k\in \NN^*$ greater than or equal to $u$,
\begin{equation}
\label{totalpopidentity}  
\Pd(\hat{T}^{(u)}_a=k)=a^{k-u}e^{k\lambda(1-G_{\nu}(a))}\Pd(\hat{T}^{(u)}_1=k).
\end{equation}
\end{enumerate}
\end{lem}
\begin{proof}
Equality \eqref{totalpopidentity} can be established by applying formulae \eqref{TGWdistrib} and \eqref{massfunctCPa}. 
\end{proof}
\begin{ex}
Set $a\in]0,1]$.  
Let us present the \BGW$(1,tm_{1,\hat{p}_{a}},\hat{p}_a)$ process used to approximated  the distribution of $|\Pi^{(1)}_{\enu{\lfloor an\rfloor}}(ant)|$ for two examples of distributions $p$.
\begin{itemize}
\item  If $p$ is the Dirac mass on $d\in\NN\setminus\{0,1\}$, 
the offspring distribution of the BGW process is $(d-1)\Pois(\frac{td}{a})$ and the total population size distribution  is: $$\sum_{k\in 1+(d-1)\NN}e^{-tk}\frac{(tk)^{\frac{k-1}{d-1}}}{k(\frac{k-1}{d-1})!}\delta_k.$$
\item If $p$ is the logarithmic$(u)$ distribution for $u\in]0,1[$, then $\hat{p}_a$ is the geometric distribution on $\NN^*$ with parameter $1-au$: $\hat{p}_a=\sum_{k=1}^{+\infty}(1-au)(au)^{k-1}\delta_{k}$.  The offspring distribution of the BGW process
is the geometric Poisson distribution \CPois$(tc(a,u), \hat{p}_a)$ where $c(a,u)=\frac{(au)^2}{-(1-au)\log(1-au)}$. This distribution is also known as P\'{o}lya-Aeppli $(tc(a,u),au)$
 distribution and is defined by:

\[e^{-tc(a,u) }\delta_{0}+\sum_{k=1}^{+\infty}\left(e^{-tc(a,u)} (1-au)^k\sum _{{j=1}}^{k}\frac{1}{j!}\binom {k-1}{j-1} \left(\frac{tc(a,u)au}{1-au}\right)^{j}\right)\delta_{k}.
\]
The total population size distribution  of the BGW process with $i$ ancestors has the following distribution: 
 \[e^{-tc(a,u) }\delta_{i}+\sum_{k=i+1}^{+\infty}\left(\frac{i}{k}e^{-tc(a,u) }(au)^{k-i}\sum _{j=1}^{k-i} \frac{k^j}{j!}{\binom {k-i-1}{j-1}}\left(tc(a,u)(\frac{1}{au}-1)\right)^{j}\right)\delta_{k}.
\]
\end{itemize}
\end{ex}
\bibliographystyle{plain}

\end{document}